\documentclass{amsart}
\usepackage{amssymb,amscd,xypic,amsmath}

\newtheorem{lemma}{Lemma}[section]
\newtheorem{theorem}[lemma]{Theorem}
\newtheorem{proposition}[lemma]{Proposition}
\newtheorem{corollary}[lemma]{Corollary}
\theoremstyle{definition}

\newtheorem{problem}[lemma]{Problem}
\newtheorem{remark}[lemma]{Remark}
\newtheorem{claim}[lemma]{Claim}

\newtheorem{example}[lemma]{Example}
\newtheorem{definition}[lemma]{Definition}

\newcommand{\pr}{\operatorname{pr}}
\newcommand{\odd}{\mathrm{odd}}
\newcommand{\even}{\mathrm{even}}

\newcommand{\F}{\mathcal {F}}
\newcommand{\E}{\mathcal E}
\newcommand{\Ra}{\Rightarrow}
\newcommand{\U}{\mathcal U}
\newcommand{\V}{\mathcal V}

\newcommand{\A}{\mathcal A}
\newcommand{\B}{\mathcal B}

\newcommand{\K}{\mathcal K}
\newcommand{\N}{\mathcal N}
\newcommand{\M}{\mathcal M}
\newcommand{\I}{\mathcal I}

\newcommand{\w}{\omega}

\newcommand{\IN}{\mathbb N}
\newcommand{\IZ}{\mathbb Z}
\newcommand{\IC}{\mathbb C}
\newcommand{\IH}{\mathbb H}
\newcommand{\HH}{\mathsf H}
\newcommand{\IT}{\mathbb T}
\newcommand{\IK}{\mathsf K}
\newcommand{\C}{\mathcal C}

\newcommand{\inv[1]}{\overset{\leftrightarrow}{#1}}

\newcommand{\LL}{\mathcal L}

\newcommand{\End}{\mathrm{End}}
\newcommand{\Enl}{\mathrm{End}_\lambda}

\newcommand{\Aut}{\mathrm{Aut}}
\newcommand{\XX}{\mathcal X}
\newcommand{\YY}{\mathcal Y}
\newcommand{\ZZ}{\mathcal Z}

\newcommand{\Fix}{\mathrm{Fix}}
\newcommand{\IFix}{\I\mbox{-}\Fix}
\newcommand{\Iso}{\mathrm{Iso}}
\newcommand{\Tau}{\mathsf T}
\newcommand{\TI}{{\I\!\!\I}}
\newcommand{\id}{\mathrm{id}}

\newcommand{\EE}{\mathcal E}
\newcommand{\wtd}{\widetilde}

\newcommand{\pT}{\mathsf{pT}}
\newcommand{\Odd}{\mathrm{Odd}}

\newcommand{\Stab}{\mathrm{Stab}}
\newcommand{\wht}{\widehat}

\def\L{\mathcal L}

\newcommand{\la}{\langle}
\newcommand{\ra}{\rangle}

\begin{document}

\keywords{Compact right-topological semigroup, superextension of a group, semigroup of maximal linked systems, faithful representation, minimal ideal, minimal left ideal, minimal idempotent, wreath product, twinic group, twin set.}
\subjclass{Primary 20M30; 20M12; 22A15; 22A25; 54D35.}

\title{Algebra in Superextensions of Twinic Groups}

\author{Taras Banakh}
\address{Department of Mathematics\\ 
Ivan Franko National University of Lviv\\
Universytetska 1, 79000, Ukraine\\
and\\
Instytut Matematyki\\ 
Uniwersytet Humanistyczno-Przyrodniczy w Kielcach,\\
\'Swi\c etokrzyska 15, Kielce, Poland\\
E-mail: tbanakh@yahoo.com; t.o.banakh@gmail.com}

\author{Volodymyr Gavrylkiv}
\address{Faculty of Mathematics and Computer Sciences,\\
Vasyl Stefanyk Precarpathian National University\\
Shevchenko str, 57, Ivano-Frankivsk, 76025, Ukraine\\
E-mail: vgavrylkiv@yahoo.com}

\begin{abstract}
Given a group $X$ we study the algebraic structure of the compact right-topological semigroup $\lambda(X)$ consisting of maximal linked systems on $X$.
This semigroup contains the semigroup $\beta(X)$ of ultrafilters as a closed subsemigroup. We construct a faithful representation of the semigroup $\lambda(X)$ in the semigroup $\mathsf P(X)^{\mathsf P(X)}$ of all self-maps of the power-set $\mathsf P(X)$ and show that the image of $\lambda(X)$ in $\mathsf P(X)^{\mathsf P(X)}$ coincides with the semigroup $\Enl(\mathsf P(X))$ of all functions $f:\mathsf P(X)\to\mathsf P(X)$ that are equivariant, monotone and symmetric in the sense that $f(X\setminus A)=X\setminus f(A)$ for all $A\subset X$. Using this representation we
describe the minimal ideal $\IK(\lambda(X))$ and minimal left ideals of the superextension $\lambda(X)$ of a twinic group $X$. A group $X$ is called {\em twinic} if it admits a left-invariant ideal $\I\subset\mathsf P(X)$ such that $xA=_\I yA$ for any subset $A\subset X$ and points $x,y\in X$ with $xA\subset_\I X\setminus A\subset_\I yA$. The class of twinic groups includes all amenable groups and all groups with periodic commutators but does not include the free group $F_2$ with two generators.

We prove that for any twinic group $X$, there is a cardinal $m$ such that all minimal left ideals of $\lambda(X)$ are algebraically isomorphic to $$2^m\times \prod_{1\le k\le\infty}C_{2^k}^{\;q(X,C_{2^k})}\times \prod_{3\le k\le\infty}Q_{2^k}^{\;q(X,C_{2^k})}$$ for some cardinals $q(X,C_{2^k})$ and $q(X,Q_{2^k})$, $k\in\IN\cup\{\infty\}$. Here $C_{2^k}$ is the cyclic group of order $2^k$, $C_{2^\infty}$ is the quasicyclic 2-group and $Q_{2^k}$, $k\in\IN\cup\{\infty\}$, are the groups of generalized quaternions. 

If the group $X$ is abelian, then $q(X,Q_{2^k})=0$ for all $k$ and  $q(X,C_{2^k})$ is the number of subgroups $H\subset X$ with quotient $X/H$ homeomorphic to $C_{2^k}$. If $X$ is an Abelian group (admitting no epimorphism onto $C_{2^\infty}$) then each minimal left ideal of the superextension $\lambda(X)$ is algebraically (and topologically) isomorphic to the product $\prod_{1\le k\le\infty} (C_{2^k}\times 2^{2^{k-1}-k})^{q(X,C_{2^k})}$ where the cube $2^{2^{k-1}-k}$ (equal to $2^\w$ if $k=\infty$) is endowed with the  left-zero multiplication. For an abelian group $X$ all minimal left ideals of $\lambda(X)$ are metrizable if and only if $X$ has finite ranks $r_0(X)$ and $r_2(X)$ and admits no homomorphism onto the group $C_{2^\infty}\oplus C_{2^\infty}$. 

Applying this result to the group $\IZ$ of integers, we prove that each minimal left ideal of $\lambda(\IZ)$ is topologically isomorphic to $2^\w\times\prod_{k=1}^\infty C_{2^k}$. Consequently, all  subgroups in the minimal ideal $\IK(\lambda(\IZ))$ of $\lambda(\IZ)$ are profinite abelian groups. On the other hand, the superextension $\lambda(\IZ)$ contains an isomorphic topological copy of each second countable profinite topological semigroup. This results contrasts with the famous Zelenyuk's Theorem saying that the semigroup $\beta(\IZ)$ contains no finite subgroups. At the end of the paper we describe the structure of minimal left ideals of finite groups $X$ of order $|X|\le 15$.
\end{abstract}
\maketitle
\newpage

\tableofcontents
\newpage

\section{Introduction}

After discovering a topological proof of
Hindman's theorem \cite{Hind} (see \cite[p.102]{HS}, \cite{H2}), topological methods become a standard
tool in the modern combinatorics of numbers, see \cite{HS},
\cite{P}. The crucial point is that any semigroup operation 
defined on a discrete space $X$ can be extended to a
right-topological semigroup operation on $\beta(X)$, the Stone-\v
Cech compactification of $X$. The extension of the operation from
$X$ to $\beta(X)$ can be defined by the simple formula:
\begin{equation}\label{extension}
\A\circ\B=\big\{A\subset X:\{x\in X:x^{-1}A\in\B\}\in\A\big\},
\end{equation}

The Stone-\v Cech compactification $\beta(X)$ of $X$ is a
subspace of the double power-set $\mathsf P^2(X)=\mathsf P(\mathsf P(X))$, which can be identified with the Cantor discontinuum $\{0,1\}^{\mathsf P(X)}$ and endowed with the compact Hausdorff topology of the Tychonoff product.
It turns out that the formula (\ref{extension}) applied to arbitrary families $\A,\mathcal B\in\mathsf P^2(X)$ of subsets of a group $X$ still defines a binary 
operation $\circ:\mathsf P^2(X)\times \mathsf P^2(X)\to\mathsf P^2(X)$ that turns the double power-set $\mathsf P^2(X)$ into a compact Hausdorff right-topological semigroup that contains $\beta(X)$ as a closed subsemigroup.  

The semigroup $\beta(X)$ lies in a bit larger subsemigroup $\lambda(X)\subset\mathsf P^2(X)$ consisting of all maximal linked systems on $X$. We recall that a family $\mathcal L$ of subsets of $X$ is 
\begin{itemize}
\item {\em linked} if any sets $A,B\in\mathcal L$ have non-empty intersection $A\cap B\ne\emptyset$;
\item {\em maximal linked} if $\mathcal L$ coincides with each linked system $\mathcal L'$ on $X$ that contains $\mathcal L$.
\end{itemize}
 The space $\lambda(X)$ is
  well-known in  General and Categorial Topology as the {\em
 superextension}
  of $X$, see \cite{vM}, \cite{TZ}. 

The thorough study of algebraic properties of the superextensions
of groups was started in \cite{BGN} and
continued in \cite{BG2} and \cite{BG3}. In particular, in \cite{BG3} we proved that the minimal left ideals of the superextension $\lambda(\IZ)$   are metrizable topological semigroups. In this paper we shall extend this result to the superextensions $\lambda(X)$ of all finitely-generated abelian groups $X$.

The results obtained in this paper completely reveal the topological and algebraic structure of the minimal ideal and minimal left ideals of the superextension $\lambda(X)$ of a twinic group $X$. A group $X$ is defined to be {\em twinic } if it admits a left-invariant ideal $\I$ of subsets of $X$ such that for any subset $A\subset X$ with $xA\subset_\I X\setminus A\subset_\I yA$ for some $x,y\in X$ we have $xA=_\I yA$. Here the symbol $A\subset_\I B$ means that $A\setminus B\in\I$ and $A=_\I B$ means that $A\subset_\I B$ and $B\subset_\I A$. In Section~\ref{s:tg} we shall prove that the class of twinic groups contains all amenable groups and all groups with periodic commutators (in particular, all torsion groups), but does not contain the free group with two generators $F_2$.    

We need to recall the notation for some standard 2-groups. By $Q_8$ we denote the group of quaternions. It is a multiplicative subgroup $\{1,i,j,k,-1,-i,-j,-k\}$ of the algebra of quaternions $\IH$ (which contains the field of complex numbers $\IC$ as a subalgebra).

For every $k\in\w$ let $C_{2^k}=\{z\in\IC:z^{2^k}=1\}$ be the cyclic group of order $2^k$. The multiplicative subgroup $Q_{2^k}\subset\IH$ generated by the union $C_{2^{k-1}}\cup Q_8$ is called the {\em group of generalized quaternions}. The union $C_{2^\infty}=\bigcup_{k=1}^\infty C_{2^k}$ is called the {\em quasicyclic 2-group} and the union $Q_{2^\infty}=\bigcup_{k=3}^\infty Q_{2^k}$ is called {\em the infinite group of generalized quaternions}. By Theorem~\ref{BCQ}, a group $G$ is isomorphic to $C_{2^n}$ or $Q_{2^n}$ for some $n\in\IN\cup\{\infty\}$ if and only if $G$ is a 2-group with a unique 2-element subgroup.

The following theorem describing the structure of minimal left ideals of the superextesions of twinic groups can be derived from Theorem~\ref{t18.11} and Proposition~\ref{p19.1}:

\begin{theorem} For each twinic group $X$ there are cardinals $q(X,C_{2^k})$, $q(X,Q_{2^k})$, $k\in\IN\cup\{\infty\}$, such that
\begin{enumerate}
\item[\textup{(1)}] each minimal left ideal of $\lambda(X)$ is algebraically isomorphic to 
$$Z\times\prod_{1\le k\le \infty}C_{2^k}^{\;q(X,C_{2^k})}\times\prod_{3\le k\le\infty}Q_{2^k}^{\;q(X,Q_{2^k})}$$for some semigroup $Z$ of left zeros;
\item[\textup{(2)}] each maximal subgroup of the minimal ideal of $\lambda(X)$ is algebraically isomorphic to
$$\prod_{1\le k\le \infty}C_{2^k}^{\;q(X,C_{2^k})}\times\prod_{3\le k\le\infty}Q_{2^k}^{\;q(X,Q_{2^k})}.$$
\item[\textup{(3)}] If $q(X,C_{2^\infty})=q(X,Q_{2^\infty})=0$, then 
each maximal subgroup of the minimal ideal of $\lambda(X)$ is topologically isomorphic to the compact topological group
$$\prod_{1\le k<\infty}C_{2^k}^{\;q(X,C_{2^k})}\times\prod_{3\le k<\infty}Q_{2^k}^{\;q(X,Q_{2^k})}.$$
\end{enumerate}
If the group $X$ is abelian, then 
\begin{enumerate}
\item[\textup{(4)}]    $q(X,Q_{2^k})=0$ for every $k\in\IN\cup\{\infty\}$ while $q(X,C_{2^k})$ is equal to the number of subgroups $H\subset X$ such that the quotient group $X/H$ is isomorphic to $C_{2^k}$;
\item[\textup{(5)}] for every $k\in\IN$ $$q(X,C_{2^k})=\frac{|\hom(X,C_{2^k})|-|\hom(X,C_{2^{k-1}})|}{2^{k-1}},$$where $\hom(X,C_{2^k})$ is the group of homomorphisms from $X$ into $C_{2^k}$.
\end{enumerate} 
\end{theorem}

\smallskip

\section{Right-topological semigroups}

In this section we recall some information from \cite{HS} related
to right-topological semigroups. By definition, a
right-topological semigroup is a topological space $S$ endowed
with a semigroup operation $\ast:S\times S\to S$ such that for
every $a\in S$ the right shift $r_a:S\to S$, $r_a:x\mapsto x\ast
a$, is continuous. If the semigroup operation $\ast:S\times S\to
S$ is (separately) continuous, then $(S,\ast)$ is a ({\em semi}-){\em topological
semigroup}. A typical example of a right-topological semigroup is the semigroup $X^X$ of all self-maps of a topological space $X$ endowed with the Tychonoff product topology and the binary operation of composition of functions.

From now on, $S$ is a compact Hausdorff right-topological
semigroup. We shall recall some known information concerning
ideals in $S$, see \cite{HS}.

A non-empty subset $I$ of $S$ is called a {\em left} (resp. {\em right})
{\em ideal\/} if $SI\subset I$ (resp. $IS\subset I$). If $I$ is both
a left and right ideal in $S$, then $I$ is called an {\em ideal}
in $S$. Observe that for every $x\in S$ the set $SxS=\{sxt:s,t\in S\}$ (resp. $Sx=\{sx:s\in S\}$, $xS=\{xs:s\in S\}$) is an ideal (resp. left ideal, right ideal) in $S$.
Such an ideal is called {\em principal}. An ideal $I\subset S$ is
called {\em minimal} if any ideal of $S$ that lies in $I$
coincides with $I$. By analogy we define minimal left and right
ideals of $S$. It is easy to see that each minimal left (resp.
right) ideal $I$ is principal. Moreover, $I=Sx$ (resp. $I=xS$) for
each $x\in I$. This simple observation implies that each minimal
left ideal in $S$, being principal, is closed in $S$. By
\cite[2.6]{HS}, each left ideal in $S$ contains a minimal left ideal.
 The union $\IK(S)$ of all minimal left ideals of $S$ coincides with the minimal ideal of $S$, \cite[2.8]{HS}. 

All minimal left ideals of $S$ are mutually homeomorphic and all
maximal groups of the minimal ideal $\IK(S)$ are algebraically
isomorphic. Moreover, if two maximal groups lie in the same
minimal right ideal, then they are topologically isomorphic.

We shall need the following known fact, see Theorem 2.11(c) \cite{HS}.

\begin{proposition}\label{p2.1} For any two minimal left ideals $A,B$ of a compact right-topological semigroup $S$ and any point $b\in B$ the right shift $r_b:A\to B$, $r_b:x\mapsto xb$, is a homeomorphism.
\end{proposition}

This proposition implies the following corollary, see \cite[Lemma 1.1]{BG3}.

\begin{corollary}\label{c2.2} If a homomorphism $h:S\to S'$ between two compact right-topological semigroups is injective on some minimal left ideal of $S$, then $h$ is injective on each minimal left ideal of $S$.
\end{corollary}

An element $z$ of a semigroup $S$ is called a {\em right zero}
(resp. a {\em left zero}) in $S$ if $xz=z$ (resp. $zx=z$) for all
$x\in S$. It is clear that $z\in S$ is a right (left) zero in $S$
if and only if the singleton $\{z\}$ is a left (right) ideal in
$S$.

An element $e\in S$ is called an {\em idempotent} if $ee=e$. By
 Ellis's Theorem \cite[2.5]{HS}, the set $E(S)$ of idempotents of any  compact right-topological
semigroup is not empty.
For every idempotent $e$ the set
$$\mathsf H_e=\{x\in S:\exists x^{-1}\in S\;\;(xx^{-1}x=x,\;x^{-1}xx^{-1}=x^{-1},\;xx^{-1}=e=x^{-1}x)\}$$
is the largest subgroup of $S$ containing $e$.

By \cite[1.48]{HS}, for an idempotent
$e\in E(S)$ the following conditions are equivalent:
\begin{itemize}
\item $e\in \IK(S)$;
\item $\IK(S)=SeS$;
\item $Se$ is a minimal left ideal in $S$;
\item $eS$ is a minimal right ideal in $S$;
\item $eSe$ is a subgroup of $S$.
\end{itemize}
An idempotent $e$ satisfying the above equivalent conditions will be called a {\em minimal idempotent} in $S$. By \cite[1.64]{HS}, for any minimal idempotent $e\in S$ the set $\mathsf E(Se)=\mathsf E(S)\cap Se$ of idempotents of the minimal left ideal $Se$ is a semigroup of left zeros, which means that $xy=x$ for all $x,y\in \mathsf E(Se)$. By the Rees-Suschkewitsch Structure Theorem (see \cite[1.64]{HS}) the map
$$\varphi:\mathsf E(Se)\times \mathsf H_e\to Se,\; \varphi:(x,y)\mapsto xy,$$ is an algebraic isomorphism of the corresponding semigroups. If the minimal left ideal $Se$ is a topological semigroup, then $\varphi$ is a topological isomorphism.

Now we see that all the information on the algebraic (and sometimes topological) structure of the minimal left ideal $Se$ is encoded in the properties of the left zero  semigroup $\mathsf E(Se)$ and the maximal group $\mathsf H_e$.

\section{Acts and their endomorphism monoids}

In this section we survey the information on acts that will be widely used in this paper for describing the algebraic structure of minimal left ideals of the superextensions of groups.

Following the terminology of \cite{KKM} by an {\em act} we understand a set $X$ endowed with a left
action $\cdot :H\times X\to X$ of a group $H$ called the {\em structure group} of the act. The action should satisfy two axioms: $1x=x$ and $g(hx)=(gh)x$ for all $x\in X$ and $g,h\in H$. Acts with the structure group $H$ will be called {\em $H$-acts} or {\em $H$-spaces}.

An act $X$ is called {\em free} if the stabilizer $\Fix(x)=\{h\in H:hx=x\}$ of each point $x\in X$ is trivial.
For a point $x\in X$ by $[x]=\{hx:h\in H\}$ we denote its {\em orbit} and by $[X]=\{[x]:x\in X\}$ the orbit space of the act $X$. More generally, for each subset $A\subset X$ we put $[A]=\{[a]:a\in A\}$.

A function $f:X\to Y$ between two $H$-acts is called {\em equivariant} if $f(hx)=hf(x)$ for all $x\in X$ and $h\in H$.
A  function $f:X\to Y$ is called an {\em isomorphism} of the $H$-acts $X$ and $Y$ if it is bijective and equivariant. An equivariant self-map $f:X\to X$ is called an {\em endomorphism} of the $H$-act $X$. If $f$ is bijective, then $f$ is an {\em automorphism} of $X$.

The set $\End(X)$ of endomorphisms of an $H$-act $X$, endowed with the operation of composition of functions, is a monoid called the {\em endomorphism monoid} of  $X$. 

Each free $H$-act $X$ is isomorphic to the product $H\times [X]$ endowed with the action $h\cdot (x,y)=(hx,y)$. For such an act the semigroup $\End(X)$ is isomorphic to the wreath product $H\wr [X]^{[X]}$ of the group $H$ and the semigroup $[X]^{[X]}$ of all self-maps of the orbit space $[X]$.

The wreath product $H\wr A^A$ of a group $H$ and the semigroup $A^A$ of self-maps of a set $A$ is defined as the semidirect product $H^A\rtimes A^A$ of the $A$-th power of $H$ with $A^A$, endowed with the semigroup operation $(h,f)*(h',f')=(h'',f'')$ where $f''=f\circ f'$ and $h''(\alpha)=h(f'(\alpha))\cdot h'(\alpha)$ for $\alpha\in A$. For any subsemigroup $S\subset A^A$ the subset $H\wr S=\{(h,f)\in H^A\rtimes A^A:f\in S\}$ is called the {\em wreath product} of $H$ and $S$. If both $H$ and $S$ are groups, then their wreath product $H\wr S$ is a group.

Observe that the maximal subgroup of $A^A$ containing the identity self-map of $A$ coincides with the group $S_A$ of all bijective functions $f:A\to A$.

\begin{theorem}\label{t2.1} Let $H$ be a group and $X$ be a free $H$-act. Then
\begin{enumerate}
\item[\textup{(1)}]  the semigroup $\End(X)$ is isomorphic to the wreath product $H\wr [X]^{[X]}$;
\item[\textup{(2)}] the minimal ideal $\IK(\End(X))$ of \/  $\End(X)$ coincides with the set $\{f\in\End(X):\forall x\in f(X)\; f(X)\subset[x]\}$;
\item[\textup{(3)}] each minimal left ideal of $\End(X)$ is isomorphic to $H\times [X]$ where $[X]$ is endowed with the left zero multiplication;
\item[\textup{(4)}] for each idempotent $f\in\End(X)$ the maximal subgroup $\mathsf H_f\subset\End(X)$ is isomorphic to $H\wr S_{[f(X)]}$;
\item[\textup{(5)}] for each minimal idempotent $f\in \IK(\End(X))$ the maximal group $\mathsf H_f=f\cdot \End(X)\cdot f$ is isomorphic to  $H$.
\end{enumerate}
\end{theorem}

\begin{proof} 1. Let $\pi:X\to[X]$, $\pi:x\mapsto[x]$, denote the orbit map and   $s:[X]\to X$ be a section of $\pi$, which means that $\pi\circ s([x])=[x]$ for all $[x]\in[X]$.

Observe that each equivariant map $f:X\to X$ induces a well-defined map $[f]:[X]\to [X]$, $[f]:[x]\mapsto [f(x)]$, of the orbit spaces. Since the action of $H$ on $X$ is free, for every orbit $[x]\in[X]$ we can find a unique point $f_H([x])\in H$ such that $f\circ s([x])=(f_H([x]))^{-1}\cdot s([f(x)])$.

We claim that the map
$$\Psi:\End(X)\to H\wr [X]^{[X]},\;\; \Psi:f\mapsto (f_H,[f]),$$
is a semigroup isomorphism.

First we check that the map $\Psi$ is a homomorphism. Pick any two equivariant functions $f,g\in\End(X)$ and consider their images $\Psi(f)=(f_H,[f])$ and $\Psi(g)=(g_H,[g])$ in $H\wr[X]^{[X]}$. Consider also the composition $f\circ g$ and its
image $\Psi(f\circ g)=((f\circ g)_H,[f\circ g])$. We claim that
$$((f\circ g)_H,[f\circ g])=(f_H,[f])*(g_H,[g])=((f_H\circ[g])\cdot g_H,[f]\circ[g]).$$
The equality $[f\circ g]=[f]\circ [g]$ is clear. To prove that $(f\circ g)_H=(f_H\circ [g])\cdot g_H$, take any orbit $[x]\in [X]$. It follows from the definition of $(f\circ g)_H([x])$ that
$$
\begin{aligned}
&((f\circ g)_H([x]))^{-1}\cdot s([f\circ g(x)])=(f\circ g)\circ s([x])=f(g\circ s([x]))=\\
&f\big((g_H([x]))^{-1}\cdot s([g(x)])\big)=(g_H([x]))^{-1}\cdot f\circ s([g(x)])=\\
&(g_H([x]))^{-1}\cdot (f_H([g(x)]))^{-1}\cdot s([f\circ g(x)])=\\
&(f_H\circ [g]([x])\cdot g_H([x]))^{-1}\cdot s([f\circ g(x)])
\end{aligned}
$$which implies the desired equality $(f\circ g)_H=(f_H\circ [g])\cdot g_H$.
\smallskip

Next, we show that the homomorphism $\Psi$ is injective. Given two  equivariant functions $f,g\in\End(X)$ with $(f_H,[f])=\Psi(f)=\Psi(g)=(g_H,[g])$, we need to show  that $f=g$. Observe that for every orbit $[x]\in[X]$ we get
$$f(s([x]))=(f_H([x]))^{-1}\cdot s\circ [f]([x]))=(g_H([x]))^{-1}\cdot s\circ[g]([x])=g(s([x])).$$
Now for each $x\in X$ we can find a unique $h\in H$ with $x=h\cdot s([x])$ and apply the equivariantness of the functions $f,g$ to conclude that
 $$f(x)=f(h\cdot s([x]))=h\cdot f(s([x]))=h\cdot g(s([x]))=g(h\cdot s([x]))=g(x).$$

Finally, we show that $\Psi$ is surjective. Given any pair $(h,g)\in H\wr [X]^{[X]}=H^{[X]}\times [X]^{[X]}$, we define an equivariant function $f\in\End(X)$ with $(h,g)=(f_H,[f])$ as follows. Given any $x\in X$ find a unique $y\in H$ with $x=y\cdot s([x])$ and let
$$f(x)=y\cdot h([x])^{-1}\cdot s(g([x])).$$
This formula determines a well-defined equivariant function $f:X\to X$ with $\Psi(f)=(h,g)$. Therefore, $\Psi:\End(X)\to H\wr[X]^{[X]}$ is a semigroup isomorphism.
\smallskip

2. Observe that the set $\mathcal I=\{f\in\End(X):\{[f(x)]:x\in X\}$ is a singleton$\}$ is a (non-empty) ideal in $\End(X)$. To show that $\mathcal I$ is the minimal ideal of the semigroup $\End(X)$, we need to check that $\mathcal I$ lies in any ideal $\mathcal J\subset \End(X)$. Take any functions $f\in\mathcal I$ and $g\in \mathcal J$. Find an orbit $[x]\in [X]$ such that $[f(z)]=[x]$ for all $z\in X$. Since the restriction $g|[x]:[x]\to[g(x)]$ is bijective and equivariant, so is its inverse $(g|[x])^{-1}:[g(x)]\to[x]$. Extend this equivariant map to any equivariant map $h:X\to X$. Then
$$f=h\circ g\circ f\in \End(X)\circ g\circ \End(X)\subset\mathcal J.$$

3. Take any idempotent $f\in \IK(\End(X))$ and consider the minimal left ideal $\End(X)\cdot f$. Fix any point $z\in f(X)$ and observe that $f([x])=[z]$ for all $x\in X$ according to the preceding item. It follows that the set $Z=f^{-1}(z)$ meets each orbit $[x]$, $x\in X$, at a single point. So, we can define a unique section $s:[X]\to Z\subset X$ of the orbit map $X\to [X]$ such that $f\circ s([X])=\{z\}$.

To each equivariant map $g\in\End(X)$ assign a unique element $g_H\in H$ such that $g(x)=g_H^{-1}\cdot s([g(x)])$.
It is easy to check that the map
$$\Phi:\End(X)\cdot f\to H\times [X],\;\Phi:g\mapsto (g_H,[g]([x])),$$
is a semigroup homomorphism where the orbit space $[X]$ is endowed with the left zero multiplication.
\smallskip

4. Take any idempotent $f\in \End(X)$ and consider the surjective semigroup homomorphism
$\pr:\End(X)\to[X]^{[X]}$, $\pr:g\mapsto[g]$. It follows that $[f]$ is an idempotent of the semigroup $[X]^{[X]}$ and the image $\pr(\HH_f)$ of the maximal group $\mathsf H_f$ is a subgroup of $[X]^{[X]}$. It is easy to see that the maximal subgroup $\mathsf H_{[f]}$ of the idempotent $[f]$ in $[X]^{[X]}$ coincides with $S_{[f(X)]}\cdot [f]$. The preimage $\pr^{-1}(\mathsf H_{[f]})$ of the maximal subgroup $\mathsf H_{[f]}=S_{[f(X)]}\cdot f$ is isomorphic to the wreath product $H\wr \mathsf H_{[f]}$ and hence is a group. Now the maximality of $\mathsf H_{f}$ guarantees that $\mathsf H_{f}=\pr^{-1}(\mathsf H_{[f]})$ and hence $\mathsf H_f$ is isomorphic to $H\wr S_{[f(X)]}$.
\smallskip

5. If $f\in \IK(\End(X))$ is a minimal idempotent, then the set $[f(X)]=\{[f(x)]:x\in X\}$ is a singleton by the second item. By the preceding item the maximal group $\mathsf H_f$ is isomorphic to $H\wr S_{[f(X)]}$, which is isomorphic to the group $H$ since $[f(X)]$ is a singleton.
\end{proof}

For each group $X$ the power-set $\mathsf P(X)$ will be considered as an $X$-act endowed with the left action $$\cdot:X\times \mathsf P(X)\to\mathsf P(X),\;\;\cdot:(x,A)\mapsto xA=\{xa:a\in A\},$$of the group $X$.
This $X$-act $\mathsf P(X)$ and its endomorphism monoid $\End(\mathsf P(X))$ will  play a crucial role in our considerations.

\section{The function representation of the semigroup $\mathsf P^2(X)$}

In this section given a group $X$ we construct a topological isomorphism 
$$\Phi:\mathsf P^2(X)\to\End(\mathsf P(X))$$
called the {\em function representation} of the semigroup $\mathsf P^2(X)$ in the endomorphism monoid of the $X$-act $\mathsf P(X)$.
We recall that the double power-set $\mathsf P^2(X)=\mathsf P(\mathsf P(X))$ of the group $X$ is endowed with the binary operation
$$\A\circ\B=\big\{A\subset X:\{x\in X:x^{-1}A\in\B\}\in\A\big\}.$$

The isomorphism $\Phi$ assigns to each family $\A$ of subsets of $X$ the function 
$$\Phi_\A:\mathsf P(X)\to\mathsf P(X),\;\;\Phi_\A:A\mapsto\{x\in X:x^{-1}A\in\A\},$$called the {\em function representation} of  $\A$.

In the following theorem by $e$ we denote the neutral element of the group $X$.

\begin{theorem}\label{t4.1} For any group $X$ the map $\Phi:\mathsf P^2(X)\to\End(\mathsf P(X))$
is a topological isomorphism with inverse $\Phi^{-1}:\varphi\mapsto \{A\subset X:e\in\varphi(A)\}$.
\end{theorem}

\begin{proof} First observe that for any family $\A\in\mathsf P^2(X)$ the function $\Phi_\A$ is equivariant, because
$$\Phi_{\A}(xA)=\{y\in X:y^{-1}xA\in{\A}\}=\{xz\in X:z^{-1}A \in{\A}\}=x\,\Phi_{\A}(A)$$for any $x\in X$ and $A\subset X$. Thus the map $\Phi:\mathsf P^2(X)\to\End(\mathsf P(X))$ is well-defined.

To prove that $\Phi$ is a semigroup homomorphism, take two families $\XX,\YY\in \mathsf P^2(X)$ and let
$\ZZ=\XX\circ\YY$. We need to check that
$\Phi_\ZZ(A)=\Phi_\XX\circ\Phi_\YY(A)$ for every
$A\subset X$. Observe that
$$
\begin{aligned}
\Phi_\ZZ(A)&=\{z\in X:z^{-1}A\in\ZZ\}=\{z\in X:\{x\in
X:x^{-1}z^{-1}A\in\YY\}\in\XX\}=\\
&=\{z\in X:\Phi_{\YY}(z^{-1}A)\in \XX\}=\{z\in
X:z^{-1}\Phi_\YY(A)\in\XX\}=\\
&=\Phi_\XX(\Phi_\YY(A))=\Phi_\XX\circ\Phi_\YY(A).
\end{aligned}
$$
\smallskip

To see that the map $\Phi$ is injective, take any two distinct families $\A,\B\in \mathsf P^2(X)$. Without loss of generality,
$\A\setminus \B$ contains some set $A\subset X$. It follows that
$e\in \Phi_\A(A)$ but $e\notin\Phi_\B(A)$ and hence
$\Phi_\A\ne\Phi_\B$.
\smallskip

To see that the map $\Phi$ is surjective, take any equivariant function $\varphi:\mathsf P(X)\to\mathsf P(X)$ and consider the family $\A=\{A\subset X:e\in\varphi(A)\}$.
It follows that for every $A\in\mathsf P(X)$
$$
\begin{aligned}
\Phi_{\A}(A)&=\{x\in X\colon x^{-1}A\in\A\}=\{x\in X\colon e\in \varphi(x^{-1}A)\}=\\
&=\{x\in X\colon e\in x^{-1}\varphi(A)\}=\{x\in X\colon x\in \varphi(A)\}=\varphi(A).
\end{aligned}$$

To prove that $\Phi:\mathsf P^2(X)\to\End(\mathsf P(X))\subset \mathsf P(X)^{\mathsf P(X)}$ is
continuous we first define a convenient subbase of the topology
on the spaces $\mathsf P(X)$ and $\mathsf P(X)^{\mathsf P(X)}$.
The product topology of $\mathsf P(X)$ is generated by the
subbase consisting of the sets
$$x^+=\{A\subset X:x\in A\}\mbox{ and }x^-=\{A\subset X:x\notin A\}$$ where $x\in X$. On the other hand, the product topology on $\mathsf P(X)^{\mathsf P(X)}$ is generated by the subbase consisting of the sets
$$
\langle x,A\rangle^+=\{f\in\mathsf P(X)^{\mathsf P(X)}:x\in f(A)\}\mbox{ and }
\langle x,A\rangle^-=\{f\in\mathsf P(X)^{\mathsf P(X)}:x\notin f(A)\}
$$where $A\in\mathsf P(X)$ and $x\in X$.

Now observe that the preimage
$$\Phi^{-1}(\langle x,A\rangle^+)=\{\A\in \mathsf P^2(X):x\in \Phi_\A(A)\}=\{\A\in \mathsf P^2(X):x^{-1}A\in\A\}$$is open in $\mathsf P^2(X)$. The same is true for the preimage
$$\Phi^{-1}(\langle x,A\rangle^-)=\{\A\in \mathsf P^2(X):x\notin \Phi_\A(A)\}=\{\A\in \mathsf P^2(X):x^{-1}A\notin\A\}$$which also is open in $\mathsf P^2(X)$.

Since the spaces $\mathsf P^2(X)\cong\{0,1\}^{\mathsf P(X)}$ and $\End(\mathsf P(X))\subset\mathsf P(X)^{\mathsf P(X)}$ are compact and Hausdorff, the continuity of the map $\Phi$ implies the continuity of its inverse $\Phi^{-1}$. Consequently, $\Phi:\mathsf P^2(X)\to\End(\mathsf P(X))$ is a topological isomorphism of compact right-topological semigroups.   
\end{proof}

\begin{remark} The functions representations $\Phi_\A$ of some families $\A\subset\mathsf P(X)$ have transparent topological interpretations. For example, if $\A$ is the filter of neighborhoods of the identity element $e$ of a left-topological group $X$ and $\A^\perp=\{B\subset X:\forall A\in\A\;\;(B\cap A\ne\emptyset)\}$, then  for any subset $B\subset X$ the set $\Phi_{\A}(B)$ coincides with the interior of the set $B$ while $\Phi_{\A^\perp}(B)$ with the closure of $B$ in $X$! 
\end{remark}

Theorem~\ref{t4.1} has a strategical importance because it allows us to translate (usually difficult) problems concerning the structure of the semigroup $\mathsf P^2(X)$ to (usually more tractable) problems about the endomorphism monoid $\End(\mathsf P(X))$. In particular, Theorem~\ref{t4.1} implies ``for free'' that the binary operation on $\mathsf P^2(X)$ is associative and right-topological and hence $\mathsf P^2(X)$ indeed is a compact right-topological semigroup.

Now let us investigate the interplay between the properties of a family $\A\in\mathsf P^2(X)$ and those of its function representation $\Phi_\A$.

Let us define a family $\A\subset\mathsf P(X)$ to be 
\begin{itemize}
\item {\em monotone} if for any subsets $A\subset B\subset X$ the inclusion $A\in\A$ implies $B\in\A$;
\item {\em left-invariant} if for any $A\in\A$ and $x\in X$ we get $xA\in\A$.
\end{itemize}

Respectively, a function $\varphi:\mathsf P(X)\to\mathsf P(X)$ is called
\begin{itemize}
\item {\em monotone} if  $\varphi(A)\subset\varphi(B)$ for any subsets $A\subset B\subset X$;
\item {\em symmetric} if $\varphi(X\setminus A)=X\setminus\varphi(A)$ for every $A\subset X$.
\end{itemize}

\begin{proposition}\label{p4.3} For an equivariant function $\varphi\in \End(\mathsf P(X))$ the family $\Phi^{-1}(\varphi)=\{A\subset X:e\in\varphi(A)\}$ is
\begin{enumerate}
\item[\textup{(1)}] monotone if and only if $\varphi$ is monotone;
\item[\textup{(2)}] left-invariant if and only if $\varphi(\mathsf P(X))\subset\{\emptyset,X\}$;
\item[\textup{(3)}] maximal linked if and only if $\varphi$ is monotone and symmetric.
\end{enumerate}
\end{proposition}

\begin{proof} Let $\A=\Phi^{-1}(\varphi)$. 

1. If $\varphi$ is monotone, then for any sets $A\subset B$ with $A\in\A$ we get $e\in\varphi(A)\subset\varphi(B)$ and hence $B\in\A$, which means that the family $\A$ is monotone.

Now assume conversely that the family $\A$ is monotone and take any sets $A\subset B\subset X$. Note that for any $x\in X$ with $xA\in\A$ we get $xB\in\A$. Then$$\varphi(A)=\{x\in X:x^{-1}A\in \A\}\subset \{x\in X:x^{-1}B\in\A\}=\varphi(B),$$witnessing that the function $\varphi$ is monotone. 
\smallskip

2. If the family $\A$ is left-invariant, then for each $A\in\A$ we get $\varphi(A)=\{x\in X:x^{-1}A\in\A\}=X$ and for each $A\notin\A$ we get $\varphi(A)=\{x\in X:x^{-1}A\in\A\}=\emptyset$. 

Now assume conversely that $\varphi(\mathsf P(X))\subset\{\emptyset,X\}$. Then for each $A\in\A$ we get $e\in\varphi(A)=X$ and then for each $x\in X$, the equivariance of $\varphi$ guarantees that $\varphi(xA)=x\varphi(A)=xX=X\ni e$ and thus  $xA\in\A$, witnessing that the family $\A$ is invariant.
\smallskip

3. Assume that the family $\A$ is maximal linked. By the maximality, $\A$ is monotone. Consequently, its function representation $\varphi$ is monotone. The maximal linked property of $\A$ guarantees that for any subset $A\subset X$ we get $(A\in\A)\Leftrightarrow (X\setminus A\notin\A)$. Then 
$$
\begin{aligned}
\varphi(X\setminus A)&=\{x\in X:x^{-1}(X\setminus A)\in\A\}=\{x\in X:X\setminus x^{-1}A\in\A\}=\\
&=\{x\in X:x^{-1}A\notin\A\}=X\setminus \{x\in X:x^{-1}A\in\A\}=X\setminus \varphi(A),
\end{aligned}
$$
which means that  the function $\varphi$ is symmetric.

Now assuming that the function $\varphi$ is monotone and symmetric, we shall show that the family $\A=\Phi^{-1}(\varphi)$ is maximal linked. The statement (1) guarantees that $\A$ is monotone. Assuming that $\A$ is not linked, we could find two disjoint sets $A,B\in\A$. Since $\A$ is monotone, we can assume that $B=X\setminus A$. Then $e\in\varphi(A)\cap\varphi(X\setminus A)$, which is impossible as $\varphi(X\setminus A)=X\setminus\varphi(A)$. Thus $\A$ is linked. To show that $\A$ is maximal linked, it suffices to check that for each subset $A\subset X$ either $A$ or $X\setminus A$ belongs to $\A$. Since $\varphi(X\setminus A)=X\setminus \varphi(A)$, either $\varphi(A)$ or $\varphi(X\setminus A)$ contains the neutral element $e$ of the group $X$. In the first case $A\in\A$ and in the second case $X\setminus A\in\A$.
\end{proof}

Let us recall that the aim of this paper is the description of the structure of minimal left ideals of the superextension $\lambda(X)$ of a group $X$. 
Instead of the semigroup $\lambda(X)$ it will be more convenient to consider its isomorphic copy
$$\End_\lambda(\mathsf P(X))=\Phi(\lambda(X))\subset\End(\mathsf P(X))$$ called the {\em function representation of} $\lambda(X)$.

Proposition~\ref{p4.3} implies

\begin{corollary}\label{c4.4} The function representation $\End_\lambda(\mathsf P(X))$ of $\lambda(X)$ consists of equivariant monotone symmetric functions $\varphi:\mathsf P(X)\to\mathsf P(X)$.
\end{corollary}

In order to describe the structure of minimal left ideals of the semigroup $\End_\lambda(\mathsf P(X))$ we shall look for a relatively small subfamily $\mathsf F\subset\mathsf P(X)$ such that the restriction operator 
$$R_{\mathsf F}:\End_\lambda(\mathsf P(X))\to\mathsf P(X)^{\mathsf F},\;\;R_{\mathsf F}:\varphi\mapsto\varphi|\mathsf F,$$
is injective on each minimal left ideal of the semigroup $\Enl(\mathsf P(X))$.

Then the composition 
$$\Phi_{\mathsf F}=R_{\mathsf F}\circ \Phi:\lambda(X)\to\mathsf P(X)^{\mathsf F}$$will be injective on each minimal left ideal of the semigroup $\lambda(X)$. 
By Proposition~\ref{c2.2}, a homomorphism between semigroups is injective on each minimal left ideal if it is injective on some minimal left ideal. Such a special minimal left ideal of the semigroup $\lambda(X)$ will be found in the left ideal of the form $\lambda^\I(X)$ for a suitable left-invariant ideal $\I$ of subsets of the group $X$.

A family $\I$ of subsets of $X$ is called an {\em ideal} on $X$ if
\begin{itemize}
\item $X\notin \I$;
\item $A\cup B\in\I$ for any $A,B\in\I$;
\item for any $A\in \I$ and $B\subset A$ we get $B\in\I$.
\end{itemize}
Such an ideal $\I$ is called {\em left-invariant} (reps. {\em right-invariant}) if $xA\in\I$ (resp. $Ax\in\I$) for all $A\in\I$ and $x\in X$. An ideal $\I$ will be called {\em invariant} if it is both left-invariant and right-invariant.

The smallest ideal on $X$ is the trivial ideal $\{\emptyset\}$ containing only the empty set. The smallest non-trivial left-invariant ideal on an infinite group $X$ is the ideal $[X]^{<\w}$ of finite subsets of $X$.
This ideal is invariant. 
From now on we shall assume that $\I$ is a left-invariant ideal on a group $X$.

For subsets $A,B\subset X$ we write 
\begin{itemize}
\item $A\subset_\I B$ if $A\setminus B\in\I$, and
\item $A=_\I B$ if $A\subset_\I B$ and $B\subset_\I A$. 
\end{itemize}
The definition of the ideal $\I$ implies that $=_\I$ is an equivalence relation on $\mathsf P(X)$. For a subset $A\subset X$ its equivalence class $\bar{\bar A}^\I=\{B\subset X:B=_\I A\}$ is called the {\em $\I$-saturation} of $A$.

A family $\A$ of subsets of $X$ is defined to be {\em $\I$-saturated} if 
$\bar{\bar A}^\I\subset\A$ for any $A\in \A$.  Let us observe that a monotone family $\A\subset\mathsf P(X)$ is $\I$-saturated if and only if for any $A\in\A$ and $B\in\I$ we get $A\setminus B\in\A$.

Respectively, a function $\varphi:\mathsf P(X)\to\mathsf P(X)$ is called {\em $\I$-saturated} if $\varphi(A)=\varphi(B)$ for any subsets $A=_\I B$ of $X$.

\begin{proposition}\label{p4.5} A family $\A\subset\mathsf P(X)$ is $\I$-saturated if and only if its function representation $\Phi_\A:\mathsf P(X)\to\mathsf P(X)$ is $\I$-saturated.
\end{proposition}

\begin{proof} Assume that $\A$ is $\I$-saturated and take two subsets $A=_\I B$ of $X$. We need to show that $\Phi_\A(A)=\Phi_\A(B)$. The left-invariance of the ideal $\I$ implies that for every $x\in X$ we get $xA=_\I xB$ and hence $(xA\in\A)\;\Leftrightarrow\;(xB\in \A)$.
Then $$\Phi_\A(A)=\{x\in X:x^{-1}A\in\A\}=\{x\in X:x^{-1}B\in\A\}=\Phi_\A(B).$$

Now assume conversely that the function representation $\Phi_\A$ is $\I$-saturated and take any subsets $A=_\I B$ with $A\in\A$. Then $e\in\Phi_\A(A)=\Phi_\A(B)$, which implies that $B\in\A$.
\end{proof}

For an left-invariant ideal $\I$ on a group $X$ let $\lambda^\I(X)\subset\lambda(X)$ be the subspace of $\I$-saturated maximal linked systems on $X$ and $\Enl^\I(\mathsf P(X))\subset\Enl(\mathsf P(X))$ be the subspace consisting of $\I$-saturated monotone symmetric endomorphisms of the $X$-act $\mathsf P(X)$. It is clear that for any functions $f,g:\mathsf P(X)\to\mathsf P(X)$ the composition $f\circ g$ is $\I$-saturated provided so is the function $g$.
This trivial remark (and Lemma~\ref{maxfree} below) imply:

\begin{proposition} For any ideal $\I$ the function representation $\Phi:\lambda^\I(X)\to\Enl^\I(\mathsf P(X))$ is a topological isomorphism between the closed left ideals $\lambda^\I(X)$ and $\Enl^\I(\mathsf P(X))$ 
of the semigroups $\lambda(X)$ and $\Enl(\mathsf P(X))$, respectively.
\end{proposition}

The following lemma (combined with Zorn's Lemma) implies that the sets $\lambda^\I(X)$ and $\Enl^\I(\mathsf P(X))$ are not empty.

\begin{lemma}\label{maxfree} Each maximal $\I$-saturated linked system $\LL$ on $X$ is maximal linked.
\end{lemma}

\begin{proof} We need to show that each set $A\subset X$ that meets all sets $L\in\LL$ belongs to $\LL$. We claim that $A\notin\I$. Otherwise, taking any subset $L\in\LL$, we get $L\setminus A=_\I L$ and hence $L\setminus A$ belongs to $\LL$, which is not possible as $L\setminus A$ misses the set $A$. Since $A\notin\I$, the $\I$-saturated family $\bar{\bar A}^\I$ is linked.

We claim that the $\I$-saturated family $\bar{\bar A}^\I\cup\LL$ is linked. Assuming the converse, we would find two disjoint sets $A'\in\bar{\bar A}^\I$ and $L\in\LL$. Then $L\cap A=_\I L\cap A'=\emptyset$ and hence the set $L\setminus A=_\I L$ belongs to $\LL$, which is not possible as  this set misses $A$.

Now we see that the family $\bar{\bar A}^\I\cup\LL$, being $\I$-saturated and linked, coincides with the maximal $\I$-saturated linked system $\LL$. Then $A\in \bar{\bar A}^\I\cup\LL=\LL$.
\end{proof}

Given a subfamily $\mathsf F\subset\mathsf P(X)$ consider the restriction operator $$R_{\mathsf F}:\mathsf P(X)^{\mathsf P(X)}\to\mathsf P(X)^{\mathsf F},\;\;R_{\mathsf F}:f\mapsto f|\mathsf F,$$ and let $\Enl(\mathsf F)=R_{\mathsf F}(\Enl(\mathsf P(X)))$ and $\Enl^\I(\mathsf F)=R_{\mathsf F}(\Enl^\I(\mathsf P(X))$ for a left-invariant ideal $\I$ on $X$. The space $\Enl(\mathsf F)$ is compact and Hausdorff as a continuous image of a compact Hausdorff space.

A subfamily $\mathsf F\subset\mathsf P(X)$ is called {\em $\lambda$-invariant} if $\Phi_\LL(\mathsf F)\subset\mathsf F$ for each maximal linked system $\LL\in\lambda(X)$. By Corollary~\ref{c4.4}, $\mathsf F$ is $\lambda$-invariant if and only if $f(\mathsf F)\subset\mathsf F$ for each equivariant monotone symmetric function $f:\mathsf P(X)\to \mathsf P(X)$.   

If a family $\mathsf F\subset\mathsf P(X)$ is $\lambda$-invariant, then the space $\Enl(\mathsf F)\subset\mathsf F^{\mathsf F}$ is a compact right-topological semigroup with respect to the operation of composition of functions and the restriction operator $R_{\mathsf F}:\Enl(\mathsf P(X))\to\Enl(\mathsf F)$ is a surjective continuous semigroup homomorphism. In this case the composition
$$\Phi_{\mathsf F}=R_{\mathsf F}\circ\Phi:\lambda(X)\to\Enl(\mathsf F)$$also is a surjective continuous semigroup homomorphism and $\Enl^\I(\mathsf F)=\Phi_{\mathsf F}(\lambda^\I(X))$ is a left ideal in the semigroup $\Enl(\mathsf F)$.

In the following proposition we characterize functions that belong to the space $\Enl^\I(\mathsf F)$ for an $\I$-saturated left-invariant symmetric subfamily $\mathsf F\subset\mathsf P(X)$.
A family $\mathsf F\subset\mathsf P(X)$ is called
 {\em symmetric} if for each set $A\in \mathsf F$ the complement $X\setminus A\in\mathsf F$.

\begin{theorem}\label{t4.8} For a left-invariant ideal $\I$ on a group $X$ and a $\I$-saturated symmetric left-invariant family $\mathsf F\subset\mathsf P(X)$, a  function $\varphi:\mathsf F\to\mathsf P(X)$ belongs to $\Enl^\I(\mathsf F)$ if and only if $\varphi$ is equivariant, symmetric, monotone, and $\I$-saturated.
\end{theorem}

\begin{proof} The ``only if'' part follows immediately from Corollary~\ref{c4.4}. To prove the ``if'' part, fix any equivariant monotone symmetric $\I$-saturated function $\varphi:\mathsf F\to\mathsf P(X)$ and consider the 
the families $$\LL_\varphi=\{x^{-1}A:A\in\mathsf F,\;x\in \varphi(A)\}\;\;\mbox{and}\;\; 
\bar{\bar\LL}^\I_\varphi=\bigcup_{A\in\LL}\bar{\bar A}^\I.$$

 We claim that the family $\bar{\bar \LL}_\varphi^\I$ is linked. Assuming the
converse, we could find two sets $A,B\in\mathsf F$ and two points $x\in
\varphi(A)$ and $y\in \varphi(B)$ such that $x^{-1}A\cap y^{-1}B\in\I$. Then
$yx^{-1}A\subset_{\I} X\setminus B$ and hence $yx^{-1}A\subset (X\setminus B)\cup C$ for some set $C\in\I$. Since $\mathsf F$ is symmetric and $\I$-saturated, the set $(X\setminus B)\cup C=_\I X\setminus B$ belongs to the family $\mathsf F$. 
 Applying to the chain of the inclusions $$yx^{-1}A\subset (X\setminus B)\cup C=_\I X\setminus B$$ the equivariant monotone symmetric $\I$-saturated function $\varphi$, we get the chain
$$yx^{-1}\varphi(A)\subset\varphi((X\setminus B)\cup C)=\varphi(X\setminus B)=X\setminus \varphi(B).$$  Then $x^{-1}\varphi(A)\subset X\setminus y^{-1}\varphi(B)$, which is
not possible because the neutral element $e$ of the group $X$
belongs to $x^{-1}\varphi(A)\cap y^{-1}\varphi(B)$.

Enlarge the $\I$-saturated linked family $\bar{\bar \LL}_\varphi^\I$ to a maximal $\I$-saturated
linked family $\LL$, which is maximal linked by Lemma \ref{maxfree} and
thus $\LL\in\lambda^{\I}(X)$. We claim that $\Phi_\LL|\mathsf
F=\varphi$. Indeed, take any set $A\in\mathsf F$ and observe that
$$\varphi(A)\subset\{x\in X:x^{-1}A\in\LL_\varphi\}\subset \{x\in X:x^{-1}A\in\LL\}=\Phi_\LL(A).$$
To prove the reverse inclusion, observe that for any $x\in X\setminus \varphi(A)=\varphi(X\setminus A)$ we get $x^{-1}(X\setminus A)=X\setminus x^{-1}A\in\LL_\varphi\subset\LL$. Since $\LL$ is linked, $x^{-1}A\notin\LL$ and hence $x\notin\Phi_\LL(A)$.
\end{proof}

\begin{corollary} For any symmetric left-invariant family $\mathsf F\subset\mathsf P(X)$, a  function $\varphi:\mathsf F\to\mathsf P(X)$ belongs to $\Enl(\mathsf F)$ if and only if $\varphi$ is equivariant, symmetric, and monotone.
\end{corollary}

\section{Twin and $\I$-twin subsets of groups}\label{s4}

In this section we start studying very interesting objects called twin sets.
For an abelian (more generally, twinic) group $X$ twin subsets of $X$ form a subfamily $\mathsf T\subset\mathsf P(X)$ for which the  function representation $\Phi_{\mathsf T}:\lambda(X)\to\Enl(\mathsf T)$ is injective on each minimal left ideal of the superextension $\lambda(X)$. The machinery related to twin sets will be developed in Sections~\ref{s4}--\ref{s15}, after which we shall return back to studying minimal (left) ideals of the semigroups $\lambda(X)$ and $\End(X)$.

For a subset $A$ of a group $X$ consider the following three subsets of $X$:
$$
\begin{aligned}
&\Fix(A)=\{x\in X:xA=A\},\\
&\Fix^-(A)=\{x\in X:xA=X\setminus A\}, \\
&\Fix^\pm(A)=\Fix(A)\cup\Fix^-(A).
\end{aligned}$$

\begin{definition}
A subset $A\subset X$ is defined to be
\begin{itemize}
\item {\em twin} if $xA=X\setminus A$ for some $x\in X$,
\item {\em pretwin} if $xA\subset X\setminus A\subset yA$ for some points $x,y\in X$.
\end{itemize}
The families of twin and pretwin subsets of $X$ will be denoted by $\mathsf T$ and $\pT$, respectively.
\end{definition}
Observe that a set $A\subset X$ is twin if and only if $\Fix^-(A)$ is not empty.

The notion of a twin set has an obvious ``ideal'' version. 
For a left-invariant ideal $\I$ of subsets of a group $X$, and a subset $A\subset X$ consider the following subsets of $X$:
$$
\begin{aligned}
&\IFix(A)=\{x\in X:xA=_\I A\},\\
&\IFix^-(A)=\{x\in X:xA=_\I X\setminus A\},\\ &\IFix^\pm(A)=\IFix(A)\,\cup\,\IFix^-(A).
\end{aligned}$$ 

\begin{definition}
A subset $A\subset X$ is defined to be
\begin{itemize}
\item {\em $\I$-twin} if $xA=_\I X\setminus A$ for some $x\in X$,
\item {\em $\I$-pretwin} if $xA\subset_\I X\setminus A\subset_\I yA$ for some points $x,y\in X$.
\end{itemize}
The families of $\I$-twin and $\I$-pretwin subsets of $X$ will be denoted by $\mathsf T^\I$ and $\pT^\I$, respectively.
\end{definition}

It is clear that $\Tau^{\{\emptyset\}}=\Tau$ and $\pT^{\{\emptyset\}}=\pT$.

\begin{proposition}\label{p5.3} 
For each subset $A\subset X$ the set $\IFix^\pm(A)$ is a subgroup
in $X$. The set $A$ is $\I$-twin if and only if $\IFix(A)$ is a normal subgroup of index 2 in $\IFix^\pm(A)$.
\end{proposition}

\begin{proof} If the set $A$ is not $\I$-twin, then $\IFix^-(A)=\emptyset$ and then $\IFix^\pm(A)=\IFix(A)=\{x\in X:xA=_\I A\}$ is a subgroup of $X$ by the transitivity and the left-invariance of the equivalence relation $=_\I$.

So, we assume that $A$ is $\I$-twin, which means that $\IFix^-(A)\ne\emptyset$.
 To show that $\IFix^\pm(A)$
is a subgroup in $X$, take any two points $x,y\in \IFix^\pm(A)$. We claim that $xy^{-1}\in\IFix^\pm(A)$.

This is clear if $x,y\in\IFix(A)\subset\IFix^\pm(A)$.
If $x\in\IFix(A)$ and $y\in\IFix^{-}(A)$, then $xA=_\I A$, $yA=_\I X\setminus A$ and thus $A=_\I X\setminus y^{-1}A$ which implies $y^{-1}A=_\I X\setminus A$. Then $xy^{-1}A=_\I x(X\setminus
A)= X\setminus xA=_\I X\setminus A$, which means that 
$xy^{-1}\in\IFix^{-}(A)\subset\Fix^\pm(A)$. 

If $x,y\in\IFix^{-}(A)$,
then $xA=_\I X\setminus A$, $y^{-1}A=_\I X\setminus A$. This implies that
$xy^{-1}A=_\I x(X\setminus A)=_\I X\setminus xA=_\I X\setminus (X\setminus
A)=_\I A$ and consequently $xy^{-1}\in\IFix(A)$.

To show that $\IFix(A)$ is a subgroup of index 2 in $\Fix^\pm(A)$, fix any element $g\in\IFix^-(A)$. Then for every $x\in \IFix(A)$ we get  $gxA=_\I gA=_\I X\setminus A$ and thus
$gx\in\IFix^{-}(A)$. This yields $\IFix^{-}(A)=g(\IFix (A))$, which means that the subgroup $\IFix(A)$ has index 2 in the group $\IFix^\pm(A)$.
\end{proof}

The following proposition shows that the family $\Tau^\I$ of $\I$-twin
sets of a group $X$ is left-invariant.

\begin{proposition}\label{p5.4}
For any $\I$-twin set $A\subset X$ and any $x\in X$ the set $xA$ is $\I$-twin and $\IFix^-(xA)=x\,(\IFix^-(A))\,x^{-1}$.
\end{proposition}

\begin{proof}
To see that $xA$ is an $\I$-twin set, take any $z\in\IFix^-(A)$ and observe that $$X\setminus xA= x(X\setminus A)=_\I xzA= xzx^{-1}xA,$$ which means that $xzx^{-1}\in\IFix^-(xA)$ for every $z\in \IFix^-(A)$. Hence $\IFix^-(xA)=x\,(\IFix^-(A))\,x^{-1}$.
\end{proof}

The preceding proposition implies that the family $\Tau^\I$ of $\I$-twin subsets of $X$ can be considered as an $X$-act with respect to the left action
$$\cdot:X\times\Tau^\I\to \Tau^\I,\quad \cdot: (x,A)\mapsto xA$$
 of the group $X$. By $[A]=\{xA:x\in X\}$ we denote the orbit of a $\I$-twin set $A\in\Tau^\I$ and by $[\Tau^\I]=\{[A]:A\in\Tau^\I\}$  the orbit space. 
If $\I=\{\emptyset\}$ is a trivial ideal, then we write $[\Tau]$ instead of $[\Tau^\I]$.

\section{Twinic groups}\label{s:tg} 

A left-invariant ideal $\I$ on a group $X$ is called {\em twinic} if for any subset $A\subset X$ and points $x,y\in X$ with $xA\subset_\I X\setminus A\subset_\I yA$ we get $A=_\I B$. In this case the families $\pT^\I$ and $\Tau^\I$ coincide.

A group $X$ is defined to be {\em twinic } if it admits a twinic ideal $\I$.
It is clear that in a twinic group $X$ the intersection $\TI$ of all twinic ideals is the smallest twinic ideal in $X$ called {\em the twinic ideal} of $X$. The structure of the twinic ideal $\TI$ can be described as follows.

Let $\TI_0=\{\emptyset\}$ and for each $n\in\w$ let $\TI_{n+1}$ be the ideal generated by sets of the form $yA\setminus xA$ where $xA\subset_{\TI_n} X\setminus A\subset_{\TI_n} yA$ for some $A\subset X$ and $x,y\in X$. By induction it is easy to check that $\TI_n\subset\TI_{n+1}\subset\I$ is an invariant ideal and hence $\TI=\bigcup_{n\in\w}\TI_n\subset\I$ is a well-defined (smallest) twinic ideal on $X$. This ideal $\TI$ is invariant. 

In fact, the above constructive definition of the additive invariant family $\TI$ is valid for each group $X$. However, $\TI$ is an ideal if and only if the group $X$ is twinic.

We shall say that a group $X$ has {\em trivial twinic ideal} if the trivial ideal $\I=\{\emptyset\}$ is twinic. This happen if and only if for any subset $A\subset X$ with $xA\subset X\setminus A\subset yA$ we get $xA=X\setminus A=yA$. In this case the twinic ideal $\TI$ of $X$ is trivial.

The class of twinic groups is sufficiently wide. In particular, it contains all amenable groups. Let us recall that a group $X$ is called {\em amenable} if it admits a Banach measure $\mu:\mathsf P(X)\to[0,1]$, which is a left-invariant probability measure defined on the family of all subsets of $\mathsf P(X)$. In this case the family$$\mathcal N_\mu=\{A\subset X:\mu(A)=0\}$$is an left-invariant ideal in $X$.
It is well-known that the class of amenable groups contains all abelian groups and is closed with respect to many operations over groups, see \cite{Pat}.

A subset $A$ of an amenable group $X$ is called {\em absolutely null\/} if $\mu(A)=0$ for each Banach measure $\mu$ on $X$. The family $\mathcal N$ of all absolutely null subsets is an ideal on $X$. This ideal coincides with the intersection $\mathcal N=\bigcap_{\mu}\mathcal N_\mu$ where $\mu$ runs over all Banach measures  of $X$.

\begin{theorem} Each amenable group $X$ is twinic. The twinic ideal $\TI$ of $X$ lies in the ideal $\mathcal N$ of absolute null subsets of $X$.
\end{theorem}

\begin{proof} It suffices to check that the ideal $\mathcal N$ is twinic. Take any set $A\subset X$ such that $xA\subset_{\N} X\setminus A\subset_{\mathcal N}yA$ for some $x,y\in X$. We need to show that $\mu(yA\setminus xA)=0$ for each Banach measure $\mu$ on $X$. It follows from $xA\subset_{\N} X\setminus A\subset_{\mathcal N}yA$ and the invariance of the Banach measure $\mu$ that
$$\mu(A)=\mu(xA)\le \mu(X\setminus A)\le\mu(yA)=\mu(A)$$and hence 
$\mu(yA\setminus xA)=\mu(A)-\mu(A)=0$.
\end{proof}

Next, we show that the class of twinic groups contains also some non-amenable groups. The simplest example is the Burnside group $B(n,m)$ for $n\ge 2$ and odd $m\ge 665$. We recall that the Burnside group $B(n,m)$ is generated by $n$ elements and one relation $x^m=1$. Adian \cite{Ad} proved that for $n\ge 2$ and any odd $m\ge 665$ the Burnside group $B(n,m)$ is not amenable, see also \cite{Osin} for a stronger version of this result. The following theorem implies that each Burnside group, being a torsion group, is twinic. Moreover, its twinic ideal $\TI$ is trivial!

\begin{theorem}\label{t6.2} A group $X$ has trivial twinic ideal $\TI=\{\emptyset\}$ if and only if the product $ab$ of any elements $a,b\in X$ belongs to the subsemigroup of $X$ generated by the set $b^\pm\cdot a^\pm$ where $a^\pm=\{a,a^{-1}\}$.
\end{theorem}

\begin{proof} To prove the ``if'' part, assume that  $\TI\ne\{\emptyset\}$.
Then $\TI_{n+1}\ne\{\emptyset\}=\TI_{n}$ for some $n\in\w$ and we can find a subset $A\subset X$ and points $a,b\in X$ such that $a^{-1}A\subset X\setminus A\subset bA$ but $a^{-1}A\ne bA$.  Consider the subsemigroup $\Fix_\subset(A)=\{x\in X:xA\subset A\}\subset X$ and observe that $b^{-1}a^{-1}\in\Fix_\subset(A)$. The inclusion $a^{-1}A\subset X\setminus A$ implies $a^{-1}A\cap A=\emptyset$ which is equivalent to $A\cap aA=\emptyset$ and yields $aA\subset X\setminus A\subset bA$. Then  $b^{-1}a\in\Fix_\subset(A)$.

Now consider the chain of the equivalences
$$X\setminus A\subset bA\;\Leftrightarrow \; A\cup bA=X\;\Leftrightarrow \; 
b^{-1}A\cup A=X \;\Leftrightarrow \; X\setminus A\subset b^{-1}A$$ and combine the last inclusion with $aA\cup a^{-1}A\subset X\setminus A$ to obtain $ba,ba^{-1}\in \Fix_{\subset}(A)$. Now we see that the subsemigroup $S$ of $X$ generated by the set $\{1,ba,ba^{-1},b^{-1}a,b^{-1}a^{-1}\}$ lies in $\Fix_\subset(A)$. Observe that $b^{-1}a^{-1}A\subsetneqq A$ implies $abA\not\subset A$, $ab\notin\Fix_\subset(A)\supset S$, and finally $ab\notin S$.
This completes the proof of the ``if'' part.

To prove the ``only if'' part, assume that the group $X$ contains elements $a,b$ whose product $ab$ does not belong to the subsemigroup generated by $b^\pm a^\pm$ where $a^\pm=\{a,a^{-1}\}$ and $b^\pm=\{b,b^{-1}\}$. Then $ab$ does not belongs also to the subsemigroup $S$ generated by $\{1\}\cup b^\pm a^\pm$. Observe that $a^\pm S=S^{-1}a^{\pm}$ and $b^\pm S^{-1}=Sb^\pm$.

We claim that
\begin{equation}\label{sa}S\cap a^\pm S=\emptyset \mbox{ and }S\cap Sb^\pm=\emptyset.
\end{equation}
Assuming that $S\cap a^\pm S\ne \emptyset$ we would find a point $s\in S$ such that $as\in S$ or $a^{-1}s\in S$. If $as\in S$, then $bs^{-1}=b(as)^{-1}a\in bS^{-1}a\subset Sb^\pm a\subset S$ and hence $b=bs^{-1}s\in S\cdot S\subset S$. Then $a^\pm=b(b^{-1}a^\pm)\subset S\cdot S\subset S$, $b^\pm=(b^\pm a)a^{-1}\in S\cdot S\subset S$ and finally $ab\in S\cdot S\subset S$, which contradicts $ab\notin S$. By analogy we can treat the case $a^{-1}s\in S$ and also prove that $S\cap Sb^\pm=\emptyset$.

Consider the family $\mathcal P$ of all pairs $(A,B)$ of disjoint subsets of $X$ such that
\begin{itemize}
\item[(a)] $a^\pm A\subset B$ and $b^\pm B\subset A$;
\item[(b)] $S^{-1}B\subset B$;
\item[(c)] $1\in A$, $ab\in B$.
\end{itemize}
The family $\mathcal P$ is partially ordered by the relation $(A,B)\le(A',B')$ defined by $A\subset A'$ and $B\subset B'$. 

We claim that the pair $(A_0,B_0)=(S\cup Sb^\pm ab,S^{-1}a^\pm\cup S^{-1} ab)$ belongs to $\mathcal P$. Indeed, 
$$a^\pm A_0=a^\pm S\cup a^\pm Sb^\pm ab\subset S^{-1}a^\pm\cup S^{-1}a^\pm b^\pm ab\subset S^{-1}a^\pm\cup S^{-1} ab\subset B_0.$$By analogy we check that $b^\pm B_0\subset A_0$.
The items (b), (c) trivially follow from the definition of $A_0$ and $B_0$. It remains to check that the sets $A_0$ and $B_0$ are disjoint.

This will follow as soon as we check that
\begin{itemize}
\item[(d)] $S\cap S^{-1}a^\pm =\emptyset$,
\item[(e)] $S\cap S^{-1} ab=\emptyset$,
\item[(f)] $Sb^\pm ab\cap S^{-1}a^\pm=\emptyset$,  
\item[(g)] $Sb^\pm ab\cap S^{-1} ab=\emptyset$.
\end{itemize}
The items (d) and (g) follow from (\ref{sa}). The item (e) follows from $ab\notin S\cdot S=S$. By the same reason, we get the item (f) which is equivalent to $ab\notin b^\pm S^{-1}\cdot S^{-1}a^\pm=b^\pm S^{-1} a^\pm=Sb^\pm a^\pm\subset S$. 

Thus the partially ordered set $\mathcal P$ is not empty and we can apply  Zorn's Lemma to find a maximal pair  $(A,B)\ge(A_0,B_0)$ in $\mathcal P$. We claim that $A\cup B=X$.
Assuming the converse, we could take any point $x\in X\setminus(A\cup B)$ and put $A'=A\cup Sx$, $B'=B\cup a^\pm Sx$. It is clear that $a^\pm A'\subset B'$ and $b^\pm B'\subset A'$, $S^{-1}B'=S^{-1}B\cup S^{-1}a^\pm Sx\subset B\cup a^\pm SSx=B'$, $1\in A\subset A'$ and $ab\in B\subset B'$. 

Now we see that the inclusion $(A',B')\in\mathcal P$ will follow as soon as we check that $A'\cap B'=\emptyset$. The choice of $x\notin B=S^{-1}B$ guarantees that $Sx\cap B=\emptyset$. Assuming that $a^{\pm}Sx\cap A\ne\emptyset$, we would conclude that $x\in S^{-1}a^\pm A\subset S^{-1}B\subset B$, which contradicts the choice of $x$.
Finally, the sets $Sx$ and $a^\pm Sx$ are disjoint because of the property (\ref{sa}) of $S$.
Thus we obtain a contradiction: $(A',B')\in\mathcal P$ is strictly greater than the maximal pair $(A,B)$. This contradiction shows that $X=A\cup B$ and consequently,
$aA\subset X\setminus A=B\subset bA$, which means that the set $A$ is pretwin and then $bA\setminus aA\in\TI_1\subset\TI$. Since $1\in A\setminus b^{-1}a^{-1}A$, we conclude that $bA\setminus a^{-1}A\ni b$ is not empty and thus $\TI\ne\{\emptyset\}$.
\end{proof}

We recall that a group $X$ is {\em periodic} (or else a {\em torsion group})  if each element $x\in X$ has finite order (which means that $x^n=e$ for some $n\in\IN$). We shall say that a group $X$ has {\em periodic commutators} if for any $x,y\in G$ the commutator $[x,y]=xyx^{-1}y^{-1}$ has finite order in $X$. It is interesting to note that this condition is strictly weaker than the requirement for $X$ to have periodic commutator subgroup $X'$ (we recall that the commutator subgroup $X'$ coincides with the set of finite products of commutators), see \cite{DK}.

\begin{proposition}\label{p9.4} Each group $X$ with periodic commutators has trivial twinic ideal $\TI=\{\emptyset\}$.
\end{proposition}

\begin{proof} Since $X$ has periodic commutators, for any points $x,y\in X$ there is a number $n\in\IN$ such that
$$xyx^{-1}y^{-1}=(yxy^{-1}x^{-1})^{-1}=(yxy^{-1}x^{-1})^{n}$$ and thus $xy=(yxy^{-1}x^{-1})^{n}\cdot yx$ belongs to the semigroup generated by the set $y^\pm\cdot x^\pm$. Applying Theorem~\ref{t6.2}, we conclude that the group $X$ has trivial twinic ideal $\TI=\{\emptyset\}$.
\end{proof}

We recall that a
group $G$ is called {\em abelian-by-finite} (resp. {\em finite-by-abelian}) if $G$ contains a
normal Abelian (resp. finite) subgroup $H\subset G$ with finite (resp. Abelian) quotient $G/H$. Observe that each finite-by-abelian group has periodic commutators and hence has trivial twinic ideal $\TI$.

In contrast, any abelian-by-finite groups, being amenable, is twinic but its twinic ideal $\TI$ need not be trivial. The simplest counterexample is the isometry group $\Iso(\IZ)$ of the group
$\IZ$ of integers endowed with the Euclidean metric.

\begin{example}\label{ex6.4} The abelian-by-finite group $X=\Iso(\IZ)$ is twinic. Its twinic ideal $\TI$ coincides with the ideal $[X]^{<\w}$ of all finite subsets of $X$.
\end{example}

\begin{proof} Let $a:x\mapsto x+1$ be the translation and $b:x\mapsto -x$ be the inversion of the group $\IZ$. It is easy to see that the elements $a,b$ generate the isometry group $X=\Iso(\IZ)$ and satisfy the relations $b^2=1$ and $bab^{-1}=a^{-1}$. Let $Z=\{a^n:n\in\IZ\}$ be the cyclic subgroup of $X$ generated by the translation $a$. This subgroup $Z$ has index 2 in $X=Z\cup Zb$.

First we show that the ideal $\I=[X]^{<\w}$ of finite subsets of $X$ is twinic. Let $A\subset X$ be a subset with $xA\subset_\I X\setminus A\subset_\I yA$ for some $x,y\in X$. We need to show that $yA=_\I xA$.

We consider three cases.
\smallskip

1) $x,y\in Z$. In this case the elements $x,y$ commute. 
The $\I$-inclusion $xA\subset_\I yA$ implies $y^{-1}xA\subset_\I A$.
We claim that $y^{-1}xA\supset_\I A$. Observe that the $\I$-inclusion 
$xA\subset_\I X\setminus A$ is equivalent to $xA\cap A\in\I$ and to $A\cap x^{-1}A\in\I$, which implies $x^{-1}A\subset_\I X\setminus A$. By analogy, $X\setminus A\subset_\I yA$ is equivalent to $yA\cup A=_\I X$ and to $A\cup y^{-1}A=_\I X$, which implies $X\setminus A\subset_\I y^{-1}A$. Then $x^{-1}A\subset_\I X\setminus A\subset_\I y^{-1}A$ implies $yx^{-1}A\subset_\I A$ and by the left-invariance of $\I$,
$A\subset_\I xy^{-1}A=y^{-1}xA$ (we recall that the elements $x,y^{-1}$ commute).
Therefore, $y^{-1}xA=_\I A$ and hence $xA=_\I yA$. 
\smallskip
  
2) $x\in Z$ and $y\in X\setminus Z$. Repeating the argument from the preceding case, we can show that $xA\subset_\I X\setminus A$ implies $x^{-1}A\subset_\I X\setminus A$. Then we get the chain of $\I$-inclusions:
$$xA\subset_\I X\setminus A\subset_\I yA\subset_\I y(X\setminus xA)=yx(X\setminus A)\subset_\I yxyA=_\I xA,$$where the last $\I$-equality follows from the case (1) since $x,yxy\in Z$. Now we see that $xA=_\I yA$.
\smallskip

3) $x\notin Z$. Then $xA\subset_\I X\setminus A\subset_\I yA$ implies
$$x^{-1}bA=bxb^{-1}bA=bxA\subset_\I X\setminus bA\subset_\I byA=y^{-1}bA.$$Since $x^{-1}b\in Z$, the cases (1),(2) imply the $\I$-equality $x^{-1}bA=_\I y^{-1}bA$. Shifting this equality by $b$, we see that $xA=bx^{-1}bA=_\I by^{-1}bA=yA$.

This completes the proof of the twinic property of the ideal $\I=[X]^{<\w}$. Then the twinic ideal $\TI\subset[X]^{<\w}$. Since $[X]^{<\w}$ is the smallest non-trivial left-invariant ideal on $X$, the equality $\TI=[X]^{<\w}$ will follow as soon as we find a non-empty set in the ideal $\TI$.

For this consider the subset $A=\{a^{n+1},ba^{-n}:n\ge0\}\subset X$ and observe that
$X\setminus A=\{a^{-n},ba^{n+1}:n\ge0\}=bA$ witnessing that $A\in\Tau$. Observe also that $aA=\{a^{n+2},aba^{-n}:n\ge0\}=\{a^{n+2},ba^{-n-1}:n\ge0\}\subsetneqq A$ and thus $baA\subsetneq X\setminus A=bA$. Then $\emptyset\ne baA\setminus bA\in\TI_1\subset\TI$ witnesses that the twinic ideal $\TI$ is not trivial.
\end{proof}
 
Next, we present (an expected) example of a group, which is not twinic.

\begin{example} The free group $F_2$ with two generators is not twinic.
\end{example}

\begin{proof} Assume that the group $X=F_2$ is twinic and let $\TI$ be the twinic ideal of $F_2$. Let $a,b$ be the generators of the free group $F_2$. Each element $w\in F_2$ can be represented by a word in the alphabet $\{a,a^{-1},b,b^{-1}\}$. The word of the smallest length representing $w$ is called the {\em irreducible representation} of $w$. The irreducible word representing the neutral element of $F_2$ is the empty word. Let $A$ (resp. $B$) be the set of words whose 
irreducible representation start with letter $a$ or $a^{-1}$ (resp. $b$ or $b^{-1}$). Consider the subset $$C=\big\{a^{2n}w:w\in B\cup\{e\},\;n\in\IZ\big\}\subset F_2$$ and observe that $abaC\subset X\setminus C=aC$. Then $aC\setminus abaC\in\TI_1$ by the definition of the subideal $\TI_1\subset\TI$. Observe that $a^3baC\subset aC\setminus abaC$ and thus $a^3baC\in\TI_1$. Then also $C\in\TI_1$ and $X\setminus C=aC\in\TI_1$ by the left-invariance of $\TI_1$. By the additivity of $\TI_1$, we finally get $X=C\cup(X\setminus C)\in\TI_1\subset\TI$, which is the desired contradiction.
\end{proof}

Next, we prove some permanence properties of the class of twinic groups.

\begin{proposition} Let $f:X\to Y$ be a surjective group homomorphism. If the group $X$ is twinic, then so is the group $Y$.
\end{proposition}

\begin{proof} Let $\TI$ be the twinic ideal of $X$. It is easy to see that  $\I=\{B\subset Y:f^{-1}(B)\in\TI\}$ is a left-invariant ideal on the group $Y$. We claim that it is twinic. Given any subset $A\subset Y$ with $xA\subset_\I Y\setminus A\subset_\I yA$ for some $x,y\in Y$, let $B=f^{-1}(A)$ and observe that $x'B\subset_\TI X\setminus B\subset_\TI y'B$ for some points $x'\in f^{-1}(x)$ and $y'\in f^{-1}(y)$. The twinic property of the twinic ideal $\TI$ guarantees that $f^{-1}(yA\setminus xA)=y'B\setminus x'B\in\TI$, which implies $yA\setminus xA\in\I$ and hence $xA=_\I Y\setminus A=_\I yA$.
\end{proof}

\begin{problem} Is a subgroup of a twinic group twinic? Is the product of two twinic groups twinic?
\end{problem}

For groups with trivial twinic ideal the first part of this problem has an affirmative solution, which follows from the characterization Theorem~\ref{t6.2}.

\begin{proposition} 
\begin{enumerate}
\item[\textup{(1)}] The class of groups with trivial twinic ideal is closed with respect to taking subgroups and quotient groups.
\item[\textup{(2)}] A group $X$ has trivial twinic ideal if and only if any 2-generated subgroup  of $X$ has trivial twinic ideal.
\end{enumerate}
\end{proposition}

\section{2-Cogroups}

It follows from Proposition~\ref{p5.3} that for a twin subset $A$ of a group $X$ the stabilizer $\Fix(A)$ of $A$ is completely determined by the subset $\Fix^-(A)$ because $\Fix(A)=x\cdot\Fix^-(A)$ for each $x\in\Fix^-(A)$. Therefore, the subset $\Fix^-(A)$ carries all the information about the pair $(\Fix^\pm(A),\Fix(A))$. The sets $\Fix^-(A)$ are particular cases of so-called 2-cogroups defined as follows.

\begin{definition} A subset $K$ of a group $X$ is called a {\em 2-cogroup} if for every $x\in K$ the shift $xK=Kx$ is a subgroup of $X$, disjoint with $K$.

By the {\em index} of a 2-cogroup $K$ in $X$ we understand the cardinality $|X/K|$ of the set $X/K=\{Kx:x\in X\}$.
\end{definition}

2-Cogroups can be characterized as follows.

\begin{proposition}\label{p6.2} A subset $K$ of a group $X$ is a 2-cogroup in $X$ if and only if there is a (unique) subgroup $H^\pm$ of $X$ and a subgroup $H\subset H^\pm$ of index 2 such that $K=H^\pm\setminus H$ and $H=K\cdot K$. 
\end{proposition}

\begin{proof} If $K$ is a 2-cogroup, then for every $x\in K$ the shift $H=xK=Kx$ is a subgroup of $X$ disjoint with $K$. It follows that $K=x^{-1}H=Hx^{-1}$. Since $x^{-1}\in x^{-1}H=K$, the shift $x^{-1}K=Kx^{-1}$ is a subgroup of $X$ according to the definition of a 2-cogroup. Consequently, $x^{-1}Kx^{-1}K=x^{-1}K$, which implies $Kx^{-1}K=K$ and $Hx^{-1}x^{-1}Hx^{-1}=Kx^{-1}K=K=Hx^{-1}$. This implies $x^{-2}\in H$ and $x^2\in H$. Consequently, $xH=x^{-1}x^2H=x^{-1}H=K=Hx^{-1}=Hx^2x^{-1}=Hx$.

Now we are able to show that $H^\pm=H\cup K$ is a group. Indeed,
$$
\begin{aligned}
(H\cup K)\cdot(H\cup K)^{-1}&\subset 
HH^{-1}\cup HK^{-1}\cup KH^{-1}\cup KK^{-1}\subset\\
&\subset H\cup HHx\cup xHH\cup Hx^{-1}xH=H\cup K\cup K\cup H=H^\pm.
\end{aligned}
$$
Since $K=Hx=xH$, the subgroup $H=K\cdot K$ has index 2 in $H^\pm$. 
The uniqueness of the pair $(H^\pm,H)$ follows from the fact that $H=K\cdot K$ and $H^\pm=KK\cup K$.
This completes the proof of the ``only if'' part.

To prove the ``if'' part, assume that $H^\pm$ is a subgroup of $X$ and $H\subset H^\pm$ is a subgroup of index 2 such that $K=H^\pm\setminus H$. Then for every $x\in K$ the shift $xK=Kx=H$ is a subgroup of $X$ disjoint with $K$. This means that $K$ is a 2-cogroup.
\end{proof}

Proposition~\ref{p6.2} implies that for each 2-cogroup $K\subset X$ the set $K^\pm=K\cup KK$ is a subgroup of $X$ and $KK$ is a subgroup of index 2 in $K^\pm$.

By $\mathcal K$ we shall denote the family of all 2-cogroups in $X$. It is partially ordered by the inclusion relation $\subset$ and is considered as an $X$-act endowed with the conjugating action 
$$\cdot:X\times\mathcal K\to\mathcal K,\;\;\cdot:(x,K)\mapsto xKx^{-1},$$
of the group $X$. 
For each 2-cogroup $K\in\mathcal K$ let $\Stab(K)=\{x\in X:xKx^{-1}=K\}$ be the stabilizer of $K$ and $[K]=\{xKx^{-1}:x\in X\}$ be the orbit of $K$. 
By $[\mathcal K]=\{[K]:K\in\mathcal K\}$ be denote the orbit space of $\mathcal K$ by the action of the group $X$.

A cogroup $K\in\mathcal K$ is called {\em normal} if $xKx^{-1}=K$ for all $x\in X$. This is equivalent to saying that $\Stab(K)=X$. 

Since for each twin subset $A\subset X$ the set $\Fix^-(A)$ is a 2-cogroup, the function
$$\Fix^-:\mathsf T\to\mathcal K,\;\;\Fix^-:A\mapsto \Fix^-(A),$$
is well-defined and equivariant according to Proposition~\ref{p5.4}. A similar equivariant function
$$\IFix^-:\mathsf T^\I\to\mathcal K,\;\;\IFix^-:A\mapsto \IFix^-(A),$$can be defined for any left-invariant ideal $\I$ on a group $X$.

Let $\wht\K$ denote the set of maximal elements of the partially ordered set $(\mathcal K,\subset)$.
The following  proposition implies that the set $\wht\K$ lies in the image $\Fix^-(\Tau)$ and is cofinal in $\mathcal K$.

\begin{proposition}\label{p7.3} 
\begin{enumerate}
\item[\textup{(1)}] For any linearly ordered family $\mathcal C\subset\mathcal K$ of 2-cogroups in $X$ the union $\cup\mathcal C$ is a 2-cogroup in $X$.
\item[\textup{(2)}] Each 2-cogroup $K\in\mathcal K$ lies in a maximal 2-cogroup $\wht K\in\wht\K$.
\item[\textup{(3)}] For each maximal 2-cogroup $K\in\wht\K$ there is a twin subset $A\in\Tau$ with $K=\Fix^-(A)$.
\end{enumerate}
\end{proposition}

\begin{proof} 1. Let $\mathcal C\subset\K$ be a linearly ordered family of 2-cogroups of $X$. Since each 2-cogroup $C\in\C$ is disjoint with the group $C\cdot C$ and $C=C\cdot C\cdot C$, we get that the union $K=\cup\C$ is disjoint with the union $\bigcup_{C\in \C}C\cdot C=K\cdot K$ and $K=\bigcup_{C\in\C}C=\bigcup_{C\in\C}C\cdot C\cdot C=K\cdot K\cdot K$ witnessing that $K$ is a 2-cogroup.
\smallskip

2. Since each chain in $\K$ is upper bounded,  Zorn's Lemma  guarantees that each 2-cogroup of $X$ lies in a maximal 2-cogroup.
\smallskip

3. Given a maximal 2-cogroup $K\in\wht{\K}$, consider the subgroups $K\cdot K$ and  $K^\pm=K\cup KK$ of $X$ and choose a subset $S\subset G$ meeting each coset $K^\pm x$, $x\in X$, at a single point. Consider the set $A=KK\cdot S$ and note that $X\setminus A=KS=xA$ for each $x\in K$, which means that $K\subset\Fix^-(A)$. The maximality of $K$ guarantees that $K=\Fix^-(A)$.
\end{proof}

It should be mentioned that in general, $\Fix^-(\Tau)\ne\mathcal K$.

\begin{example} For any twin subset $A$ in the 4-element group $X=C_2\oplus C_2$ the group $\Fix(A)$ is not trivial. Consequently, each singleton $\{a\}\subset X\setminus\{e\}$ is a 2-cogroup that does not belong to the image $\Fix^-(\Tau)$.
\end{example}


A left-invariant subfamily $\mathsf F\subset\Tau$ is called 
\begin{itemize}
\item {\em $\wht{\K}$-covering} if $\wht{\K}\subset\Fix^-(\mathsf F)$ (this means that for each maximal 2-cogroup $K\in\wht{\K}$ there is a twin set $A\in\mathsf F$ with $\Fix^-(A)=K$);
\item {\em minimal $\wht{\K}$-covering} if $\mathsf F$ coincides with each  left-invariant $\wht{\K}$-covering subfamily of $\mathsf F$.
\end{itemize}

Proposition~\ref{p7.3}(3) implies that the family
$$\wht{\Tau}=\{A\in\Tau:\Fix^-(A)\in\wht{\K}\}$$is $\wht{\K}$-covering.

\begin{proposition} For any function $f\in\Enl(\mathsf P(X))$ the family $f(\wht{\Tau})$ is $\wht{\K}$-covering.
\end{proposition}

\begin{proof} The equivariance of the function $f$ and the left-invariance of the family $\wht{\Tau}$ imply the left-invariance of the family $f(\wht{\Tau})$. To see that $f(\wht{\Tau})$ is $\wht{\K}$-covering, fix any maximal 2-cogroup $K\in\wht{\K}$ and using Proposition~\ref{p7.3}, find a twin set $A\subset X$ with $\Fix^-(A)=K$. We claim that $\Fix^-(f(A))=\Fix^-(A)=K$. By Corollary~\ref{c4.4}, the function $f$ is equivariant and symmetric. Then for every $x\in \Fix^-(A)$, applying $f$ to the equality $xA=X\setminus A$, we obtain 
$$x\,f(A)=f(xA)=f(X\setminus A)=X\setminus f(A),$$which means that $x\in\Fix^-(f(A))$ and thus $\Fix^-(A)\subset\Fix^-(f(A))$. Now the maximality of the 2-cogroup $\Fix^-(A)$ guarantees that $\Fix^-(f(A))=\Fix^-(A)$.
\end{proof}

\begin{remark} In Theorem~\ref{t17.1} we shall show that for a twinic group $X$ and a function $f\in \IK\big(\Enl(\mathsf P(X))\big)$ from the minimal ideal of $\Enl(\mathsf P(X))$ the family $f(\wht{\Tau})$ is minimal $\wht{\K}$-covering.
\end{remark}

For each 2-cogroup $K\subset X$ consider the families
$$\Tau_K=\{A\in\Tau:\Fix^-(A)=K\}\mbox{ and }\Tau_{[K]}=\{A\in\Tau:\exists x\in X \mbox{ with }\Fix^-(xA)=K\}.$$
 
The following proposition describing the structure of minimal $\wht{\K}$-covering families can be easily derived from the definitions.

\begin{proposition}\label{p6.6} A left-invariant subfamily $\mathsf F\subset  \wht{\Tau}$ is minimal $\wht\K$-covering if and only if for each $K\in\wht\K$ there is a set $A\in\mathsf F$ such that $\mathsf F\cap\Tau_{[K]}=[A]$.
\end{proposition}

\section{The characteristic group $\HH(K)$ of a 2-cogroup $K$}\label{s8n}

In this section we introduce an important notion of the characteristic group $\HH(K)$ of a 2-cogroup $K$ in a group $X$ and reveal the algebraic structure of characteristic groups of maximal 2-cogroups. 
 
Observe that for each 2-cogroup $K\subset X$ its stabilizer $\Stab(K)=\{x\in X:xKx^{-1}=K\}$ contains $KK$ as a normal subgroup. So, we can consider the quotient group $\HH(K)=\Stab(K)/KK$ called the {\em characteristic group} of the 2-cogroup $K$. Characteristic groups will play a crucial role for description of the structure of maximal subgroups of the minimal ideal of the semigroup $\lambda(X)$.

Observe that for a normal 2-cogroup $K\in\mathcal K$ the characteristic group $\HH(K)$ is equal to the quotient group $X/KK$.

The characteristic group $\HH(K)$ of each maximal 2-cogroup $K\subset X$ has a remarkable algebraic property: it is a 2-group with a unique 2-element subgroup. Let us recall that a group $G$ is called a {\em 2-group} if the order of each element of $G$ is a power of 2. Let us recall some standard examples of 2-groups.

By $Q_8=\{1,i,j,k,-1,-i,-j,-k\}$ we denote the group of quaternions. It is a multiplicative subgroup of the algebra of quaternions $\IH$. The algebra $\IH$ contains the field of complex numbers $\IC$ as a subalgebra.
For each $n\in\w$ let 
$$C_{2^n}=\{z\in\IC:z^{2^n}=1\}$$ be the cyclic group of order $2^n$. The multiplicative subgroup $Q_{2^n}\subset\IH$ generated by the set $C_{2^{n-1}}\cup Q_8$ is called {\em the group of generalized quaternions}, see \cite[\S5.3]{Rob}. The subgroup $C_{2^{n-1}}$ has index 2 in $Q_{2^n}$ and all element of $Q_{2^n}\setminus C_{2^{n-1}}$ have order 4.
According to our definition, $Q_{2^n}=Q_8$ for $n\le 3$.
For $n\ge 3$ the group $Q_{2^n}$ has the presentation
$$\langle x,y\mid x^2=y^{2^{n-2}},\; x^4=1,\; xyx^{-1}=y^{-1}\rangle.$$
The unions
$$C_{2^\infty}=\bigcup_{n\in\w}C_{2^n}\mbox{ \ and \ } Q_{2^\infty}=\bigcup_{n\in\w}Q_{2^n}$$ are called the {\em quasicyclic 2-group} and {\em the infinite group of generalized quaternions}, respectively.
 
\begin{theorem}\label{BCQ} A group $G$ is isomorphic to $C_{2^n}$ or $Q_{2^n}$ for some $1\le n\le\infty$ if and only if $G$ is a 2-group with a unique element of order 2.
\end{theorem}

\begin{proof} The ``only if'' part is trivial. To prove the ``if'' part, assume that $G$ is a 2-group with a unique element of order 2. Denote this element by $-1$ and let $1$ be the neutral element of $G$. If the group $G$ is finite, then by Theorem 5.3.6 of \cite{Rob}, $G$ is isomorphic to $C_{2^n}$ or $Q_{2^n}$ for some $n\in\IN$. So, we assume that $G$ is infinite.

Since $-1$ is a unique element of order 2, the cyclic subgroup $\{-1,1\}$ is the maximal 2-elementary subgroup of $G$ (we recall that a group is 2-elementary if it can be written as the direct sum of 2-element cyclic groups). Now Theorem 2 of \cite{Shun} implies that the group $G$ is \v Cernikov and hence contains a normal abelian subgroup $H$ of finite index. Since $G$ is infinite, so is the subgroup $H$. Let $\tilde H$ be a maximal subgroup of $G$ that contains $H$. 

We claim that $H=\tilde H$ and $H$ is isomorphic to the quasicyclic 2-group $C_{2^\infty}$.
Since $H$ is a 2-group, the unique element $-1$ of the group $G$ belongs to $H$. Let $f:\{-1,1\}\to C_2$ be the unique isomorphism. Since the group $C_{2^\infty}$ is injective, by Baer's Theorem \cite[4.1.2]{Rob}, the homomorphism $f:\{-1,1\}\to C_2\subset C_{2^\infty}$ extends to a homomorphism $\bar f:\tilde H\to C_{2^\infty}$. We claim that $\bar f$ is an isomorphism. Indeed, the kernel $\bar f^{-1}(1)$ of $\bar f$ is trivial since it is a 2-group and contains no element of order 2. So, $\bar f$ is injective and then $\bar f(\tilde H)$ coincides with $C_{2^\infty}$, being an infinite subgroup of $C_{2^\infty}$. By the same reason, $\bar f(H)=C_{2^\infty}$. Consequently, $H=\tilde H$ is isomorphic to $C_{2^\infty}$.

If $G=H$, then $G$ is isomorphic to $C_{2^\infty}$. So, it remains to consider the case of non-abelian group $G\ne H$. 

\begin{claim}\label{cl8.2g} For every $a\in H$ and $b\in G\setminus H$ we get $b^2=-1$ and $bab^{-1}=a^{-1}$.
\end{claim}

\begin{proof} The maximality of the abelian subgroup $H$ implies that $bx\ne xb$ for some element of $H$. Since $H$ is quasicyclic, we can assume that the element $x$ has order $\ge 8$ and $a$ belongs to the cyclic subgroup generated $\la x\ra$. 

Using the fact that the maximal abelian subgroup $H$ has finite index in $G$, one can show that the group $G$ is locally finite. Consequently, the subgroup 
$F=\la b,x\ra$  generated by the set $\{b,x\}$ is finite. By Theorem 5.3.6 \cite{Rob}, this subgroup is isomorphic to $Q_{2^n}$ for some $n\ge 4$.
Analyzing the properties of the group $Q_{2^n}$ we see that $b^2=-1$ and $byb^{-1}=y^{-1}$ for all $y\in\la x\ra$. In particular, $bab^{-1}=a^{-1}$. 
\end{proof}

Next, we show that the subgroup $H$ has index 2 in $G$. This will follow as soon as we show that for each $x,y\in G\setminus H$ we get $xy\in H$. Observe that for every $a\in H$ we get $xyay^{-1}x^{-1}=xa^{-1}x^{-1}=a$, which means that $xy$ commutes with each element of $H$ and hence $xy\in H$ by the maximality of $H$. Now take any elements $b\in G\setminus H$ and  $q\in Q_{2^\infty}\setminus C_{2^\infty}$. 
Extend the isomorphism $\bar f:H\to C_{2^\infty}$ to a map $\tilde f:G\to Q_{2^\infty}$ letting $\tilde f(bh)=q\cdot\bar f(h)$ for $h\in H$. Claim~\ref{cl8.2g} implies that $\tilde f$ is a well-defined isomorphism between $G=H\cup bH$ and $Q_{2^\infty}$.
\end{proof}

\begin{theorem}\label{t8.2} For each maximal 2-cogroup $K\in\wht{\mathcal K}$ in  a group $X$ the characteristic group $\HH(K)=\Stab(K)/KK$ is isomorphic either to $C_{2^n}$ or to $Q_{2^n}$ for some $1\le n\le\infty$.
\end{theorem}

\begin{proof} This theorem will follow from Theorem~\ref{BCQ} as soon as we check that $\HH(K)$ is a 2-group with a unique element of order 2.

 Let $q:\Stab(K)\to \HH(K)$ be the quotient homomorphism. Take any element $x\in K$ and consider its image $d=q(x)$. Since $K=xKK$, the image $q(K)=\{d\}$ is a singleton. Taking into account that $x\notin KK$ and $x^2\in KK$, we see that the element $d$ has order 2 in $\HH(K)$. 
We claim that any other element $a$ of order 2 in $\HH(K)$ is equal to $d$. 
Assume conversely that some element $a\ne d$ of $\HH(K)$ has order 2.

Let $C^\pm$ be the subgroup of $\HH(K)$ generated by the elements $a,d$ and $C$ be the cyclic subgroup generated by the product $ad$. We claim that $d\notin C$. Assuming conversely that $d\in C$, we conclude that $d=(ad)^n$ for some $n\in\IZ$. Then $a=add=ad(ad)^n=(ad)^{n+1}\in C$ and consequently $a=d$ (because cyclic groups contain at most one element of order 2). Therefore $d\notin C$. It is clear that $C^\pm=C\cup dC$, which means that the subgroup $C$ has index 2 in $C^\pm$.

Consider the subgroups $H^\pm=q^{-1}(C^\pm)$, $H=q^{-1}(C)$ and observe that the 2-cogroup $H^\pm\setminus H$ is strictly larger than $K$, which contradicts $K\in\wht{\mathcal K}$. 

Since $d$ is a unique element of order 2 in $\HH(K)$, the cyclic subgroup $D=\{d,d^2\}$ generated by $d$ is normal in $\HH(K)$. Consequently, for each non-trivial subgroup $G\subset \HH(K)$ the product $D\cdot G=G\cdot D$ is a subgroup in $\HH(K)$. Now we see that $G$ must contain $d$. Otherwise, $dG$ would be a 2-cogroup in $\HH(K)$ and its preimage $q^{-1}(dG)$ would be a 2-cogroup in $X$ that contains the 2-cogroup $K$ as a proper subset, which is impossible as $K$ is a maximal 2-cogroup in $X$.

Therefore each non-trivial subgroup of $\HH(K)$ contains $d$. This implies that each element $x\in \HH(K)$ has finite order which is a power of 2, witnessing that $\HH(K)$ is a 2-group with a single element of order 2.
\end{proof}

\section{Twin-generated topologies on groups}\label{s9}

In this section we study so-called twin-generated topologies on groups.
The information obtained in this section will be used in Section~\ref{s18} for studying the topological structure of maximal subgroups of the minimal ideal of the superextension $\lambda(X)$.

Given a twin subset $A$ of a group $X$ consider the topology $\tau_A$ on $X$ generated by the subbase consisting of the right shifts $Ax$, $x\in X$. In the following proposition by the weight of a topological space we understand the smallest cardinality of a subbase of its topology.  

\begin{proposition}\label{p10.1}
\begin{enumerate}
\item[\textup{(1)}] The topology $\tau_A$ turns $X$ into a right-topological group. 
\item[\textup{(2)}] If $Ax=xA$ for all $x\in \Fix^-(A)$, then the topology $\tau_A$ is zero-dimensional.
\item[\textup{(3)}] The topology $\tau_A$ is $T_1$ if and only if the intersection $\bigcap_{x\in A}Ax^{-1}$ is a singleton.
\item[\textup{(4)}] The weight of the space $(X,\tau_A)$ does not exceed the index of the subgroup $\Fix(A^{-1})$ in $X$. 
\end{enumerate}
\end{proposition}

\begin{proof} 1. It is clear that the topology $\tau_A$ is right-invariant. 

2. If $Ax=xA$ for all $x\in \Fix^-(A)$, then the set $X\setminus A$ is open in the topology $\tau_A$ because $X\setminus A=xA=Ax$ for any $x\in\Fix^-(A)$. Consequently, $A$ is an open-and-closed subbasic set. Now we see that the space $(X,\tau_A)$ has a base consisting of open-and-closed subsets, which means that it is zero-dimensional. 

3. If the topology $\tau_A$ is $T_1$, then the intersection $\bigcap_{a\in A}Aa^{-1}$ of all open neighborhoods of the neutral element $e$ of $X$ consists of a single point $e$.
Assuming conversely that $\bigcap_{a\in A}Aa^{-1}$ is a singleton $\{e\}$, for any two distinct points $x,y\in X$ we can find a shift $Aa^{-1}$, $a\in A$, that contains the neutral element $e$ but not $yx^{-1}$. Then the shift $Aa^{-1}x$ is an open subset of $(X,\tau)$ that contains $x$ but not $y$, witnessing that the space $(X,\tau_A)$ is $T_1$.

4. To estimate the weight of the space $(X,\tau_A)$, choose a subset $S\subset X$ meeting each coset $x\Fix(A^{-1})$, $x\in X$, at a single point (here $\Fix(A^{-1})=\{x\in X:xA^{-1}=A^{-1}\}$). Then the set $S^{-1}$ meets each coset $\Fix(A)x$, $x\in X$, at a single point. It is easy to see that the family $\{Ax:x\in S^{-1}\}$ forms a subbase of the topology of $\tau$ and hence the weight of $(X,\tau)$ does not exceed $|X/\Fix(A^{-1})|$.
\end{proof}

\begin{definition} A topology $\tau$ on a group $X$ will be called {\em twin-generated} if $\tau$ is equal to the topology $\tau_A$ generated by some twin subset $A\subset X$, i.e., $\tau$ is generated by the subbase $\{Ax:x\in X\}$.
\end{definition}

Because of Theorem~\ref{t8.2}, we shall be especially interested in 
twin-generated topologies on the quasi-cyclic group $C_{2^\infty}$
 and the infinite quaternion group $Q_{2^\infty}$. 
First we consider some examples. 

\begin{example}\label{e9.3}In the circle $\IT=\{z\in\IC:|z|=1\}$ consider the twin subset $C_{\!\curvearrowleft}=\{e^{i\varphi}:0\le\varphi<\pi\}$.
\begin{enumerate} 
\item[\textup{(1)}] For each $z\in\IT\setminus C_{2^\infty}$ the twin set $C_{2^\infty}\cap zC_{\!\curvearrowleft}$ generates the Euclidean topology  on $C_{2^\infty}$.
\item[\textup{(2)}] For each $z\in C_{2^\infty}$ the twin set $C_{2^\infty}\cap zC_{\!\curvearrowleft}$ generates the Sorgenfrey topology on $C_{2^\infty}$. This topology turns $C_{2^\infty}$ into a paratopological group with discontinuous inversion.
\end{enumerate}
\end{example}

A similar situation holds for the group $Q_{2^\infty}$. Its closure in the algebra of quaternions $\IH$ coincides with the multiplicative subgroup $\IT\cup\IT\mathbf j$ of $\IH$, where $\mathbf j\in Q_8\setminus \IC$ is one of non-complex quaternion units. 

\begin{example}\label{e9.4}In the group $\IT\cup\IT\mathbf j\subset\IH$ consider the twin subset $Q_{\!\curvearrowleft}=C_{\!\curvearrowleft}\cup C_{\!\curvearrowleft}\,\mathbf j$. 
\begin{enumerate}
\item[\textup{(1)}] For each $z\in\IT\setminus C_{2^\infty}$ the twin set $Q_{2^\infty}\cap zQ_{\!\curvearrowleft}$ generates the Euclidean topology  on $Q_{2^\infty}$.
\item[\textup{(2)}] For each $z\in C_{2^\infty}$ the twin set $Q_{2^\infty}\cap z Q_{\!\curvearrowleft}$ generates the Sorgenfrey topology on $C_{2^\infty}$. This topology turns $Q_{2^\infty}$ into a right-topological group with discontinuous inverse and discontinuous left shifts $l_x:Q_{2^\infty}\to Q_{2^\infty}$ for $x\in Q_{2^\infty}\setminus C_{2^\infty}$.
\end{enumerate}
\end{example}

In the following proposition by $\tau_E$ we denote the Euclidean topology on $C_{2^\infty}$.

\begin{theorem}\label{t9.5} Each metrizable right-invariant topology $\tau\supset\tau_E$ on the group $C_{2^\infty}$ (or $Q_{2^\infty}$) is twin-generated.
\end{theorem}

\begin{proof} First we consider the case of the group $C_{2^\infty}$. Let $E_0=C_{2^\infty}\cap \{e^{i\varphi}:-\pi/3<\varphi<2\pi/3\}$ be the twin subset generating the Euclidean topology $\tau_E$ on $C_{2^\infty}$ and $E_n=C_{2^\infty}\cap\{e^{i\varphi}:|\varphi|<3^{-n-1}\pi\}$ for $n\ge 1$. 
For every $n\in\IN$ let $\varphi_n=\sum_{k=1}^n\pi/4^n$ and observe that
$\varphi_\infty=\sum_{k=1}^\infty\pi/4^n=\pi/3$.

Let $\tau\supset\tau_E$ be any metrizable right-invariant topology on $C_{2^\infty}$. The metrizable space $(C_{2^\infty},\tau)$ is countable and hence zero-dimensional. Since $\tau\supset\tau_E$, there exists a neighborhood base $\{U_n\}_{n=1}^\infty\subset\tau$ at the unit 1 such that each set $U_n$ is closed and open in $\tau$ and $U_n\subset E_n$ for all $n\in\IN$.

The interested reader can check that the twin subset
$$A=(E_0\setminus\bigcup_{n=1}^\infty e^{i\varphi_n}E_n)\cup \bigcup_{n=1}^\infty e^{i\varphi_n}U_n\cup \bigcup_{n=1}^\infty e^{i(\pi+\varphi_n)}E_n\setminus U_n$$generates the topology $\tau$.
\smallskip

Next, assume that $\tau\supset \tau_E$ is a right-invariant topology on the group $Q_{2^\infty}$. This group can be written as $Q_{2^\infty}=C_{2^\infty}\cup C_{2^\infty}\mathbf j$, where $\mathsf j\in Q_8\setminus\IC$ is a non-complex quaternion unit. Since $C_{2^\infty}\in\tau_E\subset\tau$, the subgroup $C_{2^\infty}$ is open in $Q_{2^\infty}$. By the preceding item, the topology $\tau\cap \mathsf P(C_{2^\infty})$ on the group $C_{2^\infty}$ is generated by a twin set $A\subset C_{2^\infty}\setminus\{e^{i\varphi}:\varphi\in(\frac{2\pi}3,\pi)\cup(\frac{4\pi}3,\frac{5\pi}3)\}$. A simple geometric argument shows that the topology $\tau$ is generated by the twin subset $A\cup A\mathbf j$ of $Q_{2^\infty}$.  
\end{proof} 

\begin{problem} Are all metrizable right-invariant topology on 
$C_{2^\infty}$ and $Q_{2^\infty}$ twin-generated?
\end{problem}

\section{The characteristic group $\HH(A)$ of a twin subset $A$}

In this section, given a twin subset $A\in\Tau$ of a group $X$ we introduce a twin-generated topology  on the characteristic group $\HH(K)$ of the 2-cogroup $K=\Fix^-(A)$.

Consider the intersection $B=A\cap \Stab(K)=B\cdot KK$ and the image $A'=q_A(B)$ of the set $B$ under the quotient homomorphism $q_A:\Stab(K)\to \HH(K)=\Stab(K)/KK$. We claim that $A'$ is a twin subset of $\HH(K)$.

Indeed, for every $x\in\Fix^-(A)=K\subset\Stab(K)$ we get
$X\setminus A=xA$ and consequently, $\Stab(K)\setminus B=xB$ and $\HH(A)\setminus A'=zA'$ where $z\in q_A(x)$. 

Now it is legal to endow the group $\HH(K)$ with the topology $\tau_{A'}$ generated by the twin subset $A'$. This topology is generated by the subbase $\{A'x:x\in \HH(K)\}$. By Proposition~\ref{p10.1} the topology $\tau_{A'}$ turns the characteristic group $\HH(K)$ into a right-topological group, which will be called the {\em characteristic group} of $A$ and will be denoted by $\HH(A)$. By Proposition~\ref{p10.1}, the characteristic group $\HH(A)$ is a $T_1$-space and its weight does not exceed the cardinality of $\HH(A)$.

The reader should be conscious of the fact that for two twin subsets $A,B\in\Tau$ with $\Fix^-(A)=\Fix^-(B)$ the characteristic group $\HH(A)$ and $\HH(B)$ are algebraically isomorphic but topologically they can be distinct,  see Examples~\ref{e9.3} and \ref{e9.4}.

\section{Characterizing functions that belong to $\Enl^\I(\mathsf F)$}

In this section for a twinic ideal $\I$ on a group $X$ and a left-invariant subfamily $\mathsf F\subset\wht{\Tau}$ we characterize functions $f:\mathsf F\to\mathsf P(X)$ that belong to the space $\Enl^\I(\mathsf F)$. We recall that $\Enl^\I(\mathsf F)$ is the projection of $\Enl^\I(\mathsf P(X))$ onto the face $\mathsf P(X)^{\mathsf F}$. 

\begin{theorem}\label{t11.1} For a left-invariant twinic ideal $\I$ on a group $X$ and a left-invariant subfamily $\mathsf F\subset\Tau$ a function $\varphi:\mathsf F\to\mathsf P(X)$ belongs to the space $\Enl^\I(\mathsf F)$ if and only if $\varphi$ is equivariant, $\I$-saturated, and $\Fix^-(A)\subset \Fix^-(\varphi(A))$ for all $A\in\mathsf F$.
\end{theorem}

\begin{proof} To prove the ``only if'' part, take any function $\varphi\in\Enl^\I(\mathsf F)$ and find a function $\psi\in\Enl^\I(\mathsf P(X))$ such that $\varphi=\psi|\mathsf F$. By Theorem~\ref{t4.8}, the function $\psi$ is equivariant, monotone, symmetric, and $\I$-saturated. Consequently, its restriction $\varphi=\psi|\mathsf F$ is equivariant and $\I$-saturated. Now fix any subset $A\in\mathsf F$ and take any point $x\in\Fix^-(A)$. The left-invariance of $\mathsf F$ guarantees that $X\setminus A=xA\in\mathsf F$, which means that the family $\mathsf F$ is symmetric.

 Applying the equivariant symmetric function $\psi$ to the equality $xA= X\setminus A$, we get
$$x\,\varphi(A)=x\,\psi(A)=\psi(xA)=\psi(X\setminus A)=X\setminus\psi(A)=X\setminus\varphi(A)$$ and thus $x\in\Fix^-(\varphi(A))$ and $\Fix^-(A)\subset\Fix^-(\varphi(A))$. 
\smallskip

To prove the ``if'' part, fix any equivariant $\I$-saturated function $\varphi:\mathsf F\to\mathsf P(X)$ such that $\Fix^-(A)\subset\Fix^-(\varphi(A))$. 
In order to apply Theorem~\ref{t4.8}, we need to extend the function $\varphi$ to some symmetric $\I$-saturated family. This can be done as follows.

Consider the $\I$-saturization 
$\bar{\bar{\mathsf F}}^\I=\bigcup_{A\in\mathsf F}\bar{\bar A}^\I$ of $\mathsf F$. Next, extend the function $\varphi$ to the function $\bar\varphi:\bar{\bar{\mathsf F}}\to \mathsf F$ assigning to each set $B\in\bar{\bar{\mathsf F}}^\I$ the set $\varphi(A)$ where $A\in\mathsf F\cap\bar{\bar B}^\I$. Since $\varphi$ is $\I$-saturated, so defined extension $\bar\varphi$ of $\varphi$ is well-defined and $\I$-saturated. The equivariance of $\varphi$ implies the equivariance of its extension $\bar\varphi$.

Let us check that the function $\bar\varphi:\bar{\bar{\mathsf F}}^\I\to\mathsf F$ is symmetric and monotone.

To see that $\bar\varphi$ is symmetric, take any set $B\in\bar{\bar{\mathsf F}}^\I$ and find a set $A\in\mathsf F\cap\bar{\bar B}^\I$. Fix any point $x\in \Fix^-(A)$. By our hypothesis $x\in\Fix^-(A)\subset\Fix^-(\varphi(A))$.  It follows from $A=_\I B$ that $X\setminus A=_\I X\setminus B$ and hence
$$\bar\varphi(X\setminus B)=\varphi(X\setminus A)=\varphi(xA)=x\varphi(A)=X\setminus \varphi(A)=X\setminus\bar\varphi(B),$$
which means that the function $\bar\varphi$ is symmetric.

The monotonicity of $\bar\varphi$ will follow as soon as we check that $\varphi(A)=\varphi(B)$ for any sets $A,B\in\mathsf F$ with $A\subset_\I B$.
Pick points $a\in\Fix^-(A)$, $b\in\Fix^-(B)$. Since the ideal $\I$ is twinic, the chain of $\I$-inclusions $bB=X\setminus B\subset_\I X\setminus A=aA\subset_\I aB$ implies the chain of $\I$-equalities $bB=_\I X\setminus B=_\I X\setminus A=aA=_\I aB$, which yields $A=_\I B$ and $\varphi(A)=\varphi(B)$ as $\varphi$ is $\I$-saturated.

Therefore $\bar\varphi:\bar{\bar{\mathsf F}}^\I\to\mathsf F$ is a left-invariant symmetric monotone $\I$-saturated function defined on a $\I$-saturated  left-invariant symmetric family $\bar{\bar{\mathsf F}}^\I$. By Theorem~\ref{t4.8}, $\bar\varphi$ belongs to $\Enl^\I(\bar{\bar{\mathsf F}}^\I)$ and then its restriction $\varphi=\bar\varphi|\mathsf F$ belongs to $\Enl^\I(\mathsf F)$. 
\end{proof}

Let us recall that $\wht\Tau=\{A\subset X:\Fix^-(A)\in\wht\K\}$. 

\begin{corollary} For a left-invariant twinic ideal $\I$ on a group $X$  the space $\Enl^\I(\wht{\Tau})$ consists of all equivariant $\I$-saturated functions $\varphi:\wht{\Tau}\to\wht{\Tau}$ such that $\Fix^-(\varphi(A))=\Fix^-(A)$ for all $A\in\wht{\Tau}$.
\end{corollary}

A similar characterization holds for functions that belong to the space $\Enl^\I(\Tau_K)$ for $K\in\wht{\K}$ (let us observe that Theorem~\ref{t4.8} is not applicable to the family $\Tau_K$ because it is not left-invariant). A function $\varphi:\Tau_K\to\Tau_K$ is {\em $\Stab(K)$-equivariant} if $\varphi(xA)=x\varphi(A)$ for all $A\in\Tau_K$ and $x\in \Stab(K)$.

\begin{proposition}\label{p11.3} For any maximal 2-cogroup $K\subset X$ and a left-invariant twinic ideal $\I$ on a group $X$ a function $\varphi:\Tau_K\to\Tau_K$ belongs to the space $\Enl^\I(\Tau_K)$ if and only if $\varphi$ is $\Stab(K)$-equivariant and $\I$-saturated.
\end{proposition}

\begin{proof} The ``only if'' part follows from Theorem~\ref{t4.8}.
To prove the ``if part'', assume that a function $\varphi:\Tau_K\to\Tau_K$ is $\Stab(K)$-invariant and $\I$-saturated. For any $A\in\Tau_K$ and $x\in K=\Fix^-(A)=\Fix^-(\varphi(A))$ we get $\varphi(X\setminus A)=\varphi(xA)=x\varphi(A)=X\setminus \varphi(A)$, which means that the function  $\varphi$ is symmetric. 

Now consider the
families $$\LL_\varphi=\{x^{-1}A:A\in\mathsf F,\;x\in \varphi(A)\}\;\;\mbox{and}\;\; 
\bar{\bar \LL}^\I_\varphi=\bigcup_{A\in\LL_\varphi}\bar{\bar A}^\I.$$

 We claim that the family $\bar{\bar\LL}_\varphi^\I$ is linked. Assuming the
converse, we could find two sets $A,B\in\mathsf F$ and two points $x\in
\varphi(A)$ and $y\in \varphi(B)$ such that $x^{-1}A\cap y^{-1}B\in\I$. Then
$yx^{-1}A\subset_{\I} X\setminus B$. Let us show that the point $c=yx^{-1}$ belongs to the subgroup $\Stab(K)$ of $X$. Given any point $z\in K$, we need to prove that $c^{-1}zc\in K$. Taking into account that $z\in K=\Fix^-(B)=\Fix^-(A)$, we see that $cA\subset_\I X\setminus B$ implies that $$cA\subset_\I X\setminus B=zB\subset_\I zc(X\setminus A)=zczA.$$ Since the ideal $\I$ is twinic, we get $cA=_\I X\setminus B=_\I zczA$, which implies $c^{-1}zcz\in\IFix(A)$. The maximality of the 2-cogroup $K=\Fix^-(A)\subset \IFix^-(A)$ guarantees that $\IFix^-(A)=K$ and $\IFix(A)=\IFix^-(A)\cdot\IFix^-(A)=KK$. Therefore $c^{-1}zcz\in KK$ and $c^{-1}zc\in KKz^{-1}=K$. Now we see that $yx^{-1}=c\in\Stab(K)$. So it is legal to apply the $\Stab(K)$-invariant $\I$-saturated function $\varphi$ to the $\I$-equality $yx^{-1}A=_\I X\setminus B$ and obtain $yx^{-1}\varphi(A)=\varphi(X\setminus B)=X\setminus\varphi(B)$.
Then $x^{-1}\varphi(A)\subset X\setminus y^{-1}\varphi(B)$, which is
not possible because the neutral element $e$ of the group $X$
belongs to $x^{-1}\varphi(A)\cap y^{-1}\varphi(B)$.
Further we continue as in the proof of Theorem~\ref{t4.8}. 
\end{proof}

\section{The $\HH(K)$-act $\Tau_K$ of a maximal 2-cogroup $K$}

In this section, given a maximal 2-cogroup $K$ in a group $X$ we study the structure of the subspace
$$\Tau_K=\{A\in\mathsf P(X):\Fix^-(A)=K\}\subset\mathsf P(X)$$of the compact Hausdorff space $\mathsf P(X)$. The latter space is naturally homeomorphic to the Cantor discontinuum $2^X$ where the ordinal $2=\{0,1\}$ is endowed with the discrete topology.  

\begin{proposition}\label{p12.1} For any 2-cogroup $K\subset X$ the subspace $\Tau_K$ of\/ $\mathsf P(X)$ is homeomorphic to the Cantor discontinuum $2^{X/K^\pm}$ where $X/K^\pm=\{K^\pm x:x\in X\}$.
\end{proposition}

\begin{proof} Choose any subset $S\subset X$ that meets each coset $K^\pm x$, $x\in X$, at a single point, and consider the bijective function
$\Psi:\mathsf P(S)\to \Tau_K$ assigning to each subset $A\subset S$ the twin set $T_A=KKA\cup K(S\setminus A)$. Let us show that the function $\Psi$ is continuous. The subbase of the topology of $\Tau_K$ consists of the sets $\la x\ra^+=\{B\in \Tau_K:x\in B\}$ and $\la x\ra^-=\{B\in\Tau_K:x\notin B\}$ where $x\in X$. Observe that for every 
$z\in K$ we get $\la x\ra^-=\{B\in\Tau_K:x\in X\setminus B= zB\}=\la z^{-1}x\ra^+,$ which means that the sets $\la x\ra^+$, $x\in X$, form a subbase of the topology of $\Tau_K$.

Now the continuity of the map $\Psi$ will follow as soon as we check that for every $x\in X$ the set $\Psi^{-1}(\la x\ra^+)=\{A\in\mathsf P(S):x\in T_A\}$ is open in $\mathsf P(S)$. Fix any subset $A\in \Psi^{-1}(\la x\ra^+)$ and let $s$ be the unique point of the intersection $S\cap K^\pm x$.
Consider the open neighborhood $O(A)=\{A'\in\mathsf P(S):A'\cap\{s\}=A\cap\{s\}\}$ of $A$ in the space $\mathsf P(S)$. We claim that $O(A)\subset\Psi^{-1}(\la x\ra^+)$. Fix any $A'\in O(A)$ and consider two cases:

(i) If $s\in A$, then $s\in A'$ and $x\in T_A\cap K^\pm s=KKs\subset T_{A'}$.

(ii) If $s\in S\setminus A$, then $s\in S\setminus A'$ and $x\in T_A\cap K^\pm s=Ks\subset K(S\setminus A')\subset T_{A'}$.

 In both cases $\Psi(A')=T_{A'}\in\la x\ra^+$.
Now we see that $\Psi:\mathsf P(S)\to \Tau_K$, being a continuous bijective map defined on the compact Hausdorff space $\mathsf P(S)$, is a homeomorphism. It remains to observe that $\mathsf P(S)$ is  homeomorphic to $2^{X/K^\pm}$. 
\end{proof}

Let us observe that in general the subfamily $\Tau_K\subset\mathsf P(X)$ is not left-invariant. Indeed, for any $A\in\Tau_K$ and $x\in X$ the shift $xA$ belongs to $\Tau_K$ if and only if $K=\Fix^-(xA)=x\Fix^-(A)x^{-1}=xKx^{-1}$ if and only if $x\in\Stab(K)$. Thus the family $\Tau_K$ can be considered as an act endowed with the left action of the group $\Stab(K)$. 

For any twin set $A\in\Tau_K$ its stabilizer $\Fix(A)=\{x\in X:xA=A\}$ is equal to $\Fix^-(A)\cdot\Fix^-(A)=KK$ and hence is a normal subgroup of $\Stab(K)$. This implies that the characteristic group $\HH(K)=\Stab(K)/KK$ acts freely on the space $\Tau_K$. Therefore, we can (and will) consider the space $\Tau_K$ as a free $\HH(K)$-act. For each set $A\in\Tau_K$ by 
$$\lfloor A\rfloor=[A]\cap \Tau_K=\{xA:x\in \Stab(K)\}=\{hA:h\in \HH(K)\}$$ we denote the orbit of $A$ in $\Tau_K$ and by $[\Tau_K]=\{\lfloor A\rfloor:A\in\Tau_K\}$ the orbit space of the $\HH(K)$-act $\Tau_K$, endowed with the quotient topology. By  
Theorem~\ref{t2.1}, the $\HH(K)$-act $\Tau_K$ is isomorphic to $[\Tau_K]\times \HH(K)$. In some cases the isomorphism between the $\HH(K)$-acts $\Tau_K$ and $[\Tau_K]\times \HH(K)$ is topological.

\begin{proposition}\label{p12.2} The orbit space $[\Tau_K]$ is a $T_1$-space if and only if the characteristic group $\HH(K)$ is finite. In this case $[\Tau_K]$ is a compact Hausdorff space and the orbit map $q:\Tau_K\to[\Tau_K]$ has a continuous section $s:[\Tau_K]\to\Tau_K$, which implies that $\Tau_K$ is homeomorphic to the product $[\Tau_K]\times \HH(K)$ where the (finite) group $\HH(K)$ is endowed with the discrete topology.
\end{proposition}

\begin{proof} By Theorem~\ref{t8.2}, the characteristic group $\HH(K)$ is at most countable. Since $T_K$ is a free $\HH(K)$-act, each orbit $\lfloor A\rfloor$, $A\in\Tau_K$, has cardinality $|\lfloor A\rfloor|=|\HH(K)|$ and hence is at most countable. Note that the orbit $\lfloor A\rfloor$ admits a transitive action of the group $\HH(K)$ and hence is topologically homogeneous.

If $[\Tau_K]$ is a $T_1$-space, then each orbit $\lfloor A\rfloor$, $A\in\Tau_K$, is closed in the compact Hausdorff space $\Tau_K$. Now Baire theorem implies that $\lfloor A\rfloor$ has an isolated point and is discrete (being topologically homogeneous). Taking into account that $\lfloor A\rfloor$ is compact and discrete, we conclude that it is finite. Consequently $|\HH(K)|=|\lfloor A\rfloor|<\aleph_0$.

Now assume that the characteristic group $\HH(K)$ is finite. Let $q:\Tau_K\to[\Tau_K]$ denote the orbit map. To show that the orbit space $[\Tau_K]$ is Hausdorff, pick two distinct orbits $\lfloor A\rfloor$ and $\lfloor B\rfloor$. Since $\HH(K)$ is finite and $xA\ne yB$ for any $x,y\in \HH(K)$, we can find two neighborhoods $O(A)$ and $O(B)$ of $A,B$ in $\Tau_K$ such that $xO(A)\cap yO(B)=\emptyset$. Then $O(\lfloor A\rfloor)=\bigcup_{x\in \HH(K)}xO(A)$ and  $O(\lfloor B\rfloor)=\bigcup_{y\in \HH(K)}yO(B)$ are two disjoint open $\HH(K)$-invariant subsets in $\Tau_K$. Their images $q(O(\lfloor A\rfloor))$ and $q(O(\lfloor B\rfloor))$ are disjoint open neighborhoods of $\lfloor A\rfloor$, $\lfloor B\rfloor$ in $[\Tau_K]$, which means that the orbit space $[\Tau_K]$ is Hausdorff.
This space is compact and zero-dimensional as the image of the compact zero-dimensional space $\Tau_K$ under the open continuous map $q:\Tau_K\to[\Tau_K]$. 

Using the zero-dimensionality of $[\Tau_K]$ and the finiteness of $\HH(K)$ it is easy to construct a continuous section $s:[\Tau_K]\to\Tau_K$ of the map $q$ and prove that $\Tau_K$ is homeomorphic to $\HH(K)\times[\Tau_K]$.
\end{proof}

Let us recall that a subfamily $\mathsf F\subset\mathsf P(X)$ is $\lambda$-invariant if $f(\mathsf F)\subset\mathsf F$ for any equivariant symmetric monotone function $f:\mathsf P(X)\to\mathsf P(X)$. For a $\lambda$-invariant subfamily $\mathsf F\subset \mathsf P(X)$ the projection
$$\Enl(\mathsf F)=\{f|\mathsf F:f\in\Enl(\mathsf P(X))\}$$ is a subsemigroup of the semigroup $\mathsf F^{\mathsf F}$ of all self-mappings of $\mathsf F$. 

\begin{proposition}\label{p12.3} For any maximal 2-cogroup $K\subset X$ the family $\Tau_K$ is $\lambda$-invariant and hence $\Enl(\Tau_K)$ is a compact right-topological semigroup.
\end{proposition}

\begin{proof} Given any function $f\in\Enl(\mathsf P(X))$ and a set $A\in\Tau_K$ we need to show that $f(A)\in\Tau_K$. By Corollary~\ref{c4.4}, the function $f$ is equivariant and symmetric. Then for any $x\in K=\Fix^-(A)$ we get $xA=X\setminus A$ and hence $x\varphi(A)=\varphi(xA)=\varphi(X\setminus A)=X\setminus\varphi(A)$, which means that $x\in\Fix^-(\varphi(A))$ and $K\subset\Fix^-(\varphi(A))$. The maximality of the 2-cogroup $K$ guarantees that $K=\Fix^-(\varphi(A))$ and thus $\varphi(A)\in\Tau_K$. So, the family $\Tau_K$ is $\lambda$-invariant.
\end{proof}

\section{$\I$-incomparable and $\I$-independent families}\label{s12n}

Let $\I$ be a left-invariant ideal on a group $X$. A family $\mathsf F\subset\mathsf P(X)$ is called 
\begin{itemize}
\item  {\em $\I$-incomparable} if $\forall A,B\in\mathsf F$ \ $(A\subset_\I B\;\Ra\;A=_\I B)$;
\item  {\em $\I$-independent} if $\forall A,B\in\mathsf F$ \ $(A=_\I B\;\Ra\;A= B)$.
\end{itemize}

\begin{proposition}\label{p13.1} A left-invariant ideal $\I$ on a group $X$ is twinic if and only if the family $\pT^\I$ of $\I$-pretwin sets is $\I$-incomparable.
\end{proposition}

\begin{proof} First assume that the family $\pT^\I$ is $\I$-incomparable. 
To show that the ideal $\I$ is twinic, take any subset $A\subset X$ with $xA\subset_\I X\setminus A\subset_\I yA$ for some $x,y\in X$. Then $A\in\pT^\I$ and also $xA,yA\in\pT^\I$. Since $xA\subset_\I yA$, the $\I$-incomparability of the family $\pT^\I$ implies that $xA=_\I yA$ and then $xA=_\I X\setminus A=_\I yA$, which means that the ideal $\I$ is twinic.

Now assume conversely that $\I$ is twinic and take two $\I$-pretwin sets $A\subset_\I B$.  Since the sets $A,B$ are $\I$-pretwin, there are elements
$x,y\in X$ such that $xB\subset_\I X\setminus B$ and $X\setminus
A\subset_\I yA$. Taking into account that
$$xB\subset_\I X\setminus B\subset_\I X\setminus A\subset_\I yA\subset_\I yB,$$
and $\I$ is twinic, we conclude that $X\setminus B=_\I X\setminus
A$ and hence $A=_\I B$.
\end{proof}

\begin{corollary} For each twinic left-invariant ideal $\I$ on a group $X$ the family ${\Tau}$ of twin sets is $\I$-incomparable.
\end{corollary}

\begin{proposition}\label{p13.3} For a left-invariant ideal $\I$ on a group $X$ the family $\wht{\Tau}$ is $\I$-independent if and only if $\I\cap\wht{\K}=\emptyset$.
\end{proposition}

\begin{proof} To prove the ``only if'' part, assume that the ideal $\I$ contains some maximal 2-cogroup $K\in\wht{\K}$. Since $\I$ is left-invariant, for each $x\in K$, $KK=xK\in\I$ and hence $K^\pm=K\cup KK\in\I$.

 Choose a subset $S\subset X$ that contains the neutral element $e$ of the group $X$ and meets each coset $K^\pm x$, $x\in X$, at a single point. Then $A=KKS$ and $B=KK(S\setminus\{e\})\cup K$ are two distinct twin sets with $K\subset\Fix^-(A)\cap\Fix^-(B)$. By the maximality of $K$, $K=\Fix^-(A)=\Fix^-(B)$ and hence $A,B\in\wht{\Tau}$. Since the symmetric difference $A\triangle B=KK\cup K=K^\pm\in\I$, we get $A=_\I B$, which means that the family $\wht{\Tau}$ fails to be $\I$-independent.
\smallskip

To prove the ``if'' part, assume that the family $\wht{\Tau}$ is not $\I$-independent and find two subsets $A,B\in\wht{\Tau}$ such that $A\ne B$ but $A=_\I B$. The 2-cogroup $\Fix^-(A)$ of $A$ is maximal and hence coincides with the 2-cogroup $\IFix^-(A)\supset\Fix^-(A)$. By the same reason, $\Fix^-(B)=\IFix^-(B)$.
The $\I$-equality $A=_\I B$ implies $\IFix^-(A)=\IFix^-(B)$. Denote the maximal 2-cogroup $\Fix^-(A)=\IFix^-(A)=\IFix^-(B)=\Fix^-(B)$ by $K$.
Then $\Fix(A)=\Fix^-(A)\cdot\Fix^-(A)=KK=\Fix(B)$ and hence $A=KKA$ and $B=KKB$. Now we see that the symmetric difference $A\triangle B=KKA\triangle KKB$ contains a subset $KKx$ for some $x\in X$. Then for any $y\in K$, we get $Kyx=KKx\subset A\triangle B\in\I$ and hence $Kyx\in\I$. Finally observe that the set $K'=x^{-1}y^{-1}Kyx$ is a maximal 2-cogroup and by the left invariance of the ideal $\I$, $K'=x^{-1}y^{-1}Kyx\in\I$. So, $\wht{\K}\cap\I\ne\emptyset$.
\end{proof} 

\begin{proposition}\label{p13.4} A subfamily $\mathsf F\subset\mathsf P(X)$ is $\I$-independent for any left-invariant ideal $\I$ on $X$ if for each set $A\in \mathsf F$ the subgroup $\Fix(A)$ has finite index in $X$.
\end{proposition}

\begin{proof} Assume that $A,B\in\mathsf F$ be two subsets with $A=_\I B$ for some left-invariant ideal $\I$. Since the subgroups $\Fix(A)$ and $\Fix(B)$ have finite indices in $X$, their intersection $\Fix(A)\cap\Fix(B)$ also has finite index in $X$ and contains a normal subgroup $H\subset X$ of finite index in $X$, see \cite[I.Ex.9(a)]{Lang}.
Then $X=FH$ for some finite subset $F\subset X$.
Assuming that $A\ne B$, we can find a point $x\in A\triangle B$ and conclude that $xH=Hx\subset HA\triangle HB=A\triangle B\in\I$ and $X=FH\in\I$ by the left-invariance of the ideal $\I$. This contradiction completes the proof.
\end{proof}

\begin{proposition}\label{p13.5} Each minimal $\wht\K$-covering subfamily $\wtd{\Tau}\subset\wht{\Tau}$ is $\I$-independent.
\end{proposition}

\begin{proof} Fix any two sets $A,B\in\wtd{\Tau}$ with $A=_\I B$. Repeating the argument from the proof of Proposition~\ref{p13.3}, we can prove that $\IFix^-(A)=\Fix^-(A)=\IFix^-(B)=\Fix^-(B)=K$ for some maximal 2-cogroup $K\in\wht{\K}$. Since the family $\wtd{\Tau}\ni A,B$ is minimal $\wht{\K}$-covering, the sets $A,B$ lie in the same orbit and hence $A=xB$ for some $x\in X$.  It follows from $B=_\I A=xB$ that $x\in\IFix(B)=\Fix(B)$ and thus $A=xB=B$.
\end{proof}

\section{The endomorphism monoid $\End(\Tau_K)$ of the $\HH(K)$-act $\Tau_K$}

For any maximal 2-cogroup $K$ in a group $X$ the compact right-topological semigroup $\Enl(\Tau_K)$ is a subsemigroup of the endomorphism monoid $\End(\Tau_K)$ of the free $\HH(K)$-act $\Tau_K$. The endomorphism monoid $\End(\Tau_K)$ is the space of all (not necessarily continuous) functions $f:\Tau_K\to\Tau_K$ that are equivariant in the sense that $f(xA)=xf(A)$ for all $A\in\Tau_K$ and $x\in \Stab(K)$. It is easy to check that $\End(\Tau_K)$ is a closed subsemigroup of the compact Hausdorff right-topological semigroup $\Tau_K^{\;\Tau_K}$ of all self-maps of the compact Hausdorff space $\Tau_K$. So, $\End(\Tau_K)$ is a compact Hausdorff right-topological semigroup that contains $\Enl(\Tau_K)$ as a closed subsemigroup.

If $\I$ is a left-invariant ideal on the group $X$, then the left ideal $\Enl^\I(\Tau_K)$ of $\Enl(\Tau_K)$ lies in the left ideal $\End^\I(\Tau_K)\subset \End(\Tau_K)$ consisting of all equivariant functions $f:\Tau_K\to\Tau_K$, which are $\I$-saturated in the sense that $f(A)=f(B)$ for all $A,B\in\Tau_K$ with $A=_\I B$.  

In the following theorem we describe some algebraic and topological properties of the endomorphism monoid $\End(\Tau_K)$.

\begin{theorem}\label{t14.1} Let $K$ be a maximal 2-cogroup in a group $X$.
Then: 
\begin{enumerate}
\item[\textup{(1)}] $\End^\I(\Tau_K)=\Enl^\I(\Tau_K)\subset\Enl(\Tau_K)\subset\End(\Tau_K)$ for any twinic ideal $\I$ on $X$;
\item[\textup{(2)}] $\End^\I(\Tau_K)=\End(\Tau_K)$ for any left-invariant ideal $\I$ on $X$ such that $\I\cap\wht\K=\emptyset$; 
\item[\textup{(3)}] the semigroup $\End(\Tau_K)$ is algebraically isomorphic to the wreath product\newline $\HH(K)\wr[\Tau_K]^{[\Tau_K]}$;
\item[\textup{(4)}] for each idempotent $f\in\End(\Tau_K)$ the maximal subgroup $\HH_f\subset\End(\Tau_K)$\newline containing $f$ is isomorphic to $\HH(K)\wr S_{[f(\Tau_K)]}$;
\item[\textup{(5)}] the minimal ideal $\IK(\End(\Tau_K))=\{f\in\End(\Tau_K):\forall A\in f(\Tau_K),\;f(\Tau_K)\subset\lfloor A\rfloor\}$;
\item[\textup{(6)}] each minimal left ideal of the semigroup $\End(\Tau_K)$ is algebraically isomorphic to $\HH(K)\times [\Tau_K]$ where the orbit space $[\Tau_K]$ is endowed with the left zero multiplication;
\item[\textup{(7)}] each maximal subgroup of the minimal ideal $\IK(\End(\Tau_K))$ is algebraically isomorphic to $\HH(K)$;
\item[\textup{(8)}] each minimal left ideal of the semigroup $\End(\Tau_K)$ is homeomorphic to $\Tau_K$;
\item[\textup{(9)}] for each minimal idempotent $f\in \IK(\End(\Tau_K))$ the maximal subgroup\newline $\mathsf H_f=f\circ \End(\Tau_K)\circ f$ is topologically isomorphic to the twin-generated group $\HH(A)$ where $A\in f(\Tau_K)$;
\end{enumerate}
\end{theorem}

\begin{proof} 1,2. The first statement follows from Proposition~\ref{p11.3} and the second one from Proposition~\ref{p13.3}.

3--7. Since $\Tau_K$ is a free $\HH(K)$-act, the (algebraic) statements (3)--(7) follow from Theorem~\ref{t2.1}.

8. Given a minimal idempotent $f\in\End(\Tau_K)$, we need to prove that the minimal left ideal $\mathsf L_f=\End(\Tau_K)\circ f$ is homeomorphic to $\Tau_K\subset\mathsf P(X)$. For this fix any set $B\in f(\Tau_K)$ and observe that $f(\Tau_K)\subset \lfloor B\rfloor$ according to the statement (5). We claim that the map $$\Psi:\mathsf L_f\to\Tau_K,\;\;\Psi:g\mapsto g(B),$$is a homeomorphism. The definition of the topology (of pointwise convergence) on $\End(\Tau_K)$ implies that the map $\Psi$ is continuous.  
 
Next, we show that the map $\Psi$ is bijective. To show that $\Psi$ is injective, fix any two distinct functions $g,h\in\mathsf L_f$ and find a set $A\in\Tau_K$ such that $g(A)\ne h(A)$. Since $f(\Tau_K)\subset \lfloor B\rfloor$, there is $x\in X$ such that $f(A)=xB$. Then $$xg(B)=g(xB)=gf(A)=g(A)\ne h(A)=hf(A)=h(xB)=xh(B)$$and hence $\Psi(g)=g(A)\ne h(A)=\Psi(h)$.
To show that $\Psi$ is surjective, take any subset $C\in\Tau_K$ and choose any equivariant map $\varphi:[B]\to[C]$ such that $\varphi(B)=C$. Then the function $g=\varphi\circ f$ belongs to $\mathsf L_f$ and has image  $\Psi(g)=g(B)=C$ witnessing that the map $\Psi$ is surjective.
Since $\mathsf L_f$ is compact, the bijective continuous map $\Psi:\mathsf L_f\to\Tau_K$ is a homeomorphism. By Proposition~\ref{p12.1}, the space $\Tau_K$ is homeomorphic to the cube $2^{X/K^\pm}$.  
\smallskip 

9. Given a minimal idempotent  $f\in \End(\Tau_K)$ we shall show that the maximal subgroup $\mathsf H_f=f\circ \End(\Tau_K)\circ f$ is topologically isomorphic to the characteristic group $\HH(A)$ of any twin set $A\in f(\Tau_K)$.

We recall that  $\HH(A)$  is the characteristic group $\HH(K)$ of the 2-cogroup $K=\Fix^-(A)$, endowed with the topology generated by the twin set $q(A\cap \Stab(K))$ where $q:\Stab(K)\to \HH(K)=\Stab(K)/KK$ is the quotient homomorphism.

We define a topological isomorphism $\Theta_A:\mathsf H_f\to \HH(A)$ in the following way. Since 
$f$ is a minimal idempotent, $g(A)=fgf(A)\in f(\Tau_K)\subset\lfloor A\rfloor$. So we can find $x\in \Stab(K)$ with $fgf(A)=x^{-1}A$. Now define  $\Theta_A(g)$ as the image $q(x)=xKK=KKx$ of $x$ under the quotient homomorphism $q:\Stab(K)\to \HH(K)=\HH(A)$.

It remains to prove that $\Theta_A:\mathsf H_f\to \HH(A)$ is a well-defined topological isomorphism of the right-topological groups.

First we check that $\Theta_A$ is well-defined, that is $\Theta_A(g)=q(x)$ does not depend on the choice of the point $x$. Indeed, for any other point $y\in X$ with $g(A)=y^{-1}A$ we get $x^{-1}A=y^{-1}A$ and thus $yx^{-1}\in\Fix(A)=K\cdot K$ where $K=\Fix^-(A)$. Consequently,
$q(x)=KKx=KKy=q(y)$.

Next, we prove that $\Theta_A$ is a group homomorphism. Given two functions $g,h\in \mathsf H_f$, find elements $x_g,x_h\in \Fix(A)$ such that $h(A)=x^{-1}_hA$ and $g(A)=x^{-1}_gA$. It follows that $g\circ h(A)=g(x^{-1}_hA)=x^{-1}_hg(A)=x^{-1}_hx^{-1}_gA=(x_gx_h)^{-1}A$, which implies that $\Theta_A(g\circ h)=x_gx_hKK=\Theta_A(g)\cdot\Theta_A(h)$.

Now, we calculate the kernel of the homomorphism $\Theta_A$. Take any function $g\in \mathsf H_f$ with $\Theta_A(g)=e$, which means that $g(A)=fgf(A)=A$.
Then for every $A'\in\Tau_{K}$ we can find $x\in X$ with $f(A')=xA$ and conclude that $g(A')=fgf(A')=fg(xA)=xfg(A)=xfgf(A)=xA=f(A')$ witnessing that $g=fgf=f$. This means that the homomorphism $\Theta_A$ is one-to-one.

 To see that $\Theta_A$ is onto, first observe that each element of the characteristic group $\HH(A)$ can be written as $[y]=yKK=KKy\in \HH(K)$ for some $y\in\Stab(K)$. Given such an element $[y]\in \HH(A)$, consider the equivariant function $s_{[y]}:\lfloor A\rfloor\to\lfloor A\rfloor$, $s_{[y]}:zA\mapsto zy^{-1}A=zy^{-1}KKA$. Let us show that this function is well defined. Indeed, for each point $u\in X$ with $zA=uA$, we get $u^{-1}z\in\Fix(A)$ and hence, $yu^{-1}zy^{-1}\in y\Fix(A)y^{-1}=yKKy^{-1}=KK=\Fix(A)$. Then $yu^{-1}zy^{-1}A=A$ and hence $zy^{-1}A=uy^{-1}A$.

It follows from $s_{[y]}\circ f=f\circ s_{[y]}\circ f$ that the function $s_{[y]}\circ f$ belongs to the maximal group $\HH_f$. Since $s_{[y]}\circ f(A)=s_{[y]}(A)=y^{-1}A$, the image $\Theta_A(s_{[y]}\circ f)=[y]$. So, $\Theta_A(\mathsf H_f)=\HH(A)$ and $\Theta_A:\mathsf H_f\to \HH(A)$ is an algebraic isomorphism.

It remains to prove that this isomorphism is topological. Observe that for every $[y]\in \HH(A)$ we get
$s_{[y]}\circ f(A)=s_{[y]}(A)=y^{-1}KKA=y^{-1}A$. Consequently, $x\in s_{[y]}\circ f(A)$ iff $x\in y^{-1}A$ iff $y\in Ax^{-1}$.

To see that the map $\Theta_A:\mathsf H_f\to \HH(A)$ is continuous, take any sub-basic open set $$U_x=\{[y]\in \HH(A):y\in Ax^{-1}\},\;\; x\in \Stab(K),$$ in $\HH(A)$ and observe that $\Theta_A^{-1}(U_x)=\{s_{[y]}\circ f:[y]\in U_x\}=\{s_{[y]}\circ f: y\in Ax^{-1}\}=\{s_{[y]}\circ f:x\in s_{[y]}\circ f(A)\}$ is a sub-basic open set in $H(f)$. To see that the inverse map $\Theta_A^{-1}:\HH(A)\to \mathsf H_f$ is continuous, take any sub-basic open set $V_{x,T}=\{g\in H(f):x\in g(T)\}$ where $x\in X$ and $T\in\Tau_K$. It follows that $f(T)=x_TA$ for some $x_T\in X$. Then $$\begin{aligned}\Theta_A(V_{x,T})&
=\{[y]\in \HH(A):x\in s_{[y]}\circ f(T)\}=\{[y]\in \HH(A):x\in s_{[y]}(x_TA)\}=\\
&=\{[y]\in \HH(A):x_T^{-1}x\in s_{[y]}(A)\}=\{[y]\in \HH(A):y\in Ax^{-1}x_T\}
\end{aligned}$$
is a sub-basic open set in $\HH(A)$.
\end{proof}

In the following proposition we calculate the cardinalities of the objects appearing in Theorem~\ref{t14.1}. We shall say that a cardinal $n\ge 1$ divides a cardinal $m\ge 1$ if there is a cardinal $k$ such that $m=k\times n$. The smallest cardinal $k$ with this property is denoted by $\frac mn$.

\begin{proposition}\label{p15.2} If $K\in\wht{\K}$ is a maximal 2-cogroup in a group $X$, then
\begin{enumerate}
\item[\textup{(1)}] $|\Tau_K|=2^{|X/K^\pm|}$;
\item[\textup{(2)}] $|\HH(K)|\in\{2^k:k\in\IN\}\cup\{\aleph_0\}$ and $|\HH(K)|$ divides the index $|X/K|$ of $K$ in $X$;
\item[\textup{(3)}] $|[\Tau_K]|=\frac{|\Tau_K|}{|\HH(K)|}=\frac{2^{|X/K^\pm|}}{|\HH(K)|}$;
\item[\textup{(4)}] $|\HH(K)|=|X/K|$ if the 2-cogroup $K$ is normal in $X$.
\end{enumerate}
\end{proposition}

\begin{proof} Choose any subset $S\subset X$ that meets each coset $K^\pm x$, $x\in X$, of the group $K^\pm=K\cup KK$ at a single point. It is clear that $|S|=|X/K^\pm|$.
\smallskip

1. The equality $|\Tau_K|=2^{|X/K^\pm|}$ follows from Proposition~\ref{p12.1}.
\smallskip

2. By Theorem~\ref{t8.2}, $|\HH(K)|\in\{2^n:n\in\IN\}\cup\{\aleph_0\}$.  Since $\Stab(K)$ is a subgroup of $X$, $|\HH(K)|=|\Stab(K)/KK|$ divides $|X/KK|=|X/K|$.
\smallskip

3. Since $\Tau_K$ is a free $\HH(K)$-act, 
$|[\Tau_K]|=\frac{|\Tau_K|}{|\HH(K)|}$. This equality is clear if $\HH(K)$ is finite. If $\HH(K)$ is infinite, then $|\HH(K)|=\aleph_0$ and the index $|X/K^\pm|$ of the group $K^\pm$ in $X$ is infinite. In this case 
$|\Tau_K|=2^{|X/K^\pm|}>\aleph_0$ and thus $|[\Tau_K]|=\frac{2^{|X/K^\pm|}}{\aleph_0}=2^{|X/K^\pm|}$.
\smallskip

4. If the 2-cogroup $K$ is normal in $X$, then $\Stab(K)=X$ and $\HH(K)=X/KK$. In this case $|\HH(K)|=|X/KK|=|X/K|$.
\end{proof}

By Theorem~\ref{t14.1}(6), for any maximal 2-cogroup $K\subset X$ each minimal left ideal of the semigroup $\End(\Tau_K)$ is algebraically isomorphic to $\HH(K)\times[\Tau_K]$. It turns out that in some cases this isomorphism is topological. We recall that the orbit space $[\Tau_K]=\Tau_K/\HH(K)$ is endowed with the quotient topology. By Proposition~\ref{p12.2}, the orbit space $[\Tau_K]$ is compact and Hausdorff if and only if  the characteristic group $\HH(K)$ is finite.
  
Since $\Tau_K$ is a compact Hausdorff space, the Tychonoff power $\Tau_K^{\;\Tau_K}$ is a compact Hausdorff right topological semigroup (endowed with the operation of composition of functions). This semigroup contains the subsemigroup $C(\Tau_K,\Tau_K)$ consisting of all continuous maps $f:\Tau_K\to\Tau_K$. It is easy to check that the semigroup $C(\Tau_K,\Tau_K)$ is semitopological (which means that the semigroup operation is separately continuous).

We recall that a right-topological semigroup $S$ is called {\em semitopological} if the semigroup operation $S\times S\to S$ is separately continuous. If the semigroup operation is continuous, then $S$ is called a {\em topological semigroup}.

\begin{theorem}\label{t14.3} Let $K$ be a maximal 2-cogroup in a group $X$. For a minimal idempotent $f$ in the semigroup $\End(\Tau_K)$ and its minimal left ideal $\mathsf L_f=\End(\Tau_K)\circ f$ the following conditions are equivalent:
\begin{enumerate}
\item[\textup{(1)}] $\mathsf L_f$ is a topological semigroup;
\item[\textup{(2)}] $\mathsf L_f$ is topologically isomorphic to the topological semigroup $[\Tau_K]\times \HH(K)$ where the orbit space $[\Tau_K]$ is endowed with the left zero multiplication;
\item[\textup{(3)}] $\mathsf L_f$ is a semitopological semigroup;
\item[\textup{(4)}] the left shift $l_f:\mathsf L_f\to\mathsf L_f$, $l_f:g\mapsto f\circ g$, is continuous;
\item[\textup{(5)}] $f$ is continuous;
\item[\textup{(6)}] $\mathsf L_f\subset C(\Tau_K,\Tau_K)$;
\item[\textup{(7)}] $\HH(K)$ is finite and the idempotent band $E(\mathsf L_f)$ of \ $\mathsf L_f$ is compact;
\end{enumerate}
\end{theorem}

\begin{proof} The implications $(2)\Ra(1)\Ra(3)\Ra(4)$ are trivial.
\smallskip

$(4)\Ra(5)$ Assume that the left shift $l_f:\mathsf L_f\to \mathsf L_f$ is continuous. We need to check that $f$ is continuous. First we show that for any set $B\in f(\Tau_K)$ the preimage $\mathcal Z=f^{-1}(B)$ is closed in $\Tau_K$. Assume conversely that $f^{-1}(B)$ is not closed and find a point $A_0\in\overline{\mathcal Z}\setminus \mathcal Z$. It follows that the set $B_0=f(A_0)$ is not equal to $B$. Let $\varphi:\lfloor B\rfloor\to [A_0]$ be a unique equivariant function such that $\varphi(B)=A_0$. Then the function $g_0=\varphi\circ f$ belongs to the minimal left ideal $\mathsf L_f$. Observe that $f\circ g_0(B)=f(A_0)=B_0\ne B$. Since the left shift $l_f$ is continuous, for the neighborhood $O(f\circ g_0)=\{h\in \mathsf L_f:h(B)\ne B\}$ of $f\circ g_0=l_f(g_0)$ there is a neighborhood $O(g_0)\subset\mathsf L_f$ such that $f\circ g\subset O(f\circ g_0)$
for every $g\in O(g_0)$. It follows from the equivariantness of $g_0=g_0\circ f$ and the definition of the topology (of pointwise convergence) on $\mathsf L_f\subset\Tau_K^{\;\Tau_K}$ that the point $g_0(B)=A_0$ of $\Tau_K$ 
has a neighborhood $O(A_0)\subset \Tau_K$ such that each function $g\in\mathsf L_f$ with $g(B)\in O(A_0)$ belongs to the neighborhood $O(g_0)$.
Since $A_0$ is a limit point of the set $\mathcal Z$, there is a set $A\in O(A_0)\cap\mathcal Z$. For this set find an equivariant function $g=g\circ f$ such that $g(B)=A$. Then $g\in O(g_0)$ and hence $f\circ g(B)\ne B$, which contradicts $g(B)=A\in f^{-1}(B)$. This contradiction proves that all preimages $f^{-1}(B)$, $B\in f(\Tau_K)$, are closed in $\Tau_K$.

Next, we show that each orbit $\lfloor A\rfloor$, $A\in\Tau_K$, is discrete. Assume conversely that some orbit $\lfloor A\rfloor$ is not discrete and consider its closure $\overline{\lfloor A\rfloor}$ in the compact Hausdorff space $\Tau_K$. The orbit $\lfloor A\rfloor$ has no isolated points, being non-discrete and topologically homogeneous. Fix any $B\in f(\Tau_K)$. By Theorems~\ref{t2.1}(2) and \ref{t8.2}, the image $f(\Tau_K)$ has cardinality $|f(\Tau_K)|=|\lfloor B\rfloor|=|\HH(K)|\le\aleph_0$. Then we can write the compact space $\overline{\lfloor A\rfloor}$ as a countable union$$\overline{\lfloor A\rfloor}=\bigcup_{B\in f(\Tau_K)}f^{-1}(B)\cap\overline{\lfloor A\rfloor}$$  of closed subsets. By Baire's Theorem, for some $B\in f(\Tau_K)$ the set $\overline{\lfloor A\rfloor}\cap f^{-1}(B)$ has non-empty interior in $\overline{\lfloor A\rfloor}$. Since the orbit $\lfloor A\rfloor$ has no isolated points, the intersection $\lfloor A\rfloor\cap f^{-1}(B)$ is infinite, which is not possible as $f$ is equivariant.

Finally, we show that for every $B\in f(\Tau_K)$ the preimage $f^{-1}(B)$ is open in $\mathsf L_f$. Assuming the opposite, we can find a point $A_0\in f^{-1}(B)$ that lies in the closure of the set $\Tau_K\setminus f^{-1}(B)$.
Choose any equivariant function $g_0\in\mathsf L_f$ such that $g_0(B)=A_0$ and observe that $f\circ g_0(B)=B$.

Since the orbit $\lfloor B\rfloor$ of $B$ is discrete, we can find an open neighborhood $O(B)\subset \Tau_K$ of $B$ such that $O(B)\cap\lfloor B\rfloor=\{B\}$.
This neighborhood determines a neighborhood $O(f\circ g_0)=\{g\in\mathsf L_f:g(B)\in O(B)\}$ of the function $f\circ g_0$ in $\mathsf L_f\subset\Tau_K^{\;\Tau_K}$. Since the left shift $l_f:\mathsf L_f\to\mathsf L_f$ is continuous, the function $g_0$ has a neighborhood $O(g_0)\subset\mathsf L_f$ such that $l_f(O(g_0))\subset O(f\circ g_0)$. 
By the definition of the topology (of pointwise convergence) on 
$\mathsf L_f$, there is a neighborhood $O(g_0(B))\subset\Tau_K$ such that each function $g\in\mathsf L_f$ with $g(B)\in O(g_0(B))$ belongs to $O(g_0)$. By the choice of the point $A_0=g_0(B)$, there is a set $A\in O(g_0(B))\setminus f^{-1}(B)$. For this set choose an equivariant function $g\in\mathsf L_f$ such that $g(B)=A$. This function $g$ belongs to $O(g_0)$ and thus $f\circ g\in O(f\circ g_0)$, which means that $f\circ g(B)=B$. But this contradicts $g(B)=A\notin f^{-1}(B)$.

Thus for each $B\in f(\Tau_K)$ the preimage $f^{-1}(B)$ is open in $\Tau_K$, which implies that the function $f:\Tau_K\to\Tau_K$ is continuous. 
\smallskip

$(5)\Ra(6)$ Assume that $f$ is continuous. Then for any $B\in\Tau_K$ the orbit $\lfloor B\rfloor=f(\Tau_K)$ is compact (as a continuous image of the compact space $\Tau_K$). Being a compact topologically homogeneous space of cardinality $|\lfloor B\rfloor|\le|\HH(K)|\le \aleph_0$, the orbit $\lfloor B\rfloor=f(\Tau_K)$ is finite. Then for each $g\in \mathsf L_f$ the restriction $g|\lfloor B\rfloor$ is continuous and hence $g=g\circ f$ is continuous as the composition of two continuous maps $f$ and $g|\lfloor B\rfloor$.
\smallskip

$(6)\Ra(7)$ Assume that $\mathsf L_f\subset C(\Tau_K,\Tau_K)$. Then $f$ is continuous. Repeating the argument from the preceding item, we can show that the characteristic group $\HH(K)$ is finite. By the continuity of $f$, for every $B\in\Tau_K$ the preimage $f^{-1}(B)$ is closed in $\Tau_K$.
In the following claim $\mathsf E(\mathsf L_f)$ stands for the idempotent band of the semigroup $\mathsf L_f$.

\begin{claim}\label{ELf} $\mathsf E(\mathsf L_f)=\{g\in\mathsf L_f:f\circ g(B)=B\}$. 
\end{claim}

\begin{proof} If $g\in\mathsf L_f$ is an idempotent, then for the unique point $C\in g(\Tau_K)\cap f^{-1}(B)$ we get $C=g(C)$ and then $B=f(C)=fg(C)=fgf(C)=fg(B)$.

Now assume conversely that $g\in\mathsf L_f$ is a function with $fg(B)=B$.
Let $C=g(B)\in g(\Tau_K)$. Then $g(C)=gf(C)=gfg(B)=g(B)=C$. For every $A\in\Tau_K$ we can find $x\in X$ such that $g(A)=xC$ and then 
$gg(A)=g(xC)=xg(C)=xC=g(A)$, which means that $g$ is an idempotent.
\end{proof}

Since the set $f^{-1}(B)\subset \Tau_K$ is closed and the calculation map $$c_B:\mathsf L_f\to\Tau_K,\;c_B:g\mapsto g(B),$$ is continuous, the preimage $c_B^{-1}(f^{-1}(B))$ is closed in $\mathsf L_f$. By Claim~\ref{ELf}, this preimage is equal to the idempotent band $\mathsf E(\mathsf L_f)$ of the semigroup $\mathsf L_f$.
\smallskip

$(7)\Ra(2)$ Assume that the group $\HH(K)$ is finite and the idempotent band $E(\mathsf L_f)$ is compact. By Proposition~\ref{p12.2}, the orbit space $[\Tau_K]$ is compact, Hausdorff and zero-dimensional, and the quotient map $q:\Tau_K\to[\Tau_K]$ is continuous and open.

 We claim that for every $B\in f(\Tau_K)$ the preimage $f^{-1}(B)\subset\Tau_K$ is compact. Since the idempotent band $E(\mathsf L_f)$ is compact and the calculation map $c_B:\mathsf L_f\to\Tau_K$, $c_B:g\mapsto g(B)$, is continuous, the image $c_B(E(\mathsf L_f))$ is compact. By Claim~\ref{ELf}, $c_B(\mathsf E(\mathsf L_f))\subset f^{-1}(B)$.
To show the reverse inclusion, fix any subset $A\in f^{-1}(B)$ and choose any equivariant map $\varphi:\lfloor B\rfloor\to\lfloor A\rfloor$ such that $\varphi(B)=A$. Then the map $g=\varphi\circ f$ belongs to $\mathsf L_f$ and is an idempotent by Claim~\ref{ELf}. Since $A=g(B)$, we see that $f^{-1}(B)\subset c_B(\mathsf E(\mathsf L_f))$ and hence $f^{-1}(B)=c_B(\mathsf E(\mathsf L_f))$ is compact.

Fix any set $B\in f(\Tau_K)$. Since $|f(\Tau_K)|=|\lfloor B\rfloor|=|\HH(K)|<\aleph_0$, the preimage $$\mathcal Z=f^{-1}(B)=\Tau_K\setminus \bigcup_{B\ne A\in\lfloor B\rfloor}f^{-1}(A)$$ is open-and-closed in $\Tau_K$. Since the compact space $\mathcal Z$ meets each orbit $\lfloor A\rfloor$, $A\in\Tau_K$, at a single point, the restriction $q|\mathcal Z:\mathcal Z\to[\Tau_K]$, being continuous and bijective, is a homeomorphism. So, it suffices to prove that $\mathsf L_f$ is topologically isomorphic to $\mathcal Z\times \HH(K)$ where the space $\mathcal Z$ is endowed with the left zero multiplication. Define an isomorphism $\Phi:\mathcal Z\times \HH(K)\to \mathsf L_f$ assigning to each pair $(Z,x)\in \mathcal Z\times \HH(K)$ the function $g_{Z,x}\circ f$ where $g_{Z,x}:\lfloor B\rfloor\to \lfloor Z\rfloor$ is the unique equivariant function such that $g_{Z,x}(B)=x^{-1}Z$. It is easy to check that $\Phi$ is a topological isomorphism between $\mathcal Z\times \HH(K)$ and $\mathsf L_f$.
\end{proof}

In the following proposition we prove the existence of continuous or discontinuous minimal idempotents in the semigroup $\End(\Tau_K)$.
Let us recall that for a left-invariant ideal $\I$ on a group $X$ by $\End^\I(\Tau_K)$ we denote the left ideal in $\End(\Tau_K)$ consisting of all equivariant $\I$-saturated functions. 

\begin{proposition}\label{p14.5} Let $\I$ be a left-invariant ideal on a group $X$ and assume that a maximal 2-cogroup $K\subset X$ has finite characteristic group $\HH(K)$. Then the semigroup $\End^\I(\Tau_K)$ contains: 
\begin{enumerate}
\item[\textup{(1)}] a continuous minimal idempotent if $Kx\notin \I$ for all $x\in X$;
\item[\textup{(2)}] no continuous function if $Kx\in \I$ for all $x\in X$;
\item[\textup{(3)}] no discontinuous function (which is a minimal idempotent) if (and only if) for each $A\in\Tau_K$ the set $\bar{\bar A}^\I\cap\Tau_K$ is open in $\Tau_K$.
\end{enumerate}
\end{proposition}

\begin{proof} By Proposition~\ref{p12.2}, the orbit space $[\Tau_K]$ is compact, Hausdorff, and zero-dimensional and the orbit map $q:\Tau_K\to[\Tau_K]$ has a continuous section $s:[\Tau_K]\to\Tau_K$. Then $\mathcal Z=s([\Tau_K])$ is a closed subset of $\Tau_K$ that meets each orbit $\lfloor A\rfloor$, $A\in\Tau_K$, at a single point. Pick any $B\in\mathcal Z$ and define a continuous minimal idempotent $f:\Tau_K\to\Tau_K$ letting $f(xZ)=xB$ for each $x\in \HH(K)$ and $Z\in\mathcal Z$.
\smallskip

1. Assuming that $Kx\notin \I$ for all $x\in X$, we shall show that the function $f$ is $\I$-saturated and hence belongs to $\End^\I(\Tau_K)$. Given any sets $A,B\in\Tau_K$ with $A=_\I B$, we need to show that $f(A)=f(B)$. We shall prove more: $A=B$. Assume conversely that $A\ne B$ and find a point $x\in A\triangle B$ in the symmetric difference $A\triangle B=(A\setminus B)\cup(B\setminus A)$. Since $KKA=\Fix(A)A=A$ and $KKB=\Fix(B)B=B$, we get $KKx\in A\triangle B\in\I$ and then for every $y\in K$, we get $Kyx=KKx\in\I$, which contradicts our assumption. 
\smallskip

2. Now assume that $Kx\in\I$ for all $x\in X$. We shall prove that no function $g\in\End^\I(\Tau_K)$ is continuous. For this we show that for each $A\in\Tau_K$ the set $\bar{\bar A}^\I\cap\Tau_K$ is dense in $\Tau_K$.
Given any set $C\in\Tau_K$ and a neighborhood $O(C)$ of $C$ in $\Tau_K$, we need to find a set $B\in O(C)$ such that $B=_\I A$. By the definition of the topology on $\Tau_K\subset\mathsf P(X)$, there is a finite subset $F\subset X$ such that $O(C)\supset\{B\in\Tau_K:B\cap F=C\cap F\}$. Now we see that the set $B=(A\setminus K^\pm F)\cup(K^\pm F\cap C)\in\Tau_K$ belongs to the neighborhood $O(C)$ and $B=_\I A$ because $A\triangle B\subset K^\pm F\in\I$. Assuming that some $\I$-saturated equivariant function $g:\Tau_K\to\Tau_K$, is continuous, we conclude that the preimage $g^{-1}(f(A))\supset\bar{\bar A}^\I\cap\Tau_K$ coincides with $\Tau_K$, being a closed dense subset of $\Tau_K$. So, $g$ is constant. Since the action of the (non-trivial) group $\HH(K)$ on $\Tau_K$ is free, the constant map $g$ cannot be equivariant.   
\smallskip

3. If for every $A\in\Tau_K$ the set $\bar{\bar A}^\I\cap\Tau_K$ is open in $\Tau_K$, then each $\I$-saturated function is locally constant and hence continuous. So, $\End^\I(\Tau_K)$ contains no discontinuous function.

Now assuming that for some $A\in\Tau_K$ the set $\A=\bar{\bar A}^\I\cap\Tau_K$ is not open in $\Tau_K$, we shall construct a discontinuous minimal idempotent $f\in\End^\I(\Tau_K)$. Take any minimal idempotent $f\in\End^\I(\Tau_K)$. If $f$ is discontinuous, we are done. So assume that $f$ is continuous and fix any set $B\in f(\Tau_K)$. By Theorem~\ref{t14.1}(5), the image $f(\Tau_K)=\lfloor B\rfloor$ is finite. So, the preimage $\mathcal Z=f^{-1}(f(B))$ is open-and-closed in $\Tau_K$.
Take any $x\in\Stab(K)\setminus KK$ and consider the subset $\mathcal Z'=(\mathcal Z\setminus \A)\cup x\A$ which is not compact as $\A$ is not open in $\mathcal Z$. Then the $\I$-saturated minimal idempotent $g:\Tau_K\to\Tau_K$ defined by $g(xZ)=xB$ for $x\in \HH(K)$ and $Z\in\mathcal Z'$ is discontinuous (because $g^{-1}(B)=\mathcal Z'$ is not closed in $\Tau_K$).
\end{proof}

\begin{corollary}\label{c14.6} For a maximal 2-cogroup $K\subset X$ and a left-invariant ideal $\I$ on $X$ the following conditions are equivalent:
\begin{enumerate}
\item[(1)] each minimal left ideal of $\End^\I(\Tau_K)$ is a topological semigroup;
\item[(2)] each minimal left ideal of $\End^\I(\Tau_K)$ is a semitopological semigroup;
\item[(3)] for each $A\in\Tau_K$ the set $\bar{\bar A}^\I\cap\Tau_K$ is open in $\Tau_K$.
\end{enumerate}
If the ideal $\I$ is right-invariant, then the conditions (1)--(3) are equivalent to
\begin{enumerate}
\item[(4)] $K$ has finite index in $X$;
\item[(5)] the semigroup $\End(\Tau_K)$ is finite.
\end{enumerate}
\end{corollary}

\begin{proof} The equivalence $(1)\Leftrightarrow(2)\Leftrightarrow(3)$ follows from Theorem~\ref{t14.3} and Proposition~\ref{p14.5}.

Now assume that the left-invariant ideal $\I$ is right-invariant.

$(3)\Ra(4)$ Assume that for each $A\in\Tau_K$ the set 
$\bar{\bar A}^\I\cap\Tau_K$ is open in $\Tau_K$. Then it is also closed in $\Tau_K$ being the complements of the union of open subsets $\bar{\bar B}^\I\cap \Tau_K$ for $B\ne_\I A$. By Proposition~\ref{p14.5}, $Kx\notin\I$ for some $x\in X$. Since the ideal $\I$ is right-invariant, $Kx\notin\I$ for all $x\in X$. We claim that the set $\bar{\bar A}^\I\cap\Tau_K=\{A\}$. Assuming that $\bar{\bar A}^\I\cap\Tau_K$ contains a set $B$ distinct from $A$, we can find a point $x\in A\triangle B$. Since $\Fix(A)=KK=\Fix(B)$, we get $KKx\subset KKA\triangle KKB=A\triangle B\in\I$ and thus for any point $z\in K$, we arrive to the absurd conclusion $Kzx=KKx\in\I$. Since the singleton $\bar{\bar A}^\I\cap\Tau_K=\{A\}$ in open in the space $\Tau_K$, which is homeomorphic to $2^{X/K^\pm}$, the index of the group $K^\pm$ in $X$ is finite and so is the index of $K$ in $X$.
\smallskip
The implications $(4)\Ra(5)\Ra(1)$ are trivial.
\end{proof}

\section{The semigroup $\Enl(\Tau_K)$}\label{s15}

In the preceding section we studied the continuity of the semigroup operation on minimal left ideals of the semigroup $\End(\Tau_K)$. In this section we shall be interested in the continuity of the semigroup operation on the semigroup $\Enl(\Tau_K)\subset\End(\Tau_K)$. This will be done in a more general context of upper subfamilies $\mathsf F\subset\Tau$.
We define a family $\mathsf F\subset\Tau$ to be {\em upper} if for any twin set $A\in\mathsf F$ and a twin subset $B\subset X$ with $\Fix^-(A)\subset\Fix^-(B)$, we get $B\in\mathsf F$. 

Let us remark that $\wht\Tau$ is an upper subfamily of $\Tau$ while  $\Tau_{K}$ is a minimal upper subfamily of $\Tau$ for every $K\in\wht{\K}$.

\begin{proposition} Each upper subfamily $\mathsf F\subset\Tau$ is symmetric and $\lambda$-invariant. Consequently, $\Enl(\mathsf F)$ is a compact right-topological semigroup.
\end{proposition}

\begin{proof} To prove that $\mathsf F\subset\Tau$ is symmetric, given  any set $A\in\mathsf F$ choose a point $x\in\Fix^-(A)$. By Proposition~\ref{p5.4}, $\Fix^-(xA)=x\,\Fix^-(A)\,x^{-1}=\Fix^-(A)$ and hence $X\setminus A=xA\in\mathsf F$.

To see that $\mathsf F$ is $\lambda$-invariant, we need to show that $\varphi(\mathsf F)\subset\mathsf F$ for any function $\varphi\in\Enl(\mathsf P(X))$. By Corollary~\ref{c4.4}, the function $f$ is symmetric and left-invariant. Then for each $A\in\mathsf F$ and  $x\in\Fix^-(A)$ we get $x\varphi(A)=\varphi(xA)=\varphi(X\setminus A)=X\setminus\varphi(A)$ and hence $x\in\Fix^-(\varphi(A))$. Since $\Fix^-(A)\subset\Fix^-(\varphi(A))$, the set $\varphi(A)$ belongs to $\mathsf F$ by the definition of an upper family.
\end{proof}

\begin{theorem}\label{t16.1} For an upper subfamily $\mathsf F\subset\Tau$ the following conditions are equivalent:
\begin{enumerate}
\item[\textup{(1)}] $\Enl(\mathsf F)$ is a topological semigroup;
\item[\textup{(2)}] $\Enl(\mathsf F)$ is a semitopological semigroup;
\item[\textup{(3)}] for each twin set $A\in\mathsf F$ the subgroup $\Fix(A)$ has finite index in $X$.
\end{enumerate}
\end{theorem}

\begin{proof} $(3)\Ra(1)$ Assume that for each twin set $A\in\mathsf F$ the stabilizer $\Fix(A)$ has finite index in $X$.
To show that the semigroup operation $\circ :\Enl(\mathsf F)\times\Enl(\mathsf F)\to\Enl(\mathsf F)$ is continuous, fix any two functions $f,g\in\Enl(\mathsf F)$ and a neighborhood $O(f\circ g)$ of their composition. We should show that the functions $f,g$ have neighborhoods $O(f),O(g)\subset \Enl(\mathsf F)$ such that $O(f)\circ O(g)\subset O(f\circ g)$.
We lose no generality assuming that the neighborhood $O(f,g)$ is of sub-basic form: $$O(f\circ g)=\{h\in\Enl(\mathsf F): x\in h(A)\}$$
for some $x\in X$ and some twin set $A\in\mathsf F$.
Let $B=g(A)$. It follows from $f\circ g\in O(f\circ g)$ that $x\in f\circ g(A)=f(B)$.
Let $O(f)=\{h\in\Enl(\mathsf F):x\in h(B)\}$.

The definition of a neighborhood $O(g)$ is a bit more complicated.
By our hypothesis, the stabilizer $\Fix(A)$ has finite index in $X$.
Let $S\subset X$ be a (finite) subset meeting each coset $\Fix(A)\,z$, $z\in X$, at a single point. Consider the following open neighborhood of $g$ in $\Enl(\mathsf F)$:
$$O(g)=\{g'\in\Enl(\mathsf F):\forall s\in S\;\;(s\in B\;\Leftrightarrow s\in g'(A))\}.$$
We claim that  $O(f)\circ O(g)\subset O(f\circ g)$.
Indeed, take any functions $f'\in O(f)$ and $g'\in O(g)$.
By Theorem~\ref{t11.1}, $\Fix^-(A)\subset \Fix^-(g'(A))$ and hence $\Fix(A)\subset\Fix(g'(A))$. Then $g'(A)=\Fix(A)\cdot(S\cap g'(A)) =\Fix(A)\cdot (S\cap B)=B$ and thus $x\in f'(B)=f'\circ g'(A)$ witnessing that $f'\circ g'\in O(f\circ g)$.
\smallskip

The implication $(1)\Ra(2)$ is trivial.
\smallskip

$(2)\Ra (3)$ Assume that $X$
contains a twin subset $T_0\in\mathsf F$ whose stabilizer $\Fix(T_0)$ has infinite index in $X$.
Then the subgroup $H=\Fix^\pm(T_0)$ also has infinite index in $X$. By Theorem 15.5 of \cite{P}, $X\ne FHF$ for any finite subset $F\subset X$. 

\begin{lemma}\label{l16.2} There are countable sets $A,B\subset X$ such that
\begin{enumerate}
\item[\textup{(1)}] $xB\cap yB=\emptyset$ for any distinct $x,y\in A$;
\item[\textup{(2)}] $|AB\cap Hz|\le 1$ for all $z\in X$;
\item[\textup{(3)}] $e\in A$, $AB\cap H=\emptyset$.
\end{enumerate}
\end{lemma}

\begin{proof} Let $a_0=e$ and $B_{<0}=\{e\}$. Inductively we shall construct sequences  $A=\{a_n:n\in\w\}$ and $B=\{b_n:n\in\w\}$ such that
\begin{itemize}
\item $b_n\notin A_{\le n}^{-1}HA_{\le n}B_{<n}$ where $A_{\le n}=\{a_i:i\le n\}$ and
$B_{<n}=\{e\}\cup\{b_i:i< n\}$;
\item $a_{n+1}\notin HA_{\le n}B_{\le n}B_{\le n}^{-1}$.
\end{itemize}
Since $X\ne FHF$ for any finite subset $F\subset X$, the choice of the points $b_n$ and $a_{n+1}$ at the $n$-th step is always possible. It is easy to check that the sets $A,B$ satisfy the conditions (1)--(3) of the lemma.
\end{proof}

The properties (2), (3) of the set $AB$ allows us to enlarge $AB$ to a subset $S$ that contains the neutral element of $X$ and meets each coset $Hz$, $z\in X$, at a single point. Observe that each subset $E\subset S$ generates a twin subset $$T_E=\Fix(T_0)\cdot E\cup \Fix^-(T_0)\cdot (S\setminus E)$$ of $X$ such that $\Fix^-(T_0)\subset\Fix^-(T_E)$ and hence $T_E\in\mathsf F$.

\begin{lemma}\label{l16.3} There is a free ultrafilter $\mathcal B$ on $X$ and a family of subsets $\{U_a:a\in A\}\subset \mathcal B$ such that
\begin{enumerate}
\item[\textup{(1)}] $\bigcup_{a\in A}U_a\subset B$;
\item[\textup{(2)}] the set $U=\bigcup_{a\in A}aU_a$ has the property $B\not\subset x^{-1}U\cup y^{-1}U$ for every $x,y\in A$;
\item[\textup{(3)}] for every $V\in\mathcal B$ the set $\{a\in A:aV\subset U\}$ is finite.
\end{enumerate}
\end{lemma}

\begin{proof} Let $A=\{a_n:n\in\w\}$ and $B$ be the sets constructed in Lemma~\ref{l16.2}.  For every $n\in\w$ put
 $A_{\le n}=\{a_i:i\le n\}$.
  Let $B_{<0}=\{e\}$ and inductively, for every $n\in\w$ choose an element
 $b_n\in B$ so that
$$b_n\notin A_{\le_n}^{-1}A_{\le n}B_{<n}\mbox{ \ where \ }B_{<n}=\{b_i:i<n\}.$$
For every $n\in\w$ let $B_{\ge n}=\{b_{i}:i\ge n\}$. Let also
 $B_{2\w}=\{b_{2n}:n\in\w\}$.

Let us show that for any distinct numbers $n,m$ the intersection
 $a_nB_{\ge n}\cap a_mB_{\ge m}$ is empty.
Otherwise there would exist two numbers $i\ge n$ and $j\ge m$ such that
 $a_nb_{i}=a_mb_{j}$.
It follows from $a_n\ne a_m$ that $i\ne j$. We lose no generality
 assuming that $j>i$.
Then $a_nb_{i}=a_mb_{j}$ implies that
$$b_{j}=a_m^{-1}a_nb_{i}\in A_{\le j}^{-1}A_{\le j}B_{<j},$$ which contradicts
 the choice of $b_{j}$.

Let $\mathcal B\in\beta(X)$ be any free ultrafilter such that
 $B_{2\w}\in\mathcal B$ and $\mathcal B$
is not a P-point in $\beta(X)\setminus X$. To get such an ultrafilter, take $\mathcal B$ to be a
 cluster point of any
countable subset of $\beta(B_{2\w})\setminus B_{2\w}\subset \beta(X)$. Using the fact
 that $\mathcal B$ fails
 to be a P-point, we can take a decreasing sequence of subsets
 $\{V_n:n\in\w\}\subset\mathcal B$ of $B_{2\w}$
  having no pseudointersection in $\mathcal B$. The latter means that
 for every $V\in\mathcal B$
   the almost inclusion $V\subset^* V_n$ (which means that $V\setminus
 V_n$ is finite) holds only
   for finitely many numbers $n$.

For every $a=a_n\in A$ let $U_a=V_n\cap B_{\ge n}$. We claim that
 the ultrafilter $\mathcal B$,
the family $(U_a)_{a\in A}$, and the set $U=\bigcup_{a\in A}a U_a=\bigcup_{n\in\w}a_n(V_n\cap B_{\ge n})$ satisfy the requirements of
 the lemma.

First, we check that $B\not\subset a_n^{-1}U\cup a_m^{-1}U$ for all $n\le m$. Take any odd number $k>m$. We claim that $b_k\notin a_n^{-1}U\cup a_m^{-1}U$. Otherwise, $b_{k}\in a^{-1}_n a_i(V_i\cap B_{\ge i})\cup a_m^{-1}a_i(V_i\cap B_{\ge i})$ for some $i\in \w$ and hence $b_{k}=a_n^{-1}a_i b_j$ or $b_{k}=a^{-1}_ma_ib_j$ for some even $j\ge i$. If $k>j$, then both the equalities are
forbidden by the choice of $b_{k}\notin A_{\le k}^{-1}A_{\le k}B_{<k}\supset\{a^{-1}_na_ib_j,a^{-1}_ma_ib_j\}$. If $k<j$, then those
equalities are forbidden by the choice of
$b_j\notin
A_{\le j}^{-1}A_{\le j}B_{<j}\supset\{a_i^{-1}a_nb_k, a^{-1}_ia_mb_k\}$. Therefore, $B\not\subset a_n^{-1}U\cup a_m^{-1}U$.

Next, given arbitrary $V\in\mathcal B$ we show that the set $A'=\{a\in A:aV\subset U\}$ is finite. By the choice of the sequence $(V_n)$, the set $F=\{a_n:V\cap
B_{2\w}\subset^* V_n\}$ is finite. We claim that $A'\subset F$. Indeed, take any $a_n\in A'$. It follows from $a_nV\subset
U=\bigcup_{a\in A}aB_a$ and $a_nB\cap \bigcap_{i\ne n}a_i B=\emptyset$ that $$a_n(V\cap
B_{2\w})\subset^*a_n(V_n\cap B_{\ge n})\subset a_n V_n$$ and hence $a_n\in F$.
\end{proof}

Let $\mathcal A$ be any free ultrafilter on $X$ containing the set $A$ and observe that $U=\bigcup_{a\in A}a(V_n\cap B_{\ge n})\in\A\circ\mathcal B$. Let $\alpha=\Phi_{\mathsf F}(\A)$ and $\beta=\Phi_{\mathsf F}(\mathcal B)$ be the function representations of the ultrafilters $\A$ and $\mathcal B$, respectively.  We claim that the left shift $l_\alpha:\Enl(\mathsf F)\to\Enl(\mathsf F)$, $l_\alpha:f\mapsto \alpha\circ f$, is discontinuous
at $\beta$. Since $U\subset AB\subset S$, we can consider the twin set $$T=\Fix(T_0)\cdot U\cup \Fix^-(T_0)\cdot (S\setminus U)$$ and observe that $T\in \A\circ\mathcal B$. Consequently, $\alpha\circ\beta(T)=\{x\in G:x^{-1}T\in\A\circ\mathcal B\}$ contains the neutral element, which implies that  $O(\alpha\circ\beta)=\{f\in\Enl(\mathsf F):e\in f(T)\}$ is a neighborhood of $l_\alpha(\beta)=\alpha\circ\beta$ in $\Enl(\mathsf F)$.

Assuming that $l_\alpha$ is continuous at $\beta$, we can find a neighborhood $O(\beta)\subset\Enl(\mathsf F)$ of $\beta$ such that $l_\alpha(O(\beta))\subset O(\alpha\circ\beta)$. Since $\mathsf F$ is left-invariant, we can assume that $O(\beta)$ is of the basic form:
$$O(\beta)=\{f\in\Enl(\mathsf F):e\in\bigcap_{i=1}^n f(T_i)\}$$ for some twin sets $T_1,\dots,T_n\in\mathsf F$. It follows from $\beta\in O(\beta)$ that $e\in\beta(T_i)$ and thus $T_i\in\mathcal B$ for every $i\le n$.
According to Lemma~\ref{l16.3}(3), the set $F=\{a\in A:B\cap \bigcap_{i=1}^n T_i\subset a^{-1}U\}$ is finite.

We claim that the family $\LL=\{T_1,\dots, T_n,X\setminus x^{-1}T:x\in A\setminus F\}$ is linked.
This will follow as soon as we check that
\begin{itemize}
\item[(i)] $T_i\cap (X\setminus x^{-1}T)\ne\emptyset$ for any $i\le n$ and $x\in A\setminus F$;
\item[(ii)] $(X\setminus x^{-1}T)\cap (X\setminus y^{-1}T)\ne\emptyset$ for all $x,y\in A$.
\end{itemize}

The item (i) is equivalent to $T_i\not\subset x^{-1}T$ for $x\in A\setminus F$. Assuming conversely that $T_i\subset x^{-1}T$, we will consecutively get $xT_i\subset T$, $S\cap xT_i\subset S\cap T=U$, and finally  $B\cap T_i\subset x^{-1}S\cap T_i\subset x^{-1}U$, which contradicts $x\notin F$.

The item (ii) is equivalent to $x^{-1}T\cup y^{-1}T\ne X$ for $x,y\in A$. Assume conversely that $x^{-1}T\cup y^{-1}T=X$ for some $x,y\in A$.
It follows from $xB\subset S$ that $xB\cap T=xB\cap U$ and thus $B\cap x^{-1}T=B\cap x^{-1}U$. Similarly, $B\cap y^{-1}T=B\cap y^{-1}U$. Consequently, $$B=B\cap X=B\cap (x^{-1}T\cup y^{-1}T)=B\cap (x^{-1}U\cup y^{-1}U)\ne B$$ according to Lemma~\ref{l16.3}(2). This contradiction completes the proof of the linkedness of $\LL$.

Being linked, the family $\LL$ can be enlarged to a maximal linked system $\C\in\lambda(X)$. It follows from $T_1,\dots,T_n\in\LL\subset\C$ that the twin representation $\gamma=\Phi_{\mathsf F}(\C)$ belongs to the neighborhood $O(\beta)$ and consequently, $\alpha\circ\gamma\in O(\alpha\circ\beta)$, which means that $T\in\A\circ \C$. The latter is equivalent to $A'=\{x\in X:x^{-1}T\in\C\}\in\A$. On the other hand, $X\setminus A'=\{x\in X:X\setminus x^{-1}T\in\C\}$ contains the set $A\setminus F\in\A$ and thus $X\setminus A'\in\A$, which is a contradiction.
\end{proof}

\begin{theorem}\label{t16.4} If the group $X$ is twinic, then for an upper subfamily $\mathsf F\subset\Tau$ the following conditions are equivalent:
\begin{enumerate}
\item[\textup{(1)}] $\Enl(\mathsf F)$ is metrizable;
\item[\textup{(2)}] $\Enl(\mathsf F)$ is a metrizable topological semigroup;
\item[\textup{(3)}] $\mathsf F$ is at most countable.
\end{enumerate}
\end{theorem}

\begin{proof} We shall prove the implications $(3)\Ra(2)\Ra(1)\Ra(3)$.

$(3)\Ra(2)$. Assume that the family $\mathsf F$ is at most countable. We claim that for each twin subset $T\in\mathsf F$ the 2-cogroup  $K=\Fix^-(T)$ has finite index in $X$. Otherwise, the subgroup $K^\pm=KK\cup K$ also has infinite index on $X$ and then $|\mathsf F|\ge|\Tau_K|=2^{|X/K^\pm|}\ge 2^\w>\aleph_0$.

So, $\Fix(T)$ has finite index in $X$ and the implication $(3)\Ra(1)$ of Theorem~\ref{t16.1} guarantees that $\Enl(\mathsf F)$ is a topological semigroup. Now we show that this semigroup is metrizable. First observe that for every $T\in\mathsf F$ the set $\Enl(\{T\})=\{\varphi|\{T\}:\varphi\in\Enl(\mathsf P(X))\}$ has finite cardinality 
$$|\Enl(\{T\})|=|\{\varphi(T):\varphi\in\Enl(\mathsf P(X))\}|\le| \{A\in\Tau:\Fix(A)\supset \Fix(T)\}|.$$ Since the family $\mathsf F$ is countable, the space $\Enl(\mathsf F)\subset\prod_{T\in\mathsf F}\Enl(\{T\})$ is metrizable, being a subspace of the countable product of finite discrete spaces.
\smallskip

The implication $(2)\Ra(1)$ is trivial.
\smallskip

$(1)\Ra(3)$ Assuming that the family $\mathsf F$ is not countable, we shall show that the space $\Enl(\mathsf F)$ is not metrizable. We consider two cases.
\smallskip

(a) For some twin set $T\in\mathsf F$ the stabilizer $\Fix(T)$ has infinite index.
Then we can find an infinite set $S\subset X$ that intersects each coset $\Fix^\pm(T)x$, $x\in X$, at a single point. As we already know,  for each subset $E\subset S$ the set
$$T_E=\Fix(T)\cdot E\cup \Fix^-(T)\cdot(S\setminus E)$$ belongs to the family $\mathsf F$. Now take any two distinct ultrafilters $\U,\V\in\beta(S)\subset\beta(X)$  and consider their function representations $f_\U=\Phi_{\mathsf F}(\U)$ and $f_\V=\Phi_{\mathsf F}(\V)$. Since $\U\ne\V$, there is a subset $E\subset S$ such that $E\in\U\setminus\V$. It follows that $T_E\in\U$ and $T_{S\setminus E}\in\V$, which implies $T_E\notin \V$ and hence $e\in f_\U(T_E)\setminus f_\V(T_E)$. This means that $f_\U\ne f_\V$ and consequently, $|\Enl(\mathsf F)|\ge|\beta(S)|\ge 2^{\mathfrak c}$, which implies that the compact space $\Enl(\mathsf F)$ is not metrizable (because each metrizable compact space has cardinality $\le\mathfrak c$).
\smallskip

(b) For each $T\in\mathsf F$ the subgroup $\Fix(T)$ has finite index in $X$.
Then each set $T\in\mathsf F$ has finite orbit $[T]=\{xT:x\in X\}$.
Consider the smallest left-invariant family $\bar{\mathsf F} =\bigcup_{T\in\mathsf F}[T]$ that contains $\mathsf F$. By Proposition~\ref{p13.4}, the family $\bar{\mathsf F}$ is $\{\emptyset\}$-independent. Since each orbit $[T]$, $T\in\bar{\mathsf F}$, is finite and $\mathsf F$ is uncountable, the orbit space 
$[\bar{\mathsf F}]=\{[T]:T\in\mathsf F\}$ also is uncountable. It follows from Theorem~\ref{t11.1} that the space $\Enl(\bar{\mathsf F})$ is homeomorphic to the product $\prod_{[T]\in[\bar{\mathsf F}]}\Enl([T])$ where each space $\Enl([T])$ contains at least two equivariant functions:  identity $i:[T]\to[T],\; i:A\mapsto A$ and  antipodal $\alpha:[T]\to[T]$, $\alpha:A\mapsto X\setminus A$. Since the orbit space $[\bar{\mathsf F}]$ is uncountable, the product $\prod_{[T]\in[\bar{\mathsf  F}]}\Enl([T])$ is non-metrizable and so is its topological copy $\Enl(\bar{\mathsf F})$. 

It remains to observe that the restriction map $R:\Enl(\bar{\mathsf F} )\to\Enl(\mathsf F)$ is injective and thus a homeomorphism. Indeed, given two distinct equivariant functions $f,g\in \Enl(\mathsf F)$, we can find a set $A\in\bar{\mathsf F}$ with $f(A)\ne g(A)$. Since $[A]\cap \mathsf F\ne\emptyset$, there is $x\in X$ such that $xA\in\mathsf F$. Then $f(xA)=xf(A)\ne xg(A)=g(xA)$ and thus $f|\mathsf F\ne g|\mathsf F$.
\end{proof}

The following proposition characterizes groups containing only
countably many twin subsets. Following \cite{BGN}, we define a
group $X$ to be {\em odd} if each element $x\in X$ has odd order.

\begin{proposition}\label{p16.5}

The family $\Tau$ of twin subsets of a group $X$ is at most
countable if and only if each subgroup of infinite index in $X$ is
odd.
\end{proposition}

\begin{proof}
Assume that each subgroup of infinite index in $X$ is odd. We
claim that for every $A\in\Tau$ the subgroup $\Fix(A)$ has finite
index in $X$. Take any point $c\in\Fix^-(A)$ and consider the
cyclic subgroup $c^\IZ=\{c^n:n\in\IZ\}$ generated by $c$. The
subgroup $c^\IZ$ has finite index in $X$, being non-odd. Since
$c^{2\IZ}=\{c^{2n}:n\in\IZ\}\subset\Fix(A)$, we conclude that
$\Fix(A)$ also has finite index in $X$.

Next, we show that the family $\{\Fix(A):A\in\Tau\}$ is at most
countable. This is trivially true if $\Tau=\emptyset$. If
$\Tau\ne\emptyset$, then we can take any $A\in\Tau$ and choose a
point $c\in\Fix^-(A)$. The cyclic subgroup $c^\IZ$ generated by
$c$ is not odd and hence has finite index in $X$. Consequently,
the group $X$ is at most countable. Now it remains to check that
for every $x\in X$ the set $\Tau_x=\{A\in\Tau: x\in\Fix^-(A)\}$ is
finite. If the set $\Tau_x$ is not empty, then the cyclic subgroup
$x^\IZ$ generated by $x$ is not odd and hence has finite index in
$X$. Consider the subgroup $x^{2\IZ}$ of index 2 in $x^\IZ$. It is
clear that $x^{2\IZ}\subset\Fix(A)$. Let $S\subset X$ be a finite
set containing the neutral element of $X$ and meeting each coset
$x^{2\IZ} z$, $z\in X$ at a single point. It follows from
$x^{2\IZ}\subset\Fix(A)$ that $A=x^{2\IZ}\cdot(S\cap A)$ and
consequently $|\Tau_x|\le 2^{|S|}<\infty$.
\smallskip

Now assume that some subgroup $H$ of infinite index in $X$ is not
odd. Then $H$ contains an element $c\in H$ such that the sets
$c^{2\IZ}=\{c^{2n}:n\in\IZ\}$ and
$c^{2\IZ+1}=\{c^{2n+1}:n\in\IZ\}$ are disjoint. The union
$c^{2\IZ}\cup c^{2\IZ+1}$ coincides with the cyclic subgroup
$c^\IZ$ of $H$ generated by $c$. Find a set $S\subset X$ that
intersects each coset $c^\IZ x$, $x\in X$, at a single point.
Since $c^{2\IZ}$ has infinite index in $X$, the set $S$ is
infinite. Now observe that for every $E\subset S$ the union
$$T_E=c^{2\IZ}\cdot E\cup c^{2\IZ+1}\cdot(S\setminus E)$$  is a
twin set with $c\in \Fix^-(T_E)$. Consequently,
$\Tau\supset\{T_E:E\subset S\}$ has cardinality
$$|\Tau|\ge|\{T_E:E\subset S\}|\ge |2^S|\ge \mathfrak
c>\aleph_0.$$
\end{proof}

Now we shall apply the above results to the minimal upper subfamilies $\Tau_{K}$ with $K\in\wht{\K}$.
By Theorem~\ref{t14.1}(1), for a maximal 2-cogroup $K$ in a group $X$ minimal left ideals of $\End(\Tau_K)$ are metrizable if and only if $|X/K|\le\aleph_0$. The metrizability of the whole semigroup $\End(\Tau_K)$ is equivalent to $|X/K|<\aleph_0$.

\begin{theorem}\label{t16.6} For a maximal 2-cogroup $K$ of a group $X$  the following conditions are equivalent:
\begin{enumerate}
\item[\textup{(1)}] $\Enl(\Tau_K)$ is metrizable;
\item[\textup{(2)}] $\Enl(\Tau_K)$ is a semitopological semigroup;
\item[\textup{(3)}] $\Enl(\Tau_K)$ is a finite semigroup;
\item[\textup{(4)}] $\Enl(\Tau_K)$ is isomorphic to $C_{2^k}\wr m^m$ or $Q_{2^k}\wr m^m$ for some $1\le k\le m<\infty$;
\item[\textup{(5)}] $K$ has finite index in $X$.
\end{enumerate}
\end{theorem}

\begin{proof} 
The implications $(1)\Ra(2)\Ra(5)$ follow from Theorems~\ref{t16.4} and \ref{t16.1}.
\smallskip

$(5)\Ra(4)$ Assume that $K$ has finite index in $X$. Then the characteristic group $\HH(K)$ of $K$ is finite and hence is isomorphic to $C_{2^k}$ or $Q_{2^k}$ for some $k\in\IN$, see Theorem~\ref{t8.2}. Also the set $\Tau_K$ is finite and so is the orbit space $[\Tau_K]$.
By Theorem~\ref{t14.1}(3), the semigroup $\Enl(\Tau_K)$ is isomorphic to $\HH(K)\wr[\Tau_K]^{[\Tau_K]}$ and the latter semigroup is isomorphic to $C_{2^k}\wr m^m$ or $Q_{2^k}\wr m^m$ for $m=|[\Tau_K]|$.
\smallskip

The  implications $(4)\Ra(3)\Ra(1)$ are trivial.
\end{proof}

\section{Constructing nice idempotents in the semigroup $\Enl(\mathsf P(X))$} 

In this section we prove the existence some special idempotents in 
the semigroup\break $\Enl(\mathsf P(X))$. These idempotents will help us to describe the structure of the minimal ideal of the semigroup $\Enl(\mathsf P(X))$ and $\lambda(X)$ in Theorems~\ref{t17.1} and Corollary~\ref{c13.2}.

In this section we assume that $\I$ is a left-invariant ideal in a group $X$.
We recall that $\pT^\I$ and $\Tau^\I$ denote the families of $\I$-pretwin and $\I$-twin subsets of $X$, respectively. A function $f:\mathsf F\to\mathsf P(X)$ defined on a subfamily $\mathsf F\subset \mathsf P(X)$ is called {\em $\I$-saturated} if $f(A)=f(B)$ for any sets $A=_\I B$ in $\mathsf F$. 

\begin{proposition}\label{p16.1} There is an idempotent $e_{\I}\in\Enl(\mathsf P(X))$ such that
\begin{itemize}
\item $e_{\I}(\mathsf P(X)\setminus \pT^\I)\subset\{\emptyset,X\}$;
\item $e_\I|\pT^\I=\id|\pT^\I$;
\item the function $e_\I$ restricted to $\mathsf P(X)\setminus\pT^\I$ is $\I$-saturated.
\end{itemize}
\end{proposition}

\begin{proof} Consider the family $\inv[N]{}^\I_2(X)\subset\mathsf P^2(X)$ of left invariant $\I$-saturated linked systems on $X$, partially ordered by the inclusion relation. This set is not empty because it contains the invariant $\I$-saturated linked system $\{X\setminus A:A\in\I\}$. By Zorn's Lemma, the partially ordered set $\inv[N]{}^\I_2(X)$ contains a maximal element $\LL$, which is a maximal invariant $\I$-saturated linked system on the group $X$. By the maximality, the system $\LL$ is monotone. 
Now consider the family
$$\LL^\perp=\{A\subset X:\forall L\in\LL\;\;(A\cap L\ne\emptyset)\}.$$

\begin{claim}\label{cl12.2} $\LL^\perp\setminus\LL\subset\pT^\I$.
\end{claim}

\begin{proof} Fix any set $A\in\LL^\perp\setminus \LL$. First we check that $xA\cap A\in\I$
for some $x\in X$. Assuming the converse, we would conclude that
the family 
$\A=\{A'\subset X:\exists x\in X\;\;(A'=_\I xA)\}$ is invariant, $\I$-saturated and linked, and so is the union $\A\cup\LL$, which is not possible by the maximality of $\LL$. So, there is $x\in X$ with $xA\cap A\in\I$, which is equivalent to 
$xA\subset_\I X\setminus A$.

Next, we find $y\in X$ such that $A\cup yA=_\I X$, which is equivalent to $X\setminus A\subset_\I yA$. 
Assuming that no such a point $y$ exists, we conclude that for any
$x,y\in X$ the union $xA\cup yA\neq_{\I} X$. Then $(X\setminus
xA)\cap (X\setminus yA)=X\setminus (xA\cup yA)\notin\I$, which
means that the family $\mathcal B=\{B\subset X:\exists x\in X\;\;(B=_\I X\setminus xA)\}$ is invariant $\I$-saturated and linked. We claim that $X\setminus A\in\LL^\perp$. 
 Assuming the converse, we would conclude that
$X\setminus A$ misses some set $L\in\LL$. Then $L\subset A$ and
hence $A\in\LL$ which is not the case. Thus $X\setminus
A\in\LL^\perp$. Since $\LL$ is invariant and $\I$-saturated, $\mathcal B\subset\LL^\perp$ and consequently, the union $\mathcal B\cup\LL$, being an invariant $\I$-saturated linked system, coincides with $\LL$. Then $X\setminus A\in\LL$, which contradicts $A\in\LL^\perp$. This contradiction shows that $X\setminus A\subset_\I yA$ for some $y\in X$. 

Since $xA\subset_\I X\setminus A\subset_\I yA$, the set $A$ is $\I$-pretwin.
\end{proof}

Consider the function representation $\Phi_{\LL}:\mathsf P(X)\to\mathsf P(X)$ of $\LL$. By Propositions~\ref{p4.3} and \ref{p4.5}, the function $\Phi_\LL$ is equivariant, monotone, $\I$-saturated, and $\Phi_\LL(\mathsf P(X))\subset\{\emptyset,X\}$. 

It is clear that the function $e_\I:\mathsf P(X)\to\mathsf P(X)$ defined by
$$e_\I(A)=\begin{cases}A&\mbox{if $A\in\pT^\I$},\\
\Phi_\LL(A)&\mbox{otherwise}
\end{cases}
$$
has properties (1)--(3) of Proposition~\ref{p16.1}. It is also clear that $e_\I=e_\I\circ e_\I$ is an idempotent. 

We claim that $e_\I\in\Enl(\mathsf P(X))$. By Corollary~\ref{c4.4}, we need to check that $e_\I$ is equivariant, monotone and symmetric. The equivariance of $e_\I$ follows from the equivariance of the maps $\Phi_\LL$ and $\id$. 

To show that $e_\I$ is monotone, take any two subsets $A\subset B$ of $X$ and consider four cases.
\smallskip

1) If $A,B\notin\pT^\I$, then $e_\I(A)=\Phi_\LL(A)\subset\Phi_\LL(B)=e_\I(B)$ by the monotonicity of the function representation $\Phi_\LL$ of the monotone family $\LL$.
\smallskip

2) If $A,B\in\pT^\I$, then $e_\I(A)=A\subset B=e_\I(B)$.
\smallskip

3) $A\in\pT^\I$ and $B\notin\pT^\I$. We claim that $B\in\LL$. Assuming that $B\notin\LL$ and applying Claim~\ref{cl12.2}, we get $B\notin\LL^\perp$. Then $B$ does not intersect some set $L\in\L$ and then $A\cap L=\emptyset$. It follows that the set $X\setminus A\supset L$ belongs to the maximal invariant $\I$-saturated linked system and so does the set $yA\supset_\I X\setminus A$ for some $y\in X$ (which exists as $A\in\pT^\I$). By the left-invariance of $\LL$, we get $A\in\LL$ which contradicts $X\setminus A\in\LL$ and the linkedness of $\LL$. This contradiction proves that $B\in\LL$. In this case $e_\I(A)=A\subset X=\Phi_\LL(B)=e_\I(B)$.
\smallskip

4) $A\notin \pT^\I$ and $B\in\pT^\I$. In this case we prove that $A\notin\LL$. Assuming conversely that $A\in\LL$, we get $B\in\LL$. Since $B\in\pT^\I$, there is a point $x\in X$ with $xB\subset_\I X\setminus B$. Since $\LL$ is left-invariant, monotone and $\I$-saturated, we conclude that $X\setminus B\in\LL$ which contradicts $B\in\LL$. Thus $A\notin\LL$ and $e_\I(A)=\Phi_\LL(A)=\emptyset\subset e_\I(B)$.
\smallskip

Finally, we show that the function $e_\I$ is symmetric. If $A\in \pT^\I$, then  $X\setminus A\in\pT^\I$ and then $e_\I(X\setminus A)=X\setminus A=X\setminus e_\I(A)$. 

Next, assume that $A\notin\pT^\I$. If $A\in\LL$, then $X\setminus A\notin \LL$ by the linkedness of $\LL$. In this case $e_\I(X\setminus A)=\emptyset=X\setminus X=X\setminus e_\I(A)$. 

If $A\notin\LL$, then by Claim~\ref{cl12.2}, $A\notin\LL^\perp$ and thus $A$ is disjoint with some set $L\in\LL$, which implies that $X\setminus A\in\LL$. Then $e_\I(X\setminus A)=\Phi_\LL(X\setminus A)=X=X\setminus\emptyset=X\setminus\Phi_\LL(A)=X\setminus e_\I(A)$.
\end{proof}

Our second special idempotent depends on a subfamily $\wtd{\Tau}$ of the family
$$\wht{\Tau}=\{A\in\Tau:\Fix^-(A)\in\wht{\K}\}$$of twin sets with maximal 2-cogroup.

\begin{theorem}\label{t16.3} If the ideal $\I$ is twinic, then for any $\I$-independent $\wht{\K}$-covering subfamily $\wtd{\mathsf T}\subset \wht{\Tau}$ there is an idempotent ${e_{\wtd{\Tau}}}\in\Enl^\I(\mathsf P(X))$ such that
\begin{enumerate}
\item[\textup{(1)}] ${e_{\wtd{\Tau}}}(\mathsf P(X)\setminus \Tau^\I)\subset\{\emptyset,X\}$;
\item[\textup{(2)}] ${e_{\wtd{\Tau}}}(\Tau^\I)=\wtd{\Tau}$;
\item[\textup{(3)}] ${e_{\wtd{\Tau}}}|\{\emptyset,X\}\cup\wtd{\Tau}=\id$.
\end{enumerate}
\end{theorem}

\begin{proof} Let $e_\I:\mathsf P(X)\to\{\emptyset,X\}\cup\pT^\I$ be the idempotent from Proposition~\ref{p16.1}. Since the ideal $\I$ is twinic, $\pT^\I=\Tau^\I$. The idempotent ${e_{\wtd{\Tau}}}$ will be defined as the composition ${e_{\wtd{\Tau}}}=\varphi\circ e_\I$ where
$\varphi:\{\emptyset,X\}\cup\Tau^\I\to\{\emptyset,X\}\cup\wtd{\Tau}$ is an equivariant $\I$-saturated function such that
\begin{enumerate}
\item[\textup{(1)}] $\varphi\circ\varphi=\varphi$;
\item[\textup{(2)}] $\varphi|\{\emptyset,X\}\cup\wtd{\Tau}=\id$;
\item[\textup{(3)}] $\varphi(\Tau^\I)\subset\wtd{\Tau}$;
\item[\textup{(4)}] $\IFix^-(A)\subset\Fix^-(\varphi(A))$ for all $A\in\Tau^\I$.  
\end{enumerate}

To construct such a function $\varphi$, consider the family $\F$ of all possible functions $\varphi:D_\varphi\to\{\emptyset,X\}\cup\wtd{\Tau}$ such that 
\begin{itemize}
\item[\textup{(a)}] $\{\emptyset,X\}\cup\wtd{\Tau}\subset D_\varphi\subset\{\emptyset,X\}\cup\Tau^\I$;
\item[\textup{(b)}] the set $D_\varphi$ is left-invariant;
\item[\textup{(c)}] $\varphi$ is equivariant and $\I$-saturated;
\item[\textup{(d)}] $\varphi|\{\emptyset,X\}\cup\wtd{\Tau}=\id$;
\item[\textup{(e)}] $\IFix^-(A)\subset\Fix^-(\varphi(A))$ for all $A\in D_\varphi$.
\end{itemize}
 The family $\F$ is partially ordered by the relation 
$\varphi\le\psi$ defined by $\psi|D_\varphi=\varphi$.

The set $\F$ is not empty because it contains the identity function $\id$ of $\{\emptyset,X\}\cup\wtd{\Tau}$, which is $\I$-saturated because of the $\I$-independence of the family $\wtd{\Tau}$. 
By Zorn's Lemma, the family $\F$ contains a maximal element $\varphi:D_\varphi\to\{\emptyset,X\}\cup\wtd{\Tau}$. We claim that $D_\varphi=\{\emptyset,X\}\cup\Tau^\I$. Assuming the converse, fix a set $A\in \Tau^\I\setminus D_\varphi$ and define a family $D_\psi=D_\varphi\cup\{xA:x\in X\}$. Next, we shall extend the function $\varphi$ to a function $\psi:D_\psi\to\{\emptyset,X\}\cup\wtd{\Tau}$. 
We consider two cases.

1) Assume that $A=_\I B$ for some $B\in D_\varphi$. Then also $xA=_\I xB$ for all $x\in X$. In this case we define the function $\psi:D_\psi\to\{\emptyset,X\}\cup\wtd{\Tau}$ assigning to each set $C\in D_\psi$ the set $\varphi(D)$  where $D\in D_\varphi$ is any set with $D=_\I C$. It can be shown that the function $\psi:D_\psi\to\{\emptyset,X\}\cup\wtd{\Tau}$ 
belongs to the family $\F$, which contradicts the maximality of $\varphi$.

2) Assume that $A\ne_\I B$ for all $B\in D_\varphi$. By Proposition~\ref{p7.3}, the 2-cogroup $\IFix^-(A)$ lies in a maximal 2-cogroup $K\in\wht{\K}$.
Since the family  $\wtd{\Tau}$ is $\wht{\K}$-covering, there is a twin set $B\in\wtd{T}$ such that $\Fix^-(B)=K$.  In this case define the function $\psi:D_\psi\to\wht{\Tau}$ by the formula
$$\psi(C)=\begin{cases} \varphi(C)&\mbox{if $C\in D_\varphi$};\\
xB&\mbox{if $C=_\I xA$ for some $x\in X$}.
\end{cases}
$$
If $xA=_\I yA$ for some $x,y\in X$, then $y^{-1}x\in \IFix^-(A)\subset K=\Fix^-(B)$ and thus $xB=yB$, which means that the function $\psi$ is well-defined and $\I$-saturated. Also it is clear that $\psi$ is equivariant and hence belongs to the family $\F$, which is forbidden by the maximality of $\varphi$.

Thus the maximal function $\varphi$ is defined on $D_\varphi=\{\emptyset,X\}\cup\Tau^\I$ and we can put ${e_{\wtd{\Tau}}}=\varphi\circ e_\I$ where $e_\I:\mathsf P(X)\to\{\emptyset,X\}\cup\pT^\I=\{\emptyset,X\}\cup\Tau^\I$ is the idempotent constructed in Proposition~\ref{p16.1}. It follows from the properties of the functions $\varphi$ and $e_\I$ that the function ${e_{\wtd{\Tau}}}$ is equivariant and $\I$-saturated. Since the ideal $\I$ is twinic, the family $\Tau^\I=\pT^\I$ is $\I$-incomparable (by Proposition~\ref{p13.1}) and hence the monotonicity of the function $\varphi$ follows automatically from its $\I$-saturated property.
Then ${e_{\wtd{\Tau}}}$ is monotone as the composition of two monotone functions.
By Corollary~\ref{c4.4}, $e_{\wtd{\Tau}}\in\Enl(\mathsf P(X))$.
\end{proof}

Theorem~\ref{t16.3} and Proposition~\ref{p13.5} imply:

\begin{corollary}\label{c16.4} If the ideal $\I$ is twinic, then for each minimal $\wht{\K}$-covering family $\wtd{\Tau}\subset\wht{\Tau}$ there is  an idempotent ${e_{\wtd{\Tau}}}\in\Enl^\I(\mathsf P(X))$ such that
\begin{enumerate}
\item[\textup{(1)}] ${e_{\wtd{\Tau}}}(\mathsf P(X)\setminus \Tau^\I)\subset\{\emptyset,X\}$;
\item[\textup{(2)}] $e_{\wtd{\Tau}}(\Tau^\I)=\wtd{\Tau}$;
\item[\textup{(3)}] ${e_{\wtd{\Tau}}}|\{\emptyset,X\}\cup\wtd{\Tau}=\id$.
\end{enumerate}
\end{corollary}

\section{The minimal ideal of  the semigroups 
$\lambda(X)$ and $\Enl(\mathsf P(X))$}\label{s17}

In this section we apply Corollary~\ref{c16.4} to describe the structure of the minimal ideals the semigroups $\lambda(X)$ and $\Enl(\mathsf P(X))$. 

\begin{theorem}\label{t17.1} For a twinic  group $X$ a function $f\in\Enl(\mathsf P(X))$ belongs to the minimal ideal $\IK(\Enl(\mathsf P(X))$ of the semigroup $\Enl(P(X))$ if and only if the following two conditions hold:
\begin{enumerate}
\item[\textup{(1)}] the family $f(\wht{\Tau})$ is minimal $\wht{\K}$-covering;
\item[\textup{(2)}] $f(\mathsf P(X))\subset\{\emptyset,X\}\cup f(\wht\Tau)$.
\end{enumerate}
\end{theorem}

\begin{proof} Let $\wtd{\Tau}\subset\wht{\Tau}$ be a minimal $\wht\K$-covering left-invariant family and $e_{\wtd{\Tau}}\in\Enl(\mathsf P(X))$ be an idempotent satisfying the conditions (1)--(3) of in Corollary~\ref{c16.4}. By Propositions~\ref{p13.1} and \ref{p13.5}, the family $\wtd{\Tau}$ is $\I$-incomparable and $\I$-independent for any twinic ideal $\I$ on $X$.
\smallskip

To prove the ``if'' part of the theorem, assume that $f$ satisfies the conditions (1), (2).
To show that $f$ belongs to the minimal ideal $\IK(\Enl(\mathsf P(X)))$, it suffices for each $g\in\Enl(\mathsf P(X))$ to find $h\in\Enl(\mathsf P(X))$ such that $h\circ g\circ f=f$. 

The minimality and the left-invariance of the $\wht{\K}$-covering subfamily $f(\wht{\Tau})$ imply that  the equivariant function $\psi=e_{\wtd{\Tau}}\circ g|\{\emptyset,X\}\cup f(\wht{\Tau}):\{\emptyset,X\}\cup f(\wht{\Tau})\to\{\emptyset,X\}\cup\wtd{\Tau}$ is bijective. So, we can consider the inverse function $\psi^{-1}:\{\emptyset,X\}\cup \wtd{\Tau}\to\{\emptyset,X\}\cup f(\wht{\Tau})$ such that $\psi^{-1}\circ \psi=\id|\{\emptyset,X\}\cup f(\wht{\Tau})$.
This function is equivariant, symmetric, and monotone because so is $\psi$ and the family $\wtd{\Tau}$ is $\I$-incomparable and $\I$-independent.

Then the function $\varphi=\psi^{-1}\circ e_{\wtd{\Tau}}:\mathsf P(X)\to\{\emptyset,X\}\cup f(\wht{\Tau})$ is well-defined and belongs to $\Enl^{\I}(\mathsf P(X))$ by Corollary~\ref{c4.4}. Since 
$$(\varphi\circ e_{\wtd{\Tau}})\circ g\circ f=\psi^{-1}\circ e_{\wtd\Tau}\circ e_{\wtd\Tau}\circ g\circ f=\psi^{-1}\circ e_{\wtd\Tau}\circ g\circ f=\psi^{-1}\circ \psi\circ f=f,$$ the function $f$ belongs to the minimal ideal of the semigroup $\Enl(\mathsf P(X))$. 
\smallskip

To prove the ``only if'' part, take any function $f\in \IK(\Enl(\mathsf P(X)))$ and for the idempotent $e_{\wtd{\Tau}}\in\Enl(\mathsf P(X))$ find a function $g\in\Enl(\mathsf P(X))$ such that $f=g\circ e_{\wtd{\Tau}}\circ f$.
Now the properties (1), (2) of the function $f$ follow from the corresponding properties of the idempotent $e_{\wtd{\Tau}}$.
\end{proof}

Since the superextension $\lambda(X)$ of a group $X$ is topologically isomorphic to the semigroup $\Enl(\mathsf P(X))$, Theorem~\ref{t17.1} implies the following description of the minimal ideal $\IK(\lambda(X))$ of $\lambda(X)$.

\begin{corollary}\label{c13.2} For a twinic  group $X$ a maximal linked system $\LL\in\lambda(X)$ belongs to the minimal ideal $\IK(\lambda(X))$ of the superextension $\lambda(X)$ if and only if its function representation $\Phi_\LL$ satisfies two conditions:
\begin{enumerate}
\item[\textup{(1)}] the family $\Phi_\LL(\wht{\Tau})$ is minimal $\wht{\K}$-covering;
\item[\textup{(2)}] $\Phi_\LL(\mathsf P(X))\subset\{\emptyset,X\}\cup \Phi_\LL(\wht\Tau)$.
\end{enumerate}
\end{corollary}

\section{Minimal left ideals of superextensions of twinic groups}\label{s18}

After elaborating the necessary tools in Section~\ref{s4}-\ref{s17}, we now return to describing the structure of minimal left ideals of the superextension $\lambda(X)$ of a twinic group $X$. In this section we assume that $X$ is a group.

 Our first aim is to show that if $X$ is twinic, then the restriction operator $R_{\wht\Tau}:\Enl(\mathsf P(X))\to\Enl(\wht{\Tau})$ is injective on all minimal left ideals of the semigroup $\Enl(\mathsf P(X))$. Since $\wht{\Tau}=\bigcup_{K\in\wht{\K}}\Tau_{K}$, Proposition~\ref{p12.3} implies that the family $\wht{\Tau}$ is $\lambda$-invariant and hence $\Enl(\wht\Tau)$ is a compact right-topological semigroup. For each 
left-invariant ideal $\I$ on the group $X$ the semigroup $\Enl(\wht\Tau)$ contains a left ideal $\Enl^\I(\wht\Tau)$ consisting of all left-invariant  monotone $\I$-saturated functions, see Theorem~\ref{t4.8}. If  $\I$ is a twinic ideal with $\I\cap\wht\K=\emptyset$, then the family $\wht\Tau$ is $\I$-independent (see Proposition~\ref{p13.3}) and hence $\Enl^\I(\wht{\Tau})=\Enl(\wht{\Tau})$. 

\begin{proposition}\label{p18.1} If the group $X$ is twinic, then the restriction operator  $$R_{\wht{\Tau}}:\Enl(\mathsf P(X))\to\Enl(\wht{\Tau}),\;R_{\wht\Tau}:f\mapsto f|\wht\Tau,$$
 is injective on each minimal left ideal of the semigroup $\Enl(\mathsf P(X))$. If $\TI\cap\wht{\K}=\emptyset$, then for some idempotent $ e_{\wht\Tau}\in\Enl^\TI(\mathsf P(X))$  the restriction $R_{\wht{\Tau}}|\Enl(\mathsf P(X))\circ e_{\wht\Tau}$ is a topological isomorphism between the principal left ideal $\Enl(\mathsf P(X))\circ e_{\wht\Tau}$ and $\Enl(\wht{\Tau})$.
\end{proposition}

\begin{proof} Let $\wtd{\Tau}\subset\wht{\Tau}$ be any minimal $\K$-covering left-invariant subfamily. By Propositions~\ref{p13.5}, the family $\wtd{\Tau}$ is $\I$-independent. By Theorem~\ref{t16.3},
there is an idempotent $e_{\wtd{\Tau}}\in\Enl^\I(\mathsf P(X))$ such that $e_{\wtd{\Tau}}(\mathsf P(X)\setminus\Tau^\I)\subset\{\emptyset,X\}$ and $e_{\wtd{\Tau}}(\Tau^\I)=\wtd{\Tau}\subset\wht{\Tau}$. The latter property of $e_{\wtd{\Tau}}$ implies that the restriction operator $R_{\wht{\Tau}}$ is injective on the principal left ideal $\Enl(\mathsf P(X))\circ e_{\wtd{\Tau}}$ and consequently, is 
injective on each minimal left ideal of the semigroup $\Enl(\mathsf P(X))$ according to Proposition~\ref{c2.2}.

If $\TI\cap\wht{\K}=\emptyset$, then by Proposition~\ref{p13.3} the family  $\wht{\Tau}$ is $\TI$-independent and we can repeat the above argument for the idempotent $e_{\wht{\Tau}}$.
\end{proof}

Now let us look at the structure of the semigroup $\Enl(\wht{\Tau})$.
Observe that $\wht{\Tau}=\bigcup_{[K]\in[\wht{\K}]}\Tau_{[K]}$, where $\Tau_{[K]}=\{A\subset X:\Fix^-(A)\in[K]\}$, $[\wht\K]=\{[K]:K\in\wht\K\}$ and $[K]=\{xKx^{-1}:x\in X\}$ for $K\in\wht\K$. 

It follows that the restriction operators $R_{\Tau_{[K]}}:\Enl(\wht{\Tau})\to\Enl(T_{[K]})$, $[K]\in[\wht{\K}]$, compose an injective semigroup homomorphism
$$R_{\Tau_{\![\wht\K]}}:\Enl(\wht{\Tau})\to\prod_{[K]\in[\wht{\K}]}\Enl(\Tau_{[K]}),\;\;R_{\Tau_{\![\wht\K]}}:\varphi\mapsto (\varphi|\Tau_{[K]})_{[K]\in[\wht{K}]}.$$

Theorem~\ref{t11.1} implies 

\begin{lemma}\label{l18.2} For any twinic ideal $\I$ on the group $X$ we get $R_{[\wht\K]}(\Enl^\I(\wht{\Tau}))=
\prod_{[K]\in[\wtd{\K}]}\Enl^\I(\Tau_{[K]}).$
\end{lemma}

Next, we study  the structure of the semigroups $\Enl^\I(\Tau_{[K]})$ for $[K]\in[\wht{\K}]$.

\begin{lemma}\label{l18.3} For any maximal 2-cogroup $K\in\wht\K$ the restriction map
$$R_{\Tau_K}:\Enl(\Tau_{[K]})\to\Enl(\Tau_K),\;\;R_{\Tau_K}:\varphi\mapsto \varphi|\Tau_K,$$is a topological isomorphism.
\end{lemma}

\begin{proof} Because of the compactness of the semigroup $\Enl(\Tau_{[K]})$ it suffices to check that the restriction operator $R_{\Tau_K}:\Enl^\I(\Tau_{[K]})\to \Enl^\I(\Tau_K)$
is one-to-one. Given two distinct functions $f,g\in\Enl^\I(\Tau_{[K]})$ find a twin set $A\in\Tau_{[K]}$ such that $f(A)\ne g(A)$. Since $\Fix^-(A)\in[K]$, there is a point $x\in X$ such that $\Fix^-(xA)=x\,\Fix^-(A)\,x^{-1}=K$. By Proposition~\ref{p5.4}, $xA\in\Tau_K$ and $f(xA)=xf(A)\ne xg(A)=g(xA)$ witnessing that $f|\Tau_K\ne g|\Tau_K$.
\end{proof}

A subfamily $\wtd\K\subset\wht\K$ is called a {\em $[\wht\K]$-selector} if $\wtd\K$ has one-point intersection with each orbit $[K]=\{xKx^{-1}:x\in X\}$, $K\in\wht\K$. In the following theorems we assume that $\wtd\K\subset\wht\K$ is a $[\wht\K]$-selector.

All preceding discussion culminates in the following theorem, which can be considered as the main result of this paper.

\begin{theorem} Given a $[\wht\K]$-selector $\wtd\K\subset\wht\K$, consider the operator $$R_{\wtd\K}:\Enl(\mathsf P(X))\to\prod_{K\in\wtd\K}\End(\Tau_K),\;\;
R_{\wtd\K}:f\mapsto (f|\Tau_K)_{K\in\wtd\K}.$$If $\I$ is a left-invariant twinic ideal on $X$, then 
\begin{enumerate}
\item[\textup{(1)}] $R_{\wtd\K}\big(\Enl^\I(\mathsf P(X))\big)=\prod_{K\in\wtd\K}\End^\I(\Tau_K)$;
\item[\textup{(2)}] the operator $R_{\wtd\K}$ maps isomorphically each minimal left ideal of the semigroup $\Enl(\mathsf P(X))$ onto some minimal left ideal of the semigroup $\prod_{K\in\wtd\K}\End(\Tau_K)$. 
\item[\textup{(3)}] If $\I\cap\wht\K=\emptyset$, then for some idempotent $\hat e\in\Enl^\I(\mathsf P(X))$  and the principal left ideal $\mathsf L_{\hat e}=\Enl(\mathsf P(X))\circ\hat e$ the restriction 
$$R_{\wtd\K}|\mathsf L_{\hat e}:\mathsf L_{\hat e}\to \prod\limits_{K\in\wtd\K}\End(\Tau_K)$$ is a topological isomorphism.
\end{enumerate}
\end{theorem}

\begin{proof} Write the operator $R_{\wtd\K}$ as the composition $R_{\wtd\K}=R^{\wht T}_{\wtd\K}\circ R_{\wht\Tau}$ of two operators:
$$R_{\wht\Tau}:\Enl(\mathsf P(X))\to\Enl(\wht{\Tau}),\;\; R_{\wht\Tau}:f\mapsto f|\wht\Tau,$$and
$$R^{\wht\Tau}_{\wtd\K}:\Enl(\wht{\Tau})\to\prod_{K\in\wtd\K}\Enl(\Tau_K),\;\;R_{\wtd\K}^{\wht\Tau}:f\mapsto (f|\Tau_K)_{K\in\wtd\K}.$$
By Lemma~\ref{l18.3}, the operator $R^{\wht\Tau}_{\wtd\K}$ is injective.
\smallskip

1. It follows from Lemmas~\ref{l18.2} and \ref{l18.3} that 
 $$R_{\wtd\K}(\Enl^\I(\mathsf P(X))=R^{\wht\Tau}_{\wtd\K}(\Enl^\I(\wht{\Tau}))=\prod_{K\in\wtd\K}\End^\I(\Tau_K).$$

2. To prove the second item, fix any function $f\in \IK(\Enl(\mathsf P(X)))$ and consider the minimal left ideal $\mathsf L_f=\Enl(\mathsf P(X))\circ f$. We need to show that $R_{\wtd\K}(\mathsf L_f)$ is a minimal left ideal in $\prod_{K\in\wtd\K}\End(\Tau_K)$.

For this pick any function $g\in \IK(\Enl^\I(\mathsf P(X)))$ and consider the minimal left ideal $\mathsf L_g=\Enl(\mathsf P(X))\circ g=\Enl^\I(\mathsf P(X))\circ g$.

By Proposition~\ref{p18.1}, the operator $R_{\wht\Tau}:\Enl(\mathsf P(X))\to\Enl(\wht{\Tau})$ is injective on each minimal left ideal. Consequently, the operator $R_{\wtd\K}=R_{\wtd\K}^{\wht\Tau}\circ R_{\wht\Tau}$ also is injective on each minimal left ideal of the semigroup $\Enl(\mathsf P(X))$. In particular, $R_{\wtd\K}$ is injective on the minimal left ideals $\mathsf L_f$ and $\mathsf L_g$. Since $\mathsf L_g$ is a minimal left ideal of the semigroup $\Enl^\I(\mathsf P(X))$ its image $R_{\wht\Tau}(\mathsf L_g)$ is a minimal left ideal of the semigroup $\Enl^\I(\wht\Tau)$. 
By Lemmas~\ref{l18.2} and \ref{l18.3}, $R^{\wht\Tau}_{\wtd\K}$ maps isomorphically the semigroup $\Enl^\I(\wht\Tau)$ onto $\prod_{K\in\wtd\K}\End^\I(\Tau_K)$, the image $R_{\wtd\K}(\mathsf L_g)$ is a minimal left ideal of the semigroup 
$\prod_{K\in\wtd\K}\End^\I(\Tau_K)$. Since the latter semigroup is a left ideal in $\prod_{K\in\wtd\K}\End(\Tau_K)$, the image $R_{\wtd\K}(\mathsf L_g)$ remains a minimal left ideal of the semigroup $\prod_{K\in\wtd\K}\End(\Tau_K)$. 
This minimal left ideal is equal to the product $\prod_{K\in\wtd\K}\mathsf L_{g_K}$ where $g_K=g|\Tau_K$ and $\mathsf L_{g_K}=\End(\Tau_K)\circ g_K$.
Because of the compactness of $\mathsf L_g$, the operator $R_{\wtd\K}$ maps isomorphically the minimal left ideal $\mathsf L_g$ onto the minimal left ideal $\prod_{K\in\wtd\K}\mathsf L_{g_K}$ of the semigroup $\prod_{K\in\wtd\K}\End(\Tau_K)$.

Now let us look at the minimal left ideal $\mathsf L_f$. By Proposition~\ref{p2.1}, the right shift $r_f:\mathsf L_g\to\mathsf L_f$, $r_f:h\mapsto h\circ f$, is a homeomorphism. So, there is a function $\gamma\in \Enl(\mathsf P(X))$ such that $f=\gamma\circ g\circ f$. 

For every $K\in\wtd\K$ consider the restrictions $f_K=f|\Tau_K$ and $\gamma_K=\gamma|\Tau_K$, which belong to the semigroup $\End(\Tau_K)$.
It follows from $f=\gamma\circ g\circ f$ that $f_K=\gamma_K\circ g_K\circ f_K$. Since $g_K\in \IK(\End(\Tau_K))$, we conclude that $f_K$ also belongs to the minimal ideal $\IK(\End(\Tau_K))$.

Then $\mathsf L_{f_K}=\End(\Tau_K)\circ f_K$ and $\mathsf L_{g_K}=\End(\Tau_K)\circ g_K$ are minimal left ideals in $\End(\Tau_K)$.
By Proposition~\ref{p2.1}, the right shift $r_{f_K}:\mathsf L_{g_K}\to\mathsf L_{f_K}$, $r_{f_K}:h\mapsto h\circ f_K$, is a homeomorphism. The homeomorphisms $r_{f_K}$, $K\in\wtd\K$, compose a homeomorphism
$$r_{f_{\wtd\K}}:\prod_{K\in\wtd\K}\mathsf L_{g_K}\to \prod_{K\in\wtd\K}\mathsf L_{f_K},\;\;r_{f_{\wtd\K}}:(h_K)_{K\in\wtd\K}\mapsto (h_K\circ f_K)_{K\in\wtd\K}.$$
Now consider the commutative diagram 
$$\xymatrix{
\mathsf L_f\ar[r]^-{R_{\wtd\K}|\mathsf L_f}&\prod_{K\in\wtd\K}\mathsf L_{f_K}\\
\mathsf L_g\ar[u]^{r_f}\ar[r]_-{R_{\wtd\K}|\mathsf L_g}&\prod_{K\in\wtd\K}\mathsf L_{g_K}\ar[u]_{r_{f_{\wtd\K}}}
}$$Since the maps $r_f$, $r_{f_{\wtd\K}}$, and $R_{\wtd\K}|\mathsf L_g$ are homeomorphisms, so is the map $R_{\wtd\K}|\mathsf L_f$. Consequently, the operator $R_{\wtd\K}$ maps isomorphically the minimal left ideal $\mathsf L_f=\Enl(\mathsf P(X))\circ f$ onto the minimal left ideal $\prod_{K\in\wtd\K}\mathsf L_{f_K}=\prod_{K\in\wtd\K}\End(\Tau_K)\circ(f|\Tau_K)$ of the semigroup $\prod_{K\in\wtd\K}\End(\Tau_K)$.
\smallskip

3. Assume that $\I\cap\wht\K=\emptyset$. In this case $\End^\I(\wht\Tau)=\End(\wht\Tau)$ by Proposition~\ref{p13.3}. By Proposition~\ref{p18.1}, for some idempotent $\hat e\in\Enl^\I(\mathsf P(X))$ the operator $R^{\wht\Tau}_{\wtd\Tau}$ maps isomorphically the principal left ideal $\mathsf L_{\hat e}=\Enl(\mathsf P(X))\circ\hat e$ onto $\Enl^\I(\wht\Tau)=\Enl(\wht\Tau)$. By Lemma~\ref{l18.2}, the operator $R^{\wht\Tau}_{\wtd\K}:\Enl^\I(\wht\Tau)\to\prod_{K\in\wtd\K}\End^\I(\Tau_K)=\prod_{K\in\wtd\K}\End(\Tau_K)$ is an isomorphism. So, $R_{\wtd\K}$ maps isomorphically the principal left ideal $\mathsf L_{\hat e}$ onto $\prod_{K\in\wtd\K}\End(\Tau_K)$.
\end{proof}

Since the function representation $\Phi:\lambda(X)\to\Enl(\mathsf P(X))$, $\Phi:\LL\mapsto\Phi_\LL$, is a topological isomorphism, the preceding theorem implies:

\begin{corollary}\label{c18.5} Given a $[\wht\K]$-selector $\wtd\K\subset\wht\K$, consider the continuous semigroup homomorphism $$\Phi_{\wtd\K}:\lambda(X)\to\prod_{K\in\wtd\K}\End(\Tau_K),\;\;
\Phi_{\wtd\K}:\LL\mapsto (\Phi_\LL|\Tau_K)_{K\in\wtd\K}.$$If the group $X$ is twinic, then 
\begin{enumerate}
\item[\textup{(1)}] $\Phi_{\wtd\K}(\Enl^\TI(\mathsf P(X))=\prod_{K\in\wtd\K}\End^\TI(\Tau_K)$;
\item[\textup{(2)}] the homomorphism $\Phi_{\wtd\K}$ maps isomorphically each minimal left ideal of the semigroup $\lambda(X)$ onto some minimal left ideal of the semigroup $\prod_{K\in\wtd\K}\End(\Tau_K)$. 
\item[\textup{(3)}] If $\TI\cap\wht\K=\emptyset$, then for some idempotent $\E\in\lambda^\TI(X)$  and the principal left ideal $\mathsf L_{\E}=\lambda(X)\circ\E$ the restriction 
$$\Phi_{\wtd\K}|\mathsf L_{\E}:\mathsf L_{\E}\to \prod\limits_{K\in\wtd\K}\End(\Tau_K)$$ is a topological isomorphism.
\end{enumerate}
\end{corollary} 
 
\begin{corollary}\label{c18.6} If the group $X$ is twinic, then each minimal left ideal of $\lambda(X)$ is topologically isomorphic to a minimal left ideal of $\prod_{K\in\wtd\K}\End(\Tau_K)$ and each minimal left ideal of $\prod_{K\in\wtd\K}\End^\TI(\Tau_K)$ is topologically isomorphic to a minimal left ideal of $\lambda^\TI(X)$.
\end{corollary}

\begin{proof} Corollary~\ref{c18.5}(2) implies that each minimal left ideal of $\lambda(X)$ is topologically isomorphic to a minimal left ideal of $\prod_{K\in\wtd\K}\End(\Tau_K)$. Now assume that $\mathsf L$ is a minimal left ideal of the semigroup $\prod_{K\in\wtd\K}\End^\TI(\Tau_K)$. It follows from Corollary~\ref{c18.5}(1) that the preimage $\Phi_{\wtd\K}^{-1}(\mathsf L)$ is a left ideal in $\lambda^\TI(X)$ and hence a left ideal in $\lambda(X)$. This left ideal contains some minimal left ideal $\mathsf L_\lambda$ whose image $\Phi_{\wtd\K}$ coincides with $\mathsf L$ (being a left ideal in $\mathsf L$). By Corollary~\ref{c18.5}(2), the map $\Phi_{\wtd\K}|\mathsf L_\lambda:\mathsf L_\lambda\to\mathsf L$ is injective and by the compactness of $\mathsf L_\lambda$ is a topological isomorphism.
\end{proof}  

\begin{theorem}\label{t18.7} Let $X$ be a twinic  group, $\wtd\K\subset\wht{\K}$ be a $[\wht\K]$-selector, and $\mathcal E\in\lambda(X)$ be a minimal idempotent.
\begin{enumerate}
\item[\textup{(1)}] The maximal subgroup $\mathsf H_\E=\mathcal E\circ\lambda(X)\circ\E$ has the following properties:
\begin{enumerate}
\item[\textup{(a)}] $\mathsf H_\E$ is algebraically isomorphic to $\prod_{K\in\wtd\K}\HH(K)$;
\item[\textup{(b)}] $\mathsf H_\E$ is topologically isomorphic to $\prod_{K\in\wtd\K}\HH(A_K)$ for any twin sets $A_K\in\Phi_\E(\Tau_K)$, $K\in\wtd\K$;
\item[\textup{(c)}] $\mathsf H_\E$ is a compact topological group if and only if $\HH(K)$ is finite for every $K\in\wht\K$.
\end{enumerate}
\item[\textup{(2)}] The minimal left ideal $\mathsf L_\E=\lambda(X)\circ\E$  has the following properties:
\begin{enumerate}
\item[\textup{(d)}] $\mathsf L_\E$ is topologically isomorphic to the minimal left ideal $\prod_{\in\wtd\K}\End(\Tau_K)\circ(\Phi_\E|\Tau_K)$;
\item[\textup{(e)}] $\mathsf L_\E$ is homeomorphic to $\prod_{K\in\wtd\K}\Tau_K$, which is homeomorphic to the Cantor discontinuum $\prod_{K\in\wtd\K}2^{X/K^\pm}$;
\item[\textup{(f)}] $\mathsf L_\E$ is algebraically isomorphic to $\prod_{K\in\wtd\K}\HH(K)\times[\Tau_K]$ where the orbit space $[\Tau_K]$ of the $\HH(K)$-act $\Tau_K$ is endowed with the left-zero-multiplication;
\item[\textup{(g)}] $\mathsf L_\E$ a topological semigroup iff $\mathsf L_\E$ is a semitopological semigroup iff each restriction 
$\Phi_\EE|\Tau_K$, $K\in\wtd\K$, is a continuous function iff  the maximal subgroup $\mathsf H_\E$ and the idempotent band $E(\mathsf L_\E)$ of $\mathsf L_\E$ are compact iff $\mathsf L_\E$ is topologically isomorphic to  
$\prod_{K\in\wtd\K}\HH(K)\times[\Tau_K]$.
\end{enumerate}
\end{enumerate}
\end{theorem}

\begin{proof} Let $\Phi_\EE\in\Enl(\mathsf P(X))$ be the function representation of the minimal idempotent $\EE\in \IK(\lambda(X))$. For every $K\in\wtd\K$ let $f_K=\Phi_\EE|\Tau_K$ and $\mathsf L_{f_K}=\End(\Tau_K)\circ f_K$ be the principal left ideal in $\End(\Tau_K)$, generated by the function $f_K$. By Corollary~\ref{c18.5}, the minimal left ideal $\mathsf L_\EE$ is topologically isomorphic to a minimal left ideal of the semigroup $\prod_{K\in\wtd\K}\End(\Tau_K)$. This minimal ideal contains $(f_K)_{K\in\wtd\K}$ and hence is equal to the product $\prod_{K\in\wtd\K}\mathsf L_{f_K}$. This proves the statement (d) of the theorem. Now all the other statements follow from Theorems~\ref{t14.1} and \ref{t14.3}.
\end{proof}

Theorem~\ref{t18.7}(b) is completed by the following theorem.

\begin{theorem}\label{t18.8} Let $\wtd\K\subset\wht\K$ be a $[\wht\K]$-selector. If the group $X$ is twinic, then for any twin sets $A_K\in\Tau_K$, $K\in\wtd\K$, the minimal ideal $\IK(\lambda(X))$ of $\lambda(X)$ contains a maximal subgroup $\mathsf H_\E$, which is topologically isomorphic to $\prod_{K\in\wtd\K}\HH(A_K)$.
\end{theorem}

\begin{proof} In the semigroup $\prod_{K\in\wtd\K}\IK(\End^\TI(\Tau_K))$ choose a sequence of functions $(f_K)_{K\in\wtd\K}$ such that $f_K(\Tau_K)\subset\lfloor A_K\rfloor$ for all $K\in\wtd\K$. This can be done in the following way. For every $K\in\wtd\K$ first choose any minimal idempotent $g_K\in \IK(\End^\TI(\Tau_K))$. By Theorem~\ref{t14.1}(5), $g_K(\Tau_K)\subset\lfloor B_K\rfloor$ for some twin set $B\in\Tau_K$. Since $\Tau_K$ is a free $\HH(K)$-act, we can choose an equivariant function $\varphi:\lfloor B_K\rfloor \to\lfloor A_K\rfloor$. Then the composition $f_K=\varphi\circ g_K$ is $\TI$-saturated and has the required property: $f_K(\Tau_K)\subset\lfloor A_K\rfloor$.

Consider the minimal left ideal $\mathsf L_{f_{\wtd\K}}=\prod_{K\in\wtd\K}\End(\Tau_K)\circ f_K$ and let $\mathsf L=\Phi_{\wtd\K}^{-1}(\mathsf L)\subset\lambda(X)$ be its preimage under the map $\Phi_{\wtd\K}$.  By Corollary~\ref{c18.5}(1), $\Phi_{\wtd\K}(\mathsf L)=\mathsf L_{f_{\wtd\K}}$. Now let $\IK(\mathsf L)$ be the minimal ideal of the left ideal $\mathsf L$. The image $\Phi_{\wtd\K}(\IK(\mathsf L))$, being  a left ideal in $\mathsf L_{f_{\wtd\K}}$, coincides with $\mathsf L_{f_{\wtd\K}}$. So, we can find a maximal linked system $\LL\in \IK(\mathsf L)$ such that $\Phi_{\wtd\K}(\LL)=(f_K)_{K\in\wtd\K}$. By Theorem~\ref{t18.7}(b), the maximal group $\mathsf H_\LL$ is topologically isomorphic to $\prod_{K\in\wtd\K}\HH(A_K)$.
\end{proof}

\begin{proposition}\label{p18.9} If $X$ is a twinic group, then each minimal left ideal of $\lambda(X)$ is a topological semigroup if and only if each maximal 2-cogroup $K\subset X$ has finite index in $X$.
\end{proposition}

\begin{proof} Let $\wtd\K\subset\wht\K$ be a $[\wht\K]$-selector.

If each maximal 2-cogroup $K\subset X$ has finite index in $X$, then the set $\Tau_K$ is finite and hence the semigroup $\End(\Tau_K)$ is finite. Consequently, $\prod_{K\in\wtd\K}\End(\Tau_K)$ is a compact topological semigroup and so is each minimal left ideal of this semigroup. By Corollary~\ref{c18.5}(2), each minimal left ideal of the semigroup $\lambda(X)$ is a topological semigroup.

If some maximal 2-cogroup in $X$ has infinite index, then Corollary~\ref{c14.6} implies that some minimal left ideal in $\prod_{K\in\wtd\K}\End^\TI(\Tau_K)$ is not a topological semigroup. By Corollary~\ref{c18.6}, some minimal left ideal in $\prod_{K\in\wtd\K}\End^\TI(\Tau_K)$ is not a topological semigroup.
\end{proof}

\begin{proposition}\label{p18.10} For a twinic group $X$ with $\TI\cap\wht\K=\emptyset$ the following conditions are equivalent:
\begin{enumerate}
\item[\textup{(1)}] some minimal left ideal of $\lambda(X)$ is a topological semigroup;
\item[\textup{(2)}] each maximal subgroup of $\lambda(X)$ is a topological group;
\item[\textup{(3)}] some maximal subgroup of $\lambda(X)$ is compact;
\item[\textup{(4)}] the characteristic group $\HH(K)$ is finite for each maximal 2-cogroup $K\subset X$.
\end{enumerate}
\end{proposition}

\begin{proof} $(1)\Ra(3)$ If some minimal left ideal of $\lambda(X)$ is a (necessarily compact) topological semigroup, then each maximal subgroup of this minimal ideal is a compact  topological group.

$(3)\Ra(4)$ If some maximal subgroup of $\IK(\lambda(X))$ is compact, then by Theorem~\ref{t18.7}(c), each characteristic group $\HH(K)$, $K\in\wht\K$, is finite.

$(4)\Ra(1)$ If each characteristic group $\HH(K)$, $K\in\wht\K$, is finite, then Proposition~\ref{p14.5}(1) and Theorem~\ref{t14.3} guarantee that the semigroup $\prod_{K\in\wtd\K}\End^\TI(\Tau_K)$ contains a minimal left ideal, which is a topological semigroup. By Corollary~\ref{c18.6}, this minimal left ideal is topologically isomorphic to some minimal left ideal of $\lambda(X)$.

$(4)\Ra(3)$ If each characteristic group $\HH(K)$, $K\in\wht\K$, is finite, then Theorem~\ref{t18.7}(c) guarantees that each maximal subgroup of $\IK(\lambda(X))$ is a compact topological group. 

$(3)\Ra(4)$. Assume that for some maximal 2-cogroup $K_\infty\in\wht\K$ the characteristic group $H(K_\infty)$ is infinite. Replacing $K_\infty$ by a conjugate cogroup, we can assume that $K_\infty\in\wtd\K$. By Theorem~\ref{t8.2}, the group $H(K_\infty)$ is isomorphic to $C_{2^\infty}$ or $Q_{2^\infty}$. In both cases, by Theorems~\ref{t9.5}, there is a twin set $A_\infty\in\Tau_{K_\infty}$ whose characteristic group $\HH(A_\infty)$ is not a topological group. Choose a minimal idempotent $f_{\wtd\K}=(f_K)_{K\in\wtd\K}\in \prod_{K\in\wtd\K}\End^\TI(\Tau_K)$ such that $f_{K_\infty}(\Tau_{K_\infty})\subset\lfloor A\rfloor$. 
For every $K\in\wtd\K$ choose any twin set $A_K\in f_K(\Tau_K)$ so that $A_K=A_{\infty}$ if $K=K_\infty$.

By Corollary~\ref{c18.6}, there is a minimal idempotent $\E\in\lambda^\TI(X)$ such that $\Phi_{\wtd\K}(\E)=f_{\wtd\K}$. By Theorem~\ref{t18.7}(b), the maximal subgroup $\mathsf H_\E=\lambda(X)\circ\E\circ\lambda(X)$ is topologically isomorphic to $\prod_{K\in \wtd\K}\HH(A_K)$. This subgroup is not a topological group as it contains an isomorphic copy of the right-topological group $\HH(A_\infty)$, which is not a topological group.  
\end{proof}

Now let us write Corollary~\ref{c18.5} and Theorem~\ref{t18.7} in a form, more convenient for calculations.

For every group $G\in\{C_{2^k},Q_{2^k}:k\in\IN\cup\{\infty\}\}$ denote by $q(X,G)$ the number of all orbits $[K]\in[\wht{\mathcal K}]$ such that for some (equivalently, every) 2-cogroup $K\in[K]$ the characteristic group $\HH(K)$ is isomorphic to $G$.
  
\begin{theorem}\label{t18.11} For each twinic group $X$ 
there is a cardinal $m$ such that
\begin{enumerate}
\item[\textup{(1)}]  each minimal left ideal of $\lambda(X)$ is algebraically  isomorphic to the semigroup 
$$2^m\times \prod_{1\le k\le \infty} C_{2^k}^{\;q(X,C_{2^k})}\times \prod_{3\le k\le\infty} Q_{2^k}^{\;q(X,Q_{2^k})}$$
where the Cantor discontinuum $2^{m}$ is endowed with the left zero multiplication;
\item[\textup{(2)}] If $q(X,C_{2^\infty})=q(X,Q_{2^\infty})=0$ and $\TI\cap\wht\K=\emptyset$, then some minimal left ideal of $\lambda(X)$ is topologically isomorphic to the compact topological semigroup 
$$2^m\times \prod_{1\le k<\infty} C_{2^k}^{\;q(X,C_{2^k})}\times \prod_{3\le k<\infty} Q_{2^k}^{\;q(X,Q_{2^k})}.$$
\item[\textup{(3)}] each maximal subgroup of the minimal ideal $\IK(\lambda(X))$ of $\lambda(X)$ is algebraically  isomorphic to the  group  
$$\prod_{1\le k\le\infty} C_{2^k}^{\;q(X,C_{2^k})}\times\prod_{3\le k\le\infty} Q_{2^k}^{\;q(X,Q_{2^k})}.$$
\item[\textup{(4)}] If $q(X,C_{2^\infty})=q(X,Q_{2^\infty})=0$, then each maximal subgroup of the minimal ideal $\IK(\lambda(X))$ of $\lambda(X)$ is topologically isomorphic to the compact topological group  
$$\prod_{1\le k<\infty} C_{2^k}^{\;q(X,C_{2^k})}\times\prod_{3\le k<\infty} Q_{2^k}^{\;q(X,Q_{2^k})}.$$
\end{enumerate} 
\end{theorem}

\begin{proof} 1. Fix any $[\wht\K]$-selection $\wtd\K\subset\wht\K$. For every $K\in\wtd\K$ put $m_K=|X/K^\pm|=\{K^\pm x:x\in X\}$ if the index of $K$ in $X$ is infinite and $m_K=2^{|X/K^\pm|}/|H(K)|$ otherwise. It follows that $|[\Tau_K]|=\frac{|\Tau_K|}{|H(K)|}=2^{m_K}$ and $[\Tau_K]$ is homeomorphic to the Cantor cube $2^{m_K}$ if the characteristic group $\HH(K)$ is finite. Let $m=\sum_{K\in\wtd\K}m_K$. 

By Theorem~\ref{t18.7}(f), any minimal left ideal $\mathsf L$ of $\lambda(X)$ is algebraically isomorphic to the semigroup $\prod_{K\in\wtd\K}\HH(K)\times [\Tau_K]$, where the orbit spaces $[\Tau_K]$ are endowed with the left zero multiplication. By Theorem~\ref{t8.2}, for every $K\in\wtd\K$ the characteristic group $\HH(K)$ is isomorphic to $C_{2^k}$ or $Q_{2^k}$ for some $k\in\IN\cup\{\infty\}$. According to the definition, for $k\in\{1,2\}$ the group $Q_{2^k}$ is isomorphic to the quaternion group $Q_8$.
By the definition of the number $q(X,G)$, for any group $G\in\{C_{2^k},Q_{2^k}:k\in\IN\cup\{\infty\}\}$ we get $q(X,G)=\{K\in\wtd\K:\HH(K)\cong G\}$ where $\cong$ denotes the (semi)group isomorphism. 

Now we see that 
$$\mathsf L\cong \prod_{K\in\wtd\K}\HH(K)\times [\Tau_K]
\cong\prod_{K\in\wtd\K}\HH(K)\times 2^{m_K}\cong
\prod_{1\le k\le\infty}C_{2^k}^{\;q(X,C_{2^k})}\times \prod_{3\le k\le\infty}Q_{2^k}^{\;q(X,Q_{2^k})}\times 2^m.$$
\smallskip

2. If $q(X,C_{2^\infty})=q(X,Q_{2^\infty})=0$, then for every $K\in\wtd\K$ the characteristic group $\HH(K)$ is finite and the orbit space $[\Tau_K]$ is a zero-dimensional compact Hausdorff space. In this case the space $\Tau_K$ is homeomorphic to $[\Tau_K]\times \HH(K)$.  If 
$K$ has finite index in $X$, then $\Tau_K$ has cardinality $2^{m_K}$ and hence is homeomorphic to the finite cube $2^{m_K}$. If $K$ has infinite index in $X$, then the space $\Tau_K$ is homeomorphic to the Cantor cube $2^{m_K}$ by Proposition~\ref{p12.1}.  It follows from the topological equivalence of $\Tau_K$ and $[\Tau_K]\times \HH(K)$ that $[\Tau_K]$ is a retract of the Cantor cube $\Tau_K$ and each point of $[\Tau_K]$ has character $m_K$. Now Shchepin's characterization of Cantor's cubes \cite{Shch} implies that the space $[\Tau_K]$ is homeomorphic to the Cantor cube $2^{m_K}$. Then the product $\prod_{K\in\wtd\K}[\Tau_K]$ is homeomorphic to the Cantor cube $2^m=\prod_{K\in\wtd\K}2^{m_K}$.

If $\TI\cap\wht\K=\emptyset$, then for every $K\in\wtd\K$ the endomorphism monoid $\End^\TI(\Tau_K)$ contains a continuous minimal idempotent $f_K$ according to Proposition~\ref{p14.5}(1). By Theorem~\ref{t14.3}, the minimal left ideal $\mathsf L_{f_K}=\End^\TI(\Tau_K)\circ f_K$ is topologically isomorphic to the compact topological semigroup $\HH(K)\times[\Tau_K]$ where the space $[\Tau_K]$ is endowed with the left zero multiplication. By (the proof of) Corollary~\ref{c18.6}, the minimal idempotent $\IK(\lambda(X))$ contains a maximal linked system $\E$ such that the minimal left ideal $\mathsf L_\E=\lambda(X)\circ\E$ is topologically isomorphic to the minimal left ideal $\prod_{K\in\wtd\K}\mathsf L_{f_K}$, which is topologically isomorphic to the compact topological semigroups
$$
\prod_{K\in\wtd\K}\HH(K)\times[\Tau_K]\mbox{ \ and \ }\prod_{1\le k< \infty}C_{2^k}^{\;q(X,C_{2^k})}\times \prod_{3\le k<\infty}Q_{2^k}^{\;q(X,Q_{2^k})}\times 2^m.$$
\smallskip

3. By Theorem~\ref{t18.7}(b) each maximal subgroup $H$ of the minimal ideal $\IK(\lambda(X))$ is topologically isomorphic to the right topological group $G=\prod_{K\in\wtd\K}\HH(A_K)$ for some twin sets $A_K\in\Tau_K$, $K\in\wtd\K$. The latter right-topological group is algebraically isomorphic to the group 
$$\prod_{1\le k\le \infty}C_{2^k}^{\;q(X,C_{2^k})}\times\prod_{3\le k\le\infty}Q_{2^k}^{\;q(X,Q_{2^k})}.$$

4. If $q(X,C_{2^\infty})=q(X,Q_{2^\infty})=0$, then all characteristic groups $\HH(K)$, $K\in\wtd\K$, are finite and then the group $G$ is topologically isomorphic to the compact topological group    
$$\prod_{1\le k<\infty}C_{2^k}^{\;q(X,C_{2^k})}\times\prod_{3\le k<\infty}Q_{2^k}^{\;q(X,Q_{2^k})}.$$
\end{proof}

\section{The structure of the superextensions of abelian groups}

In this section we consider the structure of the superextension of abelian groups. In this case some results of the preceding section can be simplified. In this section we assume that $X$ is an abelian group. By Theorem~\ref{t6.2}, $X$ is twinic and has trivial twinic ideal $\TI=\{\emptyset\}$. Let us recall that for a group $G$ by $q(X,G)$ we denote the number of orbits $[K]$, $K\in\wht\K$, such that for each $K\in[K]$ the characteristic group $\HH(K)$ is isomorphic to the group $G$. It is clear that $q(X,Q_{2^k})=0$ for all $k\in\IN\cup\{\infty\}$. On the other hand, the numbers $q(X,C_{2^k})$ can be easily calculated using the following proposition.

\begin{proposition}\label{p19.1} If $X$ is an abelian group, then for every $k\in\IN\cup\{\infty\}$ the cardinal $q(X,C_{2^k})$ is equal to the number of subgroups $H\subset X$ such that the quotient group $X/H$ is isomorphic to $C_{2^k}$. If $k\in\IN$, then
$$q(X,C_{2^k})=\frac{|\hom(X,C_{2^k})|-|\hom(X,C_{2^{k-1}})|}{2^{k-1}},$$
where $\hom(X,C_{2_k})$ is the group of all homomorphisms from $X$ to $C_{2^k}$. 
\end{proposition}

\begin{proof} Since each maximal 2-cogroup $K\subset X$ is normal,  each orbit $[K]\in[\wht{\mathcal K}]$ consists of a single maximal 2-cogroup. Consequently, $q(X,H)$ is equal to the number of maximal 2-cogroups $K\subset X$ whose characteristic group $\HH(K)=\Stab(K)/KK=X/KK$ is isomorphic to $H$. In other words, $q(X,H)$ equals to cardinality of the set
$$\wht{\K}_H=\{K\in\wht{\K}:X/KK\cong H\},$$
where $\cong$ stands for the group isomorphism.

Let $\mathcal G_H$ be the set of all subgroups $G\subset X$ such that the quotient group $X/G$ is isomorphic to $H$.
 The proposition will be proved as soon as we check that the function
$$f:\wht{\K}_H\to\mathcal G_H,\;\;f:K\mapsto KK,$$ is bijective.

To show that $f$ is injective, take any two maximal 2-cogroups $K,C$ with $KK=f(K)=f(C)=CC$. The quotient group $X/KK=X/CC$, being isomorphic to $H$, contains a unique element of order 2. Since $K$ and $C$ are cosets of order 2 in $X/KK=X/CC$, we conclude that $K=C$.

To show that $f$ is surjective, take any subgroup $G\in\mathcal G_H$. The quotient group $X/G$ is isomorphic to $H$ and thus contains a unique element $K$ of order 2. This element $K$ is a maximal 2-cogroup such that $f(K)=KK=G$.
\smallskip

To prove the second part of the proposition, observe that for a subgroup $H\subset X$ the quotient group $X/H$ is isomorphic to $C_{2^k}$ if and only if $H$ coincides with the kernel of some epimorphism $f:X\to C_{2^k}$. Observe that two epimorphisms $f,g:X\to C_{2^k}$ have the same kernel if and only if $g=\alpha\circ f$ for some automorphism of the group $C_{2^k}$. The group $C_{2^k}$ has exactly $2^{k-1}$ automorphisms determined by the image of the generator $a=e^{i\pi 2^{-k+1}}$ of $C_{2^k}$ in the 2-cogroup $aC_{2^{k-1}}$. A homomorphism $h:X\to C_{2^k}$ is an epimorphism if and only if $h(X)\not\subset C_{2^{k-1}}$. Consequently, $$q(X,C_{2^k})=\frac{|\hom(X,C_{2^k})\setminus \hom(X,C_{2^{k-1}})|}{2^{k-1}}.$$
\end{proof}

\begin{theorem}\label{t19.2} If $X$ is an abelian group, then  
\begin{enumerate}
\item[\textup{(1)}] each maximal subgroup of the minimal ideal $\IK(\lambda(X))$ is algebraically isomorphic to $\prod\limits_{1\le k\le\infty}C_{2^k}^{\;q(X,C_{2^k})}$.
\item[\textup{(2)}]  each minimal left ideal of $\lambda(X)$ is 
homeomorphic to the Cantor cube $(2^\w)^{q(X,C_{2^\infty})}\times\prod\limits_{1\le k<\infty}(2^{2^{k-1}})^{q(X,C_{2^k})}$ and is algebraically isomorphic to the semigroup $$\prod_{1\le k\le \infty}(C_{2^k}\times Z_k)^{q(X,C_{2^k})}$$where the cube $Z_k=2^{2^{k-1}-k}$ (equal to $2^{\w}$ if $k=\infty$) is endowed with the left zero multiplication.
\item[\textup{(3)}]  the semigroup $\lambda(X)$ contains a principal left ideal, which is algebraically isomorphic to the semigroup
$$\prod_{1\le k\le \infty}(C_{2^k}\wr Z_k^{Z_k})^{q(X,C_{2^k})}.$$
\end{enumerate}
\end{theorem}

\begin{proof} Since $X$ is abelian, each 2-cogroup $K\in\wht\K$ is normal in $X$ and hence has one-element orbit $[K]=\{xKx^{-1}:x\in X\}$. Then the family $\wht\K$ is a unique $[\wht\K]$-selector. Since $\Stab(K)=X$, the characteristic group $\HH(K)=\Stab(X)/KK$ is equal to the quotient group $X/KK$ and is abelian. By Theorem~\ref{t8.2}, $\HH(K)$ is isomorphic to the (quasi)cyclic 2-group $C_{2^n}$ for some $n\in\IN\cup\{\infty\}$. 
Consequently, $\wht\K=\bigcup_{1\le n\le\infty}\wht\K_n$ where $\wht\K_n$ is the subset of $\wht\K$ that consists of all maximal 2-cogroups $K$ whose characteristic group $\HH(K)=X/KK$ is isomorphic to the group $C_{2^k}$. By the definition of the numbers $q(X,G)$, we get $q(X,C_{2^n})=|\wht\K_n|$ for all $n\in\IN\cup\{\infty\}$.
\smallskip

1. By Theorem~\ref{t18.7}(a), each maximal subgroup of $\IK(\lambda(X))$ is algebraically isomorphic to $\prod_{K\in\wht\K}\HH(K)$ and the latter group is isomorphic to $$\prod_{1\le n\le\infty}C_{2^n}^{|\wht\K_n|}=  \prod_{1\le n\le\infty}C_{2^n}^{\;q(X,C_{2^n})}.$$
\smallskip

2. By Theorem~\ref{t18.7}(f), each minimal left ideal of $\lambda(X)$ is algebraically isomorphic to $\prod_{K\in\wht\K}(\HH(K)\times[\Tau_K])$ where the orbit spaces $[\Tau_K]$ are endowed with the left zero multiplication. Let $Z_\infty=2^\w$ and $Z_n=2^{2^{n-1}-n}$ for every $n\in\IN$. The cubes $Z_n$, $n\in\IN\cup\{\infty\}$, are endowed with the left zero multiplication. We claim that $|[\Tau_K]|=|Z_n|$ for each $n\in\IN\cup\{\infty\}$. If $n$ is finite, then $|X/K^\pm|=\frac12|X/KK|=\frac12|\HH(K)|=2^{n-1}$ and
$$|[\Tau_K]|=\frac{|\Tau_K|}{|\HH(K)|}=\frac{2^{X/K^\pm}}{2^n}=2^{2^{n-1}-n}=|Z_n|.$$
If $n$ is infinite, then the quotient group $\HH(K)=X/KK$ is isomorphic to $C_{2^\infty}$ and then $|X/K^\pm|=\w$. By Proposition~\ref{p12.1}, the space $\Tau_K$ is homeomorphic to the Cantor cube $2^\w$ and hence has cardinality of continuum. Since the group $\HH(K)$ is countable, the orbit space $[\Tau_K]$ also has cardinality of continuum and hence $|[\Tau_K]|=|2^\w|=|Z_\infty|$. 
Now we see that $\prod_{K\in\wht\K}(\HH(K)\times[\Tau_K])$ is algebraically isomorphic to $\prod_{1\le k\le\infty}(C_{2^k}\times Z_k)^{q(X,C_{2^k})}$.
\smallskip

3. Since $X$ has trivial twinic ideal, $\TI\cap\wht\K=\emptyset$ and by Corollary~\ref{c18.5}(c) and Theorem~\ref{t14.1}(3), the semigroup $\lambda(X)$ contains a principal left ideal that is algebraically isomorphic to the semigroup
$\prod_{K\in\wht\K}\HH(K)\wr [\Tau_K]^{[\Tau_K]}$, which is algebraically isomorphic to $\prod_{1\le k\le \infty}(C_{2^k}\wr Z_k^{Z_k})^{q(X,C_{2^k})}$.
\end{proof}

The following theorem characterizes the groups $X$ for which the algebraic isomorphism in Theorem~\ref{t19.2} are topological.

\begin{theorem}\label{t19.3} For an abelian group $X$ the following conditions are equivalent\textup{:}
\begin{enumerate}
\item[\textup{(1)}] The group $X$ admits no homomorphism onto the quasicyclic 2-group $C_{2^\infty}$.
\item[\textup{(2)}] Each maximal subgroup in the minimal ideal $\IK(\lambda(X))$ of $\lambda(X)$ is topologically isomorphic to the compact topological group
$\prod\limits_{k\in\IN}C_{2^k}^{\;q(X,C_{2^k})}$.
\item[\textup{(3)}] Each maximal subgroup in $\IK(\lambda(X))$ is a topological group.
\item[\textup{(4)}] Some maximal subgroup of $\IK(\lambda(X))$ is compact.
\item[\textup{(5)}] Each minimal left ideal of $\lambda(X)$ is topologically isomorphic to the compact topological semigroup
$$\prod\limits_{k\in\IN}(C_{2^k}\times Z_k)^{q(X,C_{2^k})}$$where the finite cube $Z_k=2^{2^{k-1}-k}$ is endowed with the left zero multiplication.
\item[\textup{(6)}] $\lambda(X)$ contains a principal left ideal, which is topologically isomorphic to the compact topological semigroup
$$\prod\limits_{k\in\IN}(C_{2^k}\wr Z_k^{Z_k})^{q(X,C_{2^k})}.$$
 \end{enumerate}
\end{theorem}

\begin{proof} Let the subfamilies $\wht\K_n\subset\wht\K$ and the Cantor cubes $Z_n=2^{2^{n-1}-n}$, $n\in\IN\cup\{\infty\}$, be defined as in the proof of Theorem~\ref{t19.2}. Then $q(X,C_{2^n})=|\wht\K_n|$.
\smallskip

 $(1)\Ra(5,6)$ If $X$ admits no homomorphism onto $C_{2^\infty}$, then 
$q(X,C_{2^\infty})=0$ and $\wht\K_\infty=\emptyset$. In this case  $\wht\K=\bigcup_{n\in\IN}\wht\K_n$. For every $n\in\IN$ and $K\in\wht\K_n$ the characteristic group $\HH(K)$ is isomorphic to $C_{2^n}$ and the orbit space $[\Tau_K]$ is homeomorphic to the cube $Z_k$. By Theorem~\ref{t14.1}(3), the endomorphism monoid $\End(\Tau_K)$ is (topologically) isomorphic to $\HH(K)\wr[\Tau_K]^{[\Tau_K]}$ and the latter semigroup is topologically isomorphic to $C_{2^n}\wr Z_n^{\,Z_n}$. 

By Theorem~\ref{t18.7}(g), each minimal left ideal of $\lambda(X)$ is topologically isomorphic to $$\prod_{K\in\wht\K}\HH(K)\times[\Tau_K]=\prod_{n\in\IN}\prod_{K\in\wht\K_n}\HH(K)\times[\Tau_K]$$ and the latter semigroup is topologically isomorphic to the compact topological semigroup  $$\prod\limits_{k\in\IN}(C_{2^k}\times Z_k)^{q(X,C_{2^k})}.$$

By Corollary~\ref{c18.5}(3) and Theorem~\ref{t14.1}(3), the semigroup $\lambda(X)$ contains a principal left ideal that is topologically isomorphic to the compact topological semigroup  $\prod_{K\in\wtd\K}\HH(K)\wr[\Tau_K]^{[\Tau_K]}=\prod_{n\in\IN}\prod_{K\in\wht\K_n}\HH(K)\wr[\Tau_K]^{[\Tau_K]}$, which is topologically isomorphic to the compact topological semigroup
$$\prod\limits_{k\in\IN}(C_{2^k}\wr Z_k^{Z_k})^{q(X,C_{2^k})}.$$
\smallskip

The implication $(5)\Ra(2)\Ra(3)$ are trivial.
\smallskip

$(3)\Ra(1)$ If the group $X$ admits a homomorphism onto $C_{2^\infty}$, then the family $\wht\K$ contains a 2-cogroup $K_\infty$ whose characteristic group $\HH(K_\infty)$ is isomorphic to $C_{2^\infty}$.
It follows from Example~\ref{e9.3}(2) that for some twin set $A_{K_\infty}\in\Tau_{K_\infty}$  the twin-generated group $\HH(A_{K_\infty})$ is not a topological group. Now choose a sequence $(A_K)_{K\in\wht\K}\in\prod_{K\in\wht\K}\Tau_K$ of twin sets such that $A_K=A_{K_\infty}$ if $K=K_\infty$. Then the right-topological group $\prod_{K\in\wht\K}\HH(A_K)$ is not a topological group. By Corollary~\ref{c18.6}, this right-topological group is topologically isomorphic to some maximal subgroup of the minimal ideal $\IK(\lambda(X))$. So, $\IK(\lambda(X))$ contains a maximal subgroup, which is not a topological group. 
\smallskip

$(6)\Ra(4)$ If $\lambda(X)$ contains a left ideal, which is a topological semigroup, then $\lambda(X)$ contains a minimal left ideal, which is a topological semigroup. Any maximal subgroup of this minimal left ideal is a compact topological group. 
\smallskip

$(4)\Ra(1)$ If $K(\lambda(X))$ contains a compact maximal subgroup, then by  Theorem~\ref{t18.7}(c), each characteristic group $\HH(K)$, $K\in\wht\K$, is finite and hence $q(X,C_{2^\infty})=0$.
\end{proof}

Finally, we shall characterize abelian groups  whose superextension contains metrizable minimal left ideals. The characterization involves the notion of the free rank and 2-rank, see \cite[\S16]{Fu} or \cite[\S4.2]{Rob}.

Let us recall that a subset $A\not\ni e$ of an abelian group $G$ with neutral element $e$ is called {\em independent} if for any disjoint subsets $B,C\subset A$ the subgroups $\la B\ra$ and $\la C\ra$ generated by $B,C$ intersect by the trivial subgroup. The cardinality of a maximal independent subset $A\subset G$ that consists of element of infinite order (resp. of order that is a power of 2) is called the {\em free rank} (resp. the {\em 2-rank}) of $G$ and is denoted by $r_0(G)$ (resp. $r_2(G)$).    

\begin{theorem}\label{t19.4} For an abelian group $X$ the following conditions are equivalent:
\begin{enumerate}
\item[\textup{(1)}] each minimal left ideal of $\lambda(X)$ is metrizable;
\item[\textup{(2)}] the family $\wht\K$ of maximal 2-cogroups is at most countable;
\item[\textup{(3)}] the group $X$ admits no epimorphism onto the group $C_{2^\infty}\oplus C_{2^\infty}$ and $X$ has finite ranks $r_0(X)$ and $r_2(X)$.  
\end{enumerate}
\end{theorem}

\begin{proof} $(1)\Leftrightarrow(2)$ By Theorem~\ref{t19.2}(2), each minimal left ideal is homeomorphic to the cube $$(2^\w)^{q(X,C_{2^\infty})}\times \prod_{1\le k<\infty}(2^{2^{k-1}})^{q(X,C_{2^k})},$$ which  is metrizable if and only if $|\wht\K|=\sum_{1\le k\le\infty}q(X,C_{2^k})\le\aleph_0$.  
\smallskip

For the proof of the equivalence $(2)\Leftrightarrow(3)$ we need two lemmas. We define a group $G$ to be {\em $\wht\K$-countable} if the family of maximal 2-cogroups in $G$ is at most countable.

\begin{lemma}\label{l19.5} Each subgroup and each quotient group of a $\wht\K$-countable group is $\wht\K$-countable.
\end{lemma}

\begin{proof} Assume that a group $G$ is $\wht\K$-countable. 
To prove that any subgroup $H\subset G$ is $\wht\K$-countable, observe that by Proposition~\ref{p7.3}(2), each 2-cogroup $K\subset H$ can be enlarged to a maximal 2-cogroup $\bar K$ in $G$. The maximality of $K$ in $H$ guarantees that $K=\bar K\cap H$. This implies that the number of maximal 2-cogroups in $H$ does not exceed the number of maximal 2-cogroups in $G$.

To prove that any quotient group $G/H$ of $G$ by a normal subgroup $H\subset G$ is $\wht\K$-countable, observe that for each maximal 2-cogroup $K\subset G/H$ the preimage $q^{-1}(K)$ under the quotient homomorphism $q:G\to G/H$ is a maximal 2-cogroup in $G$. This implies that the number of maximal 2-cogroups in $G/H$ does not exceed the number of maximal 2-cogroups of $G$.
\end{proof}

For a group $G$ with neutral element $e$ and a set $A$ by $$\oplus^A G=\{(x_\alpha)_{\alpha\in A}\in G^A:|\{\alpha\in A:x_\alpha\ne e\}|<\aleph_0\}$$ we denote the direct sum of $|A|$  many copies of $G$.

\begin{lemma}\label{l19.6} The groups $\oplus^\w C_2$, $\oplus^\w \IZ$ and $C_{2^\infty}\times C_{2^\infty}$ are not  $\wht\K$-countable.
\end{lemma} 

\begin{proof} Observe that for any abelian group $X$ the number $q(X,C_2)$ is equal to the number of subgroups having index 2 and is equal to the number of non-trivial homomorphisms $h:X\to C_2$.

Each (non-empty) subset $A\subset \w$ determines a (non-trivial) homomorphism $$h_A:\oplus^\w C_2\to C_2,\;\;h_A:(x_i)_{i\in\w}\mapsto \prod_{i\in A}x_i.$$
For any distinct subsets $A,B\subset\w$ the homomorphisms $h_A$ and $h_B$ are distinct. Consequently, for the group $X=\oplus^\w C_2$, the family $\wht\K$ of maximal 2-cogroups has cardinality $|\wht\K|\ge \hom(X,C_2)=2^\w$ and hence this group is not $\wht\K$-countable.  

Since $\oplus^\w C_2$ is a quotient group of $\oplus^\w\IZ$, the latter group is not $\wht\K$-countable.

Finally, we show that the group $X=\C_{2^\infty}\times C_{2^\infty}$ is not $\wht\K$-countable. It is well-known (see \cite[\S43]{Fu}) that the quasicyclic group $C_{2^\infty}$ has uncountable automorphism group $\Aut(C_{2^\infty})$.
For any automorphism $h:C_{2^\infty}\to C_{2^\infty}$ its graph $\Gamma_h=\{(x,h(x)):x\in C_{2^\infty}\}$ is a subgroup of $X=C_{2^\infty}\times C_{2^\infty}$ such that the quotient group $X/\Gamma_h$ is isomorphic to $C_{2^\infty}$. Consequently, $q(X,C_{2^\infty})\ge\mathrm{Auth}(C_{2^\infty})>\aleph_0$ and hence the group $X=C_{2^\infty}\oplus C_{2^\infty}$ is not $\wht\K$-countable.
\end{proof}

Now we are able to prove the equivalence $(2)\Leftrightarrow (3)$.
The implication $(2)\Ra(3)$ follows from Lemmas~\ref{l19.5} and \ref{l19.6}.

To prove the implication $(3)\Ra(2)$ of Theorem~\ref{t19.4}, assume that an abelian group $X$ has finite free and 2-ranks and $X$ admits no homomorphism onto the group $C_{2^\infty}\oplus C_{2^\infty}$. By Proposition~\ref{p19.1}, the cardinality  of the set $\wht\K$ of maximal 2-cogroups in $X$ is equal to the cardinality of the family $\mathcal H$ of subgroups $H\subset X$ such that the quotient group $X/H$ is isomorphic to $C_{2^k}$ for some $1\le k\le\infty$.
So, it suffices to prove that $|\mathcal H|\le\aleph_0$.

Consider the subgroup $X_{\odd}\subset X$ consisting of the elements of odd order. Since $X$ has finite free and 2-ranks, so does the quotient group $X/X_{\odd}$. Then quotient group $Y=X/X_{\odd}$ is at most countable (because it contains no elements of odd order and has finite free and 2-ranks). Let $q:X\to Y$ be the quotient homomorphism. 

Let $\mathcal M$ be the family of maximal independent subsets consisting of elements of infinite order in the group $Y=X/X_{\odd}$. 
Since the free rank of $Y$ is finite, each (independent) set $M\in\M$ is finite and hence $\M$ is at most countable.   

For each $M\in\M$ consider the free abelian subgroup $\la M\ra\subset Y$ generated by $M$. Let $G_M=Y/\la M\ra$ be the quotient group and $q_M:Y\to G_M$ be the quotient homomorphism. The maximality of $M$ implies that $G_M$ is a torsion group. Since the free and 2-ranks of the group $Y$ are finite, the quotient group $G_M$ has finite 2-rank. The group $G_M$ is the direct sum $G_M=O_M\oplus D_M$ of the subgroup $O_M$ of elements of odd order and  the maximal 2-subgroup $D_M\subset G_M$. Let $p_M:G_M\to D_M=G_M/O_M$ be the quotient homomorphism. 

We claim that the group $D_M$ has at most countably many subgroups.
Since $D_M$ is a quotient group of $X$ and $X$ admits no homomorphism onto the group $(C_{2^\infty})^2$ the group $D_M$ also admits no homomorphism onto $C_{2^\infty}^2$. Two cases are possible.

1) The group $D_M$ contains no subgroup isomorphic to $C_{2^\infty}$. In this case the Pr\"ufer's Theorem~17.2 \cite{Fu} guarantees that $D_M$ is a direct sum of cyclic 2-groups. Since $D_M$ has finite 2-rank, it is finite, being a finite sum of cyclic 2-groups. Then $D_M$ has finitely many subgroups. 

2) The group $D_M$ contains a subgroup $D\subset M$ isomorphic to $C_{2^\infty}$. Being divisible, the subgroup $D$ is complemented in  $D_M$, which means that $D_M=D\oplus F$ for some subgroup $F\subset D_M$. Since $D_M$ admits no homomorphism onto $(C_{2^\infty})^2$, the subgroup $F$ contains no subgroup isomorphic to $C_{2^\infty}$ and hence is finite by the preceding case. Taking into account that the quasicyclic 2-group $D$ has countably many subgroups, we conclude that the group $D_M=D\oplus F$ also has countably many subgroups.

In both cases the family $\mathcal D_M$ of subgroups of $D_M$ is at most countable. Then the family $\mathcal H_{M}=\{(p_M\circ q_M\circ q)^{-1}(H):H\in \mathcal D_M\}$ also is at most countable. It remains to check that $\mathcal H\subset\bigcup_{M\in\mathcal M}\mathcal H_M$.

Fix any subgroup $H\in\mathcal H$. By the definition of $\mathcal H$, the quotient group $X/H$ is a 2-group, which implies $X_{\odd}\subset H$. 
Then $H=q^{-1}(H_Y)$ where $H_Y=q(H)$. Let $M$ be a maximal independent subset of $H_Y$ that consists of elements of infinite order. Since $Y/H_Y=X/H$ is a torsion group, the set $M$ is maximal in $Y$ and hence belongs to the family $\M$. It follows that $\la M\ra\subset H_Y$ and hence $H_Y=q_M^{-1}(H_M)$ where $H_M=q_M(H_Y)\subset G_M$. Since $G_M/H_M=Y/H_Y=X/H$ is a 2-group, the subgroup $H_M$ contains the subgroup $O_M$ of elements of odd order in $G_M$. Then $H_M=p_M^{-1}(G_M)$ where $G_M=p_M(H_M)\subset D_M$. Since $G_M\in\mathcal D_M$, we conclude that the group $H=(p_M\circ q_M\circ q)^{-1}(G_M)$ belongs to the family $\mathcal H_M\subset\mathcal H$.
\end{proof}

\section{Compact reflexions of groups}\label{s18}

In this section $X$ is an arbitrary group. 
Till this moment our strategy in describing the minimal left ideals of the semigroups $\lambda(X)$ consisted in finding a relatively small subfamily $\mathsf F\subset\mathsf P(X)$ such that the function representation $\Phi_{\mathsf F}:\lambda(X)\to\Enl(\mathsf F)$ is injective on all minimal left ideals of $\lambda(X)$. Now we shall simplify the group $X$ preserving the minimal left ideals of $\lambda(X)$ unchanged.

We shall describe three such simplifying procedures. One of them is the factorization of $X$ by the subgroup $$\Odd=\bigcap_{K\in\wht\K}KK.$$ Here we assume that $\Odd=X$ if the set $\wht\K$ is empty.

The following proposition explains the choice of the notation for the subgroup $\Odd$. We recall that a group $G$ is called {\em odd} if each element of $G$ has odd order.

\begin{proposition}\label{p20.1} $\Odd$ is the largest normal odd subgroup of $X$. If $X$ is Abelian, then $\Odd$ coincides with the set of all elements having odd order in $X$.
\end{proposition}

\begin{proof}
The normality of the subgroup $\Odd=\bigcap_{K\in\wht\K}KK$ follows from the fact that $xKx^{-1}\in\wht\K$ for every $K\in\wht\K$ and $x\in X$. Next, we show that the group $\Odd$ is odd. Assuming the converse, we could find an element $a\in\Odd$ such that the sets $a^{2\IZ}=\{a^{2n}:n\in\IZ\}$ and $a^{2\IZ+1}=\{a^{2n+1}:n\in\IZ\}$ are disjoint. Then the 2-cogroup $a^{2\IZ+1}$ of $X$  can be enlarged to a maximal 2-cogroup $K\in\wht\K$. It follows that $a\in K\subset X\setminus KK$ and thus $a\notin\Odd$, which is a contradiction.

It remains to prove that $\Odd$ contains any normal odd subgroup $H\subset X$. It suffices to check that for every maximal 2-cogroup $K\in\wht\K$ the subgroup  $H\subset X$ lies in the group $KK$. Let $K^\pm=K\cup KK$. Since the subgroup $H$ is normal in $X$, the sets $KKH=HKK$ and $K^\pm H=HK^\pm$ are subgroups. We claim that the sets $KH=HK$ and $KKH=HKK$ are disjoint. 
Assuming that $KH\cap KKH\ne\emptyset$, we can find a point $x\in K$ such that $x\in KKH$. Since $KK=xK$, there are points $z\in K$ and $h\in H$ such that $x=xzh$. Then $z=h^{-1}\in K\cap H$. 
Now consider the cyclic subgroup $z^{2\IZ}=\{z^{2n}:n\in\IZ\}$. Since $z\in K$, the subgroup $z^{2\IZ}$ does not intersect the set $z^{2\IZ+1}=\{z^{2n+1}:n\in\IZ\}$. On the other hand, since $H$ is odd, there is an integer number $n\in\IZ$ with $z^{2n+1}=z^0\in z^{2\IZ+1}\cap z^{2\IZ}$.
This contradiction shows that $KH$ and $KKH$ are disjoint. Consequently, the subgroup $KKH$ has index 2 in the group $K^\pm H$ and hence $KH=K^\pm H\setminus KKH$ is a 2-cogroup in $X$ containing $H$. The maximality of $K$ in $\K$ guarantees that $K=KH$ and hence $H\subset KK$.
\end{proof}

The quotient homomorphism $q_{\mathrm{odd}}:X\to X/\Odd$ generates a continuous semigroup homomorphism $\lambda(q_{\mathrm{odd}}):\lambda(X)\to\lambda(X/\Odd)$.

The following theorem was proved in \cite[3.3]{BG3}.

\begin{theorem}\label{t20.2} The homomorphism $\lambda(q_{\odd}):\lambda(X)\to\lambda(X/\Odd)$ is injective on each minimal left ideal of $\lambda(X)$.
\end{theorem}

Next, we define two compact topological groups called the first and second profinite reflexions of the group $X$. To define the first profinite reflexion, consider the family $\N$ of all normal subgroups of $X$ with finite index in $X$. For each subgroup $H\in\N$ consider the quotient homomorphism $q_H:X\to X/H$. The diagonal product of those homomorphisms determines the homomorphism $q:X\to\prod_{H\in\N}X/H$ of $X$ into the compact topological group $\prod_{H\in\N}X/H$. The closure of the image $q(X)$ in $\prod_{H\in\N}X/H$ is denoted by $\bar X$ and is called the {\em profinite reflexion} of $X$.

The second profinite reflexion $\bar X_2$ is defined in a similar way with help of the subfamily $$\N_2=\Big\{\bigcap_{x\in X}xKKx^{-1}:K\in\wht\K,\; |X/K|<\aleph_0\Big\}$$ of $\N$. The quotient homomorphisms $q_H:X\to X/H$, $H\in\N_2$, compose a homomorphism $q_2:X\to\prod_{H\in\N_2}X/H$. The closure of the image $q_2(X)$ in $\prod_{H\in\N_2}X/H$ is denoted by $\bar X_2$ and is called the {\em second profinite reflexion} of $X$. Since $\mathrm{Ker}(q_2)=\bigcap\N_2\supset\bigcap_{K\in\wht{\K}}KK\supset \Odd$, the homomorphism $q_2:X\to\bar X_2$ factorizes through the group $X/\Odd$ in the sense that there is a unique homomorphism $q_{\even}:X/\Odd\to \bar X_2$ such that $q_2=q_\even\circ q_\odd$.

Thus we get the following commutative diagram:
$$
\xymatrix{
X\ar[r]^-{q_{\odd}}\ar[dr]^{q_2}\ar[d]^{q} & {X/\Odd}\ar[d]^{q_\even}\\
\bar X \ar[r]_-{\pr} &\bar X_2
}
$$
Applying to this diagram the functor $\lambda$ of superextension we get the diagram
$$
\xymatrix{
\lambda(X)\ar[r]^-{\lambda(q_{\odd})}\ar[dr]^{\lambda(q_2)}\ar[d]^{\lambda(q)} & {\lambda(X/\Odd)}\ar[d]^{\lambda(q_\even)}\\
\lambda(\bar X) \ar[r]_-{\lambda(\pr)} &\lambda(\bar X_2)
}
$$
In this diagram $\lambda(\bar X)$ and $\lambda(\bar X_2)$ are the superextensions of the compact topological groups $\bar X$ and $\bar X_2$. We recall that the superextension $\lambda(K)$ of a compact Hausdorff space $K$ is the closed subspace of the second exponent $\exp(\exp(K))$ that consists of the maximal linked systems of closed subsets of $K$, see \cite[\S2.1.3]{TZ}.

\begin{theorem}\label{t18.3} If each maximal 2-cogroup $K$ of a twinic  group $X$ has finite index in $X$, then the homomorphism $\lambda(q_2):\lambda(X)\to\lambda(\bar X_2)$ is injective on each minimal left ideal of $\lambda(X)$.
\end{theorem}

\begin{proof} The injectivity of the homomorphism $\lambda(q_2)$ on a minimal left ideal  $\mathsf L$ of $\lambda(X)$ will follow as soon as for any distinct maximal linked systems $\A,\mathcal B\in\mathsf L$ we find a subgroup $H\in\mathcal N_2$ such that $\lambda q_H(\A)\ne\lambda q_H(\mathcal B)$. Fix any $[\wht\K]$-selector $\wtd\K\subset\wht\K$.

By Corollary~\ref{c18.5}, the homomorphism $\Phi_{\wtd{\Tau}}:\lambda(X)\to\prod_{K\in\wtd\K}\Enl(\Tau_K)$, $\Phi_{\wtd\Tau}:\LL\mapsto (\Phi_\LL|\Tau_K)_{K\in\wtd\K}$ is injective on the minimal left ideal $\mathsf L$. Consequently, $\Phi_\A|\Tau_K\ne \Phi_{\mathcal B}|\Tau_K$ for some $K\in\wtd\K$ and we can find a set $T\in\Tau_K$ such that $\Phi_\A(T)\ne \Phi_{\mathcal B}(T)$.

Since the 2-cogroup $K$ has finite index in $X$, the normal subgroup $H=\bigcap_{x\in X}xKKx^{-1}$ has finite index in $X$ and belongs to the family $\mathcal N_2$. Consider the finite quotient group $X/H$ and let $q_H:X\to X/H$ be the quotient homomorphism. Since $H\subset KK$, the set $T=KKT$ coincides with the  preimage $q_H^{-1}(T')$ of some twin set $T'\in X/H$. This fact can be used to show that $\lambda q_H(\A)\ne \lambda q_H(\mathcal B)$.
\end{proof}

\begin{remark} For each finite abelian group $X$ the group $X/\Odd$ is a 2-group. For non-commutative groups it is not always true: for the group $X=A_4$ of even permutations of the set $4=\{0,1,2,3\}$ the group $X/\Odd$ coincides with $X$, see Section~\ref{s21.5}. Also $X/\Odd$ coincides with $X$ for any simple group.
\end{remark}

\section{Some examples}

Now we consider the superextensions of some concrete groups. 

\subsection{The infinite cyclic group $\IZ$}  In order to compare the algebraic properties of the semigroups $\lambda(\IZ)$ and $\beta(\IZ)$ let us recall a deep result of E.~Zelenyuk \cite{Zel} (see also \cite[\S7.1]{HS}) who proved that each finite subgroup in the subsemigroup $\beta(\IZ)\subset\lambda(\IZ)$ is trivial. It turns out that the semigroup $\lambda(\IZ)$ has a totally different property.

\begin{theorem}\label{t21.1} 
\begin{enumerate}
\item[\textup{(1)}] The semigroup $\lambda(\IZ)$ contains a principal left ideal topologically isomorphic to $\prod\limits_{k=1}^\infty C_{2^k}\wr Z_k^{Z_k}$ where $Z_k=2^{2^{k-1}-k}$.
\item[\textup{(2)}] Each minimal left ideal of $\lambda(\IZ)$ is topologically isomorphic  to $2^\w\times \prod_{k=1}^\infty C_{2^k}$ where the Cantor cube $2^\w$ is endowed with the left-zero multiplication.
\item[\textup{(3)}] each maximal group of the minimal ideal $\IK(\lambda(\IZ))$ is topologically isomorphic to $\prod\limits_{k=1}^\infty C_{2^k}$.
\item[\textup{(4)}] The semigroup $\lambda(\IZ)$ contains a topologically isomorphic copy of each second countable profinite topological semigroup.
\end{enumerate}
\end{theorem}

\begin{proof} The group $\IZ$ is abelian and hence has trivial twinic ideal according to Theorem~\ref{t6.2}. It is easy to see that $q(\IZ,C_{2^k})=1$ for all $k\in\IN$, while $q(\IZ,C_{2^\infty})=0$.
\smallskip

1. By Theorem~\ref{t19.3}(6), the semigroup $\lambda(\IZ)$ contains a principal left ideal that is topologically isomorphic  to $\prod_{k=1}^\infty C_{2^k}\wr Z_k^{Z_k}$ where $Z_k=2^{2^{k-1}-k}$.
\smallskip

2. By Theorem~\ref{t19.3}(5), each minimal left ideal $\mathsf L$ of $\lambda(\IZ)$ is topologically isomorphic to $\prod_{k=1}^\infty C_{2^k}\times Z_k$ where each cube $Z_k=2^{2^{k-1}-k}$ is endowed with the left zero multiplication. It is easy to see that the left zero semigroup $\prod_{k=1}^\infty Z_k$ is topologically isomorphic to the Cantor cube $2^\w$ endowed with the left zero multiplication. Consequently, $\mathsf L$ is topologically isomorphic to $2^\w\times \prod_{k=1}^\infty C_{2^k}$.
\smallskip

3. The preceding item implies that each maximal group of the minimal ideal $\IK(\lambda(\IZ))$ is topologically isomorphic to $\prod\limits_{k=1}^\infty C_{2^k}$.

4. The fourth item follows from the first item and the following well-known fact, see \cite[I.1.3]{CP}.
\end{proof}

\begin{lemma}\label{emb} Each semigroup $S$ is algebraically isomorphic to a subsemigroup of the semigroup $A^A$ of all self-maps of a set $A$ of cardinality $|A|\ge |S^1|$ where $S^1$ is $S$ with attached unit.
\end{lemma}

\subsection{The (quasi)cyclic 2-groups $C_{2^n}$}

For a cyclic 2-group $X=C_{2^n}$ the number 
$$q(X,C_{2^{k}})=\begin{cases}1&\mbox{if $k\le n$}\\
0&\mbox{otherwise}.
\end{cases}
$$ Applying Theorem~\ref{t19.3} we get:

\begin{theorem}For every $n\in\IN$ 
\begin{enumerate}
\item[\textup{(1)}] The semigroup $\lambda(C_{2^n})$ contains a principal left ideal isomorphic to $\prod\limits_{k=1}^n C_{2^k}\wr Z_k^{Z_k}$ where $Z_k=2^{2^{k-1}-k}$.
\item[\textup{(2)}] Each minimal left ideal of $\lambda(C_{2^n})$ is isomorphic  to $ \prod_{k=1}^n C_{2^k}\times Z_k$ where each cube $Z_k=2^{2^{k-1}-k}$ is endowed with left-zero multiplication.
\item[\textup{(3)}] Each maximal group of the minimal ideal $\IK(\lambda(C_{2^n}))$ is isomorphic to $\prod\limits_{k=1}^n C_{2^k}$.
\item[\textup{(4)}] The semigroup $\lambda(C_{2^n})$ contains an isomorphic copy of each semigroup $S$ of cardinality $|S|<2^{2^{n-1}-n}$.
\end{enumerate}
\end{theorem}

The superextension $\lambda(C_{2^\infty})$ has even more interesting properties.

\begin{theorem}\label{t21.4}
\begin{enumerate}
\item[\textup{(1)}] Minimal left ideals of the semigroup $\lambda(C_{2^\infty})$ are not   topological semigroups.
\item[\textup{(2)}] each minimal left ideal of $\lambda(C_{2^\infty})$ is homeomorphic to the Cantor cube $2^\w$ and is algebraically isomorphic to $\mathfrak c\times (C_{2^\infty})^\w$ where the cardinal $\mathfrak c=2^{\aleph_0}$ is  endowed with left zero multiplication;
\item[\textup{(3)}] the semigroup $\lambda(C_{2^\infty})$ contains a principal left ideal, which is algebraically isomorphic to $(C_{2^\infty}\wr \mathfrak c^\mathfrak c)^\w$;
\item[\textup{(4)}] $\lambda(C_{2^\infty})$ contains an isomorphic copy of each semigroup of cardinality $\le \mathfrak c$;
\item[\textup{(5)}] each maximal subgroup of the minimal ideal $\IK(\lambda(C_{2^\infty}))$ of $\lambda(C_{2^\infty})$ is algebraically isomorphic to $(C_{2^\infty})^\w$;
\item[\textup{(6)}] each maximal subgroup of the minimal ideal $\IK(\lambda(C_{2^\infty}))$ is topologically isomorphic to the countable product $\prod_{n=1}^\infty (C_{2^\infty},\tau_n)$ of quasicyclic 2-groups endowed with twin-generated topologies;
\item[\textup{(7)}] for any twin-generated topologies $\tau_n$, $n\in\IN$, on $C_{2^\infty}$ the right-topological group $\prod_{n=1}^\infty(C_{2^\infty},\tau_n)$ is topologically isomorphic to a maximal subgroup of $\IK(\lambda(C_{2^\infty}))$.
\end{enumerate}
\end{theorem}

\begin{proof} Since each proper subgroup of $C_{2^\infty}$ is finite, the family $\wht{\K}$ of maximal 2-cogroups is countable and hence can be enumerated as $\wht\K=\{K_n:n\in\w\}$. Each maximal 2-cogroup $K\in\wht{\K}$ has infinite index and its characteristic group $\HH(K)$ is isomorphic to $C_{2^\infty}$. 
\smallskip

1. The equivalence $(1)\Leftrightarrow (2)$ of Theorem~\ref{t19.3} implies that no minimal left ideal of $\lambda(C_{2^\infty})$ is a topological semigroup.
\smallskip

2,3,5. The statements (2), (3) and (5) follow from Theorem~\ref{t19.2}.
\smallskip

4. The forth item follows from the third one because each semigroup $S$ of cardinality $|S|\le\mathfrak c$ embeds into the semigroup $\mathfrak c^{\mathfrak c}$ according to Lemma~\ref{emb}.
\smallskip

6. By Theorem~\ref{t18.7}(b), each maximal subgroup $G$ in the minimal ideal $\IK(\lambda(C_{2^\infty}))$ is topologically isomorphic to the product $\prod_{K\in\wht\K}\HH(A_K)$ of the structure groups of suitable twin subsets $A_K\in\Tau_K=\Tau_{[K]}$, $K\in\wht\K$. For each maximal 2-cogroup $K\in\wht{\K}$ the structure group $\HH(A_K)$ is just $C_{2^\infty}$ endowed with a  twin-generated topology. 
\smallskip

7. Now assume conversely that $\tau_n$, $n\in\IN$, are twin generated topologies on the quasicyclic group $C_{2^\infty}$. For every $n\in\IN$  find a twin subset $A_n\in\Tau_{K_n}$ whose structure group $\HH(A_n)$ is topologically isomorphic to $(C_{2^\infty},\tau_n)$. By 
Theorem~\ref{t18.8}, the product $\prod_{n=1}^\infty \HH(A_n)$ is topologically isomorphic to some maximal subgroup of $\IK(\lambda(C_{2^\infty}))$.
\end{proof}

\begin{remark} Theorems~\ref{t21.4}(7) and \ref{t9.5} imply that among maximal subgroups of the minimal ideal of $\lambda(C_{2^\infty})$ there are:
\begin{itemize}
\item Raikov complete topological groups;
\item incomplete totally bounded topological groups;
\item paratopological groups, which are not topological groups;
\item semitopological groups, which are not paratopological groups.
\end{itemize}
\end{remark} 

\subsection{The groups of generalized quaternions $Q_{2^n}$} 

We start with the quaternion group  $Q_8=\{\pm 1,\pm\mathbf i,\pm\mathbf j,\pm\mathbf k\}$. It contains 3 cyclic subgroups of order 4
 corresponding to 4-element maximal 2-cogroups: $K_1=Q_8\setminus \langle\mathbf i\rangle$, $K_2=Q_8\setminus \langle\mathbf j\rangle$, $K_3=Q_8\setminus \langle\mathbf k\rangle$. The characteristic groups of those 2-cogroups are isomorphic to $C_2$. The trivial subgroup of $Q_8$ corresponds to the maximal 2-cogroup $K_0=\{-1\}$ whose characteristic group coincides with $Q_8$. By Proposition~\ref{p15.2}, we get
$$|[\Tau_{K_0}]|=\frac{|\Tau_{K_0}|}{|\HH(K_0)|}=\frac{2^{|X/K_0^\pm|}}{|Q_8|}=2$$ and $|[\Tau_{K_i}]|=1$ for  $i\in\{1,2,3\}$.
By Theorem~\ref{t18.7}(2), each minimal left ideal of the semigroup $\lambda(Q_8)$ is isomorphic to
$$(Q_8\times 2)\times (C_2\times 1)^3 =2\times Q_8\times C_2^{\;3}.$$
\smallskip

Next, given any finite number $n\ge 3$ we consider the generalized quaternion group $Q_{2^{n+1}}$. Maximal 2-cogroups in $Q_{2^{n+1}}$ are of the following form:
$$K_0=\{-1\},\; K_1=Q_{2^{n+1}}\setminus C_{2^n} \mbox{ \ and \ }K_{k,x}=\{1,x\}\cdot (C_{2^{k}}\setminus C_{2^{k-1}})$$ for $2\le k\le n$ and $x\in Q_{2^{n+1}}\setminus C_{2^n}$.
It follows that $H(K_0)=Q_{2^{n+1}}$, $H(K_1)=C_2$ and $H(K_{k,x})=C_2$.
Also $$|[T_{K_0}]|=\frac{|T_{K_0}|}{|H(K_0)|}=
\frac{2^{|Q_{2^{n+1}}/K_0^\pm|}}{|Q_{2^{n+1}}|}=\frac{2^{2^n}}{2^{n+1}}=
2^{2^n-n-1},$$
$$|[T_{K_1}]|=\frac{|T_{K_1}|}{|H(K_1)|}=\frac{2^{|Q_{2^{n+1}}/K_1^\pm|}}{|C_2|}=\frac{2^1}{2}=1,$$and 
$$|[T_{K_{k,x}}]|=\frac{|T_{K_{k,x}}|}{|H(K_{k,x})|}=
\frac{2^{|Q_{2^{n+1}}/K_{k,x}^\pm|}}{|C_2|}=\frac{2^{2^{n+1}/2^{k+1}}}{2}=
2^{2^{n-k}-1}.$$

It is easy to check that two 2-cogroups $K_{k,x}$ and $K_{k,y}$ are conjugated if and only if $xy^{-1}\in C_{2^{n-1}}$. Taking any elements $x,y\in Q_{2^{n+1}}\setminus C_{2^n}$ with $xy^{-1}\notin C_{2^n}$, we conclude that the family $$\wtd\K=\{K_0,K_{k,x},K_{k,y}:2\le k\le n\}$$ is a $[\wht\K]$-selector. Applying Theorems~\ref{t18.7}, \ref{t14.1}(3) and Corollary~\ref{c18.5}(3), we get:

\begin{theorem}\label{t21.6} Let $n\ge 2$ be a finite number. Then 
\begin{enumerate}
\item[\textup{(1)}] each minimal left ideal of the semigroup $\lambda(Q_{2^{n+1}})$ is isomorphic to $$Q_{2^{n+1}}\times 2^{2^n-n-1}\times C_2\times \prod_{k=2}^{n}(C_2\times 2^{2^{n-k}-1})^2,$$
where the cubes $2^{2^n-n-1}$ and $2^{2^{n-k}-1}$ are endowed with the left zero multiplication;
\item[\textup{(2)}] each maximal subgroup of the minimal ideal $\IK(\lambda(Q_{2^{n+1}}))$ is isomorphic to\newline $Q_{2^{n+1}}\times C_2^{\;2n-1}$. 
\end{enumerate}
\end{theorem}

The infinite group $Q_{2^\infty}$ of generalized quaternions has a similar structure. This group contains the following maximal 2-cogroups:
$$K_0=\{-1\},\; K_1=Q_{2^\infty}\setminus C_{2^\infty},\;\mbox{ and }\;K_{k,x}=\{1,x\}\cdot C_{2^{k}}\setminus C_{2^{k-1}}$$where $k\ge 2$ and $x\in Q_{2^\infty}\setminus C_{2^\infty}$. For these 2-cogroups we get
$$H(K_0)=Q_{2^\infty},\; H(K_1)=C_2,\;\mbox{ and }\; H(K_{k,x})=C_2$$and
$$|[T_{K_0}]|=\mathfrak c,\; |[T_{K_1}]|=1,\;\mbox{ and }\;|[T_{K_{k,x}}]|=\mathfrak c.$$
Any two 2-cogroups $K_{k,x}$, $K_{k,y}$ are conjugated. Then for any $b\in Q_{2^\infty}\setminus C_{2^\infty}$ the family $\wtd\K=\{K_0,K_{k,b}:k\in\IN\}$ is a $[\wht\K]$-selector. By analogy with Theorem~\ref{t21.4} we can prove:

\begin{theorem}\label{t21.7} For the group $Q_{2^\infty}$ 
\begin{enumerate}
\item[\textup{(1)}] minimal left ideals of the semigroup $\lambda(C_{2^\infty})$ are not   topological semigroups;
\item[\textup{(2)}] each minimal left ideal of the semigroup $\lambda(Q_{2^\infty})$ is homeomorphic to the Cantor cube and is algebraically isomorphic to $$Q_{2^\infty}\times C_2^{\;\w}\times\mathfrak c,$$ where the cardinal $\mathfrak c$ is endowed with the left zero multiplication;
\item[\textup{(3)}] the semigroup $\lambda(Q_{2^\infty})$ contains a principal ideal isomorphic to $$(Q_{2^\infty}\wr \mathfrak c^{\mathfrak c})\times C_2\times (C_2\wr\mathfrak c^{\mathfrak c})^\w;$$
\item[\textup{(4)}]  $\lambda(Q_{2^\infty})$ contains an isomorphic copy of each semigroup of cardinality $\le \mathfrak c$;
\item[\textup{(5)}] each maximal subgroup of the minimal ideal $\IK(\lambda(Q_{2^\infty}))$ is topologically isomorphic to $(Q_{2^\infty},\tau)\times C_2^{\;\w}$ where $\tau$ is a twin-generated topology on $Q_{2^\infty}$;
\item[\textup{(6)}] for any twin-generated topology $\tau$ on $Q_{2^\infty}$ the right-topological group $(Q_{2^\infty},\tau)\times C_2^{\;\w}$ is topologically isomorphic to a maximal subgroup of $\IK(\lambda(Q_{2^\infty}))$.
\end{enumerate}
\end{theorem}

\begin{remark} Theorems~\ref{t21.7}(6) and \ref{t9.5} imply that among maximal subgroups of the minimal ideal of $\lambda(Q_{2^\infty})$ there are:
\begin{itemize}
\item Raikov complete topological groups,
\item incomplete totally bounded topological groups,
\item right-topological groups, which are not left-topological groups,
\item semitopological groups, which are not paratopological groups.
\end{itemize}
\end{remark} 

\subsection{The dihedral 2-groups $D_{2^n}$}

By the {\em dihedral group} $D_{2n}$ of even order $2n$ we understand any group with presentation
$$\la a,b\mid a^n=b^2=1, \; bab^{-1}=a^{-1}\ra.$$
It can be realized as the group of symmetries of a regular $n$-gon.
So, $D_{2n}$ is a subgroup of the orthogonal group $O(2)$.
The group $D_{2n}$ contains the cyclic subgroup $C_n=\la a\ra$ as a subgroup of index 2. The subgroup of all elements of odd order is normal in $D_{2n}$ and hence coincide with the maximal normal odd subgroup $\Odd$. By Theorem~\ref{t20.2}, the superextension $\lambda(D_{2n})$ is isomorphic to the superextension $\lambda(D_{2n}/\Odd)$ of the quotient group $D_{2n}/\Odd$. The latter group is isomorphic to the dihedral group $D_{2^k}$ where $2^k$ maximal power of 2 that divides $2n$. Therefore it suffices to consider the superextensions of the dihedral 2-groups $D_{2^k}$. 

By the {\em infinite dihedral 2-group} we understand the union $$D_{2^\infty}=\bigcup_{k\in\IN}D_{2^k}\subset O(2).$$
It contains the quasicyclic 2-group $C_{2^\infty}$ as a normal subgroup of index 2.

Now we analyze the structure of the superextension $\lambda(D_{2^{n}})$ for finite $n\ge 1$. Maximal 2-cogroup in $D_{2^{n}}$ are of the following form:
$$K_0=D_{2^{n}}\setminus C_{2^{n-1}} \mbox{ \ and \ } K_{k,x}=\{1,x\}\cdot(C_{2^k}\setminus C_{2^{k-1}})$$where 
$1\le k<n$ and $x\in K_0=D_{2^{n}}\setminus C_{2^{n-1}}$.
The characteristic groups of these maximal 2-cogroups are isomorphic to the 2-element cyclic group $C_2$. Also 
$$|[\Tau_{K_0}]|=1\mbox{ and }|[\Tau_{K_{k,x}}]|=\frac{|2^{D_{2^{n}}/K_{k,x}^\pm}|}{|\HH(K)|}=2^{2^{n-k}-1}$$ for all $1\le k<n$ and $x\in K_0$. 

Let $b\in D_{2^{n}}\setminus C_{2^{n-1}}$ be any element and $a$ be the generator of the cyclic subgroup $C_{2^{n-1}}\subset D_{2^{n}}$.
 One can check that two 2-cogroups $K_{k,x}$ and $K_{k,y}$ are conjugated if and only if $x^{-1}y\in C_{2^{n-2}}$. 
Therefore the family
$$\wtd\K=\{K_0,K_{k,b},K_{k,ab}:1\le k<n\}$$ is a $[\wht\K]$-selector.

Applying Theorems~\ref{t18.7}, \ref{t14.1} and Corollary~\ref{c18.5}(3), we get 

\begin{theorem}\label{t21.9} For every $n\in\IN$ 
\begin{enumerate}
\item[\textup{(1)}] The semigroup $\lambda(D_{2^{n}})$ contains a principal left ideal isomorphic to $C_2\times \prod\limits_{k=1}^{n-1} (C_2\wr Z_k^{\;Z_k})^2$ where $Z_k=2^{2^{n-k}-1}$.
\item[\textup{(2)}] Each minimal left ideal of $\lambda(D_{2^{n}})$ is isomorphic  to $C_2\times \prod_{k=1}^n (C_{2}\times Z_k)^2$ where cubes $Z_p$ are endowed with left-zero multiplication.
\item[\textup{(3)}] Each maximal group of the minimal ideal $\IK(\lambda(D_{2^{n+1}}))$ is isomorphic to $C_2^{\;2n-1}$.
\item[\textup{(4)}] The semigroup $\lambda(D_{2^{n+1}})$ contains an isomorphic copy of each semigroup $S$ of cardinality $|S|<2^{2^{n-1}-1}$.
\end{enumerate}
\end{theorem}

The superextension of the infinite dihedral 2-group $D_{2^\infty}$ has quite interesting properties. All maximal subgroups of the minimal ideal $\mathsf K(\lambda(D_{2^\infty}))$ are compact topological groups. On the other hand, in the semigroup $\lambda(D_{2^\infty})$ there are minimal left ideals, which are  (or are not) topological semigroups.

\begin{theorem} For the group $D_{2^\infty}$ 
\begin{enumerate}
\item[\textup{(1)}] each minimal left ideal of the semigroup $\lambda(D_{2^\infty})$ is homeomorphic to the Cantor cube $2^\w$ and is algebraically isomorphic to the compact topological semigroup $C_2^{\;\w}\times 2^\w$ where the Cantor cube $2^\w$ is endowed with the left zero multiplication;
\item[\textup{(2)}] each maximal subgroup of the minimal ideal $\IK(\lambda(D_{2^\infty}))$ is topologically isomorphic to the compact topological group $C_2^{\;\w}$;
\item[\textup{(3)}]  $\lambda(D_{2^\infty})$ contains a minimal left ideal, which is topologically isomorphic to the compact topological semigroup 
$C_2^{\;\w}\times 2^\w;$
\item[\textup{(4)}] $\lambda(D_{2^\infty})$ contains a minimal left ideal, which is not a semitopological semigroup.
\item[\textup{(5)}] the semigroup $\lambda(D_{2^\infty})$ contains a principal ideal isomorphic to $C_2\times (C_2\wr \mathfrak c^{\mathfrak c})^\w;$
\item[\textup{(6)}]  $\lambda(D_{2^\infty})$ contains an isomorphic copy of each semigroup of cardinality $\le \mathfrak c$.
\end{enumerate}
\end{theorem}

\begin{proof} First note that by Theorem~\ref{t6.2} the torsion group $X=D_{2^\infty}$ is twinic and has trivial twinic ideal.

 Maximal 2-cogroup in $D_{2^{\infty}}$ are of the following form:
$$K_0=D_{2^{\infty}}\setminus C_{2^\infty} \mbox{ \ and \ } K_{k,x}=\{1,x\}\cdot(C_{2^k}\setminus C_{2^{k-1}})$$where 
$k\in\IN$ and $x\in K_0$.
The characteristic groups of these maximal 2-cogroups are isomorphic to the 2-element cyclic group $C_2$. Consequently, for any twin set $A\in\wht\Tau$ its characteristic group $\HH(A)$ is topologically isomorphic to $C_2$. Observe that 
$$|[\Tau_{K_0}]|=1\mbox{ and }|[\Tau_{K_{k,x}}]|=2^\w$$ for all $k\in\IN$ and $x\in K_0$. Since the characteristic group $H(K_{k,x})=C_2$ is finite, the orbit space $[\Tau_{K_{p,x}}]$ is a compact Hausdorff space, homeomorphic to the Cantor cube $2^\w$.
 One can check that any two 2-cogroups $K_{k,x}$ and $K_{k,y}$ are conjugated. 
Therefore for any $b\in D_{2^\infty}\setminus C_{2^\infty}$ the family
$\wtd\K=\{K_0,K_{k,b}:k\in\IN\}$ is a $[\wht\K]$-selector.
\smallskip

1. By Theorem~\ref{t18.7}(e) and Proposition~\ref{p12.1}, each minimal left ideal of $\lambda(X)$ is homeomorphic to the product $\prod_{K\in\wtd\K}\Tau_K$, which is homeomorphic to $2^{X/K_0^\pm}\times\prod_{k\in\IN}2^{X/K_{k,b}^\pm}$. The latter space is homeomorphic to the Cantor cube $2^\w$.

By Theorem~\ref{t18.7}(e), each minimal left ideal of $\lambda(X)$ is algebraically isomorphic to $\prod_{K\in\wtd\K}\HH(K)\times [\Tau_K]$ and the latter semigroup is isomorphic to $C_2^\w\times 2^\w$ where the Cantor cube $2^\w$ is endowed with the left zero multiplication.
\smallskip

2. Taking into account that each characteristic group $\HH(A)$, $A\in\wht\Tau$, is topologically isomorphic to $C_2$ and applying Theorem~\ref{t18.7}(b), we conclude that each maximal subgroup in the minimal ideal $\IK(\lambda(X))$ is topologically isomorphic to the compact topological group $C_2^\w$.
\smallskip

3. Since each characteristic group $\HH(K)$, $K\in\wht\K$, is finite (being isomorphic to $C_2$), Proposition~\ref{p18.10} implies that some minimal left ideal of $\lambda(X)$ is a topological semigroup, which is topologically isomorphic to the compact topological semigroup $$\prod_{K\in\wtd\K}\HH(K)\times [\Tau_K]=\HH(K_0)\times[\Tau_{K_0}]\times\prod_{k\in\IN}(\HH(K_{k,b})\times[\Tau_{K_{k,b}}])$$ by Theorem~\ref{t18.7}(g). The latter topological semigroup is topologically isomorphic to $C_2^{\;\w}\times 2^\w$.
\smallskip

4. Since the maximal 2-cogroups $K_{k,b}$, $k\in\IN$, have infinite index in $D_{2^\infty}$, Proposition~\ref{p18.9} implies that the semigroup $\lambda(D_{2^\infty})$ contains a minimal left ideal, which is not a semitopological semigroup.
\smallskip

5. By Corollary~\ref{c18.5}(3) and Theorem~\ref{t14.1}(3), the semigroup $\lambda(D_{2^\infty})$ contains a principal left ideal that is algebraically isomorphic to the semigroup $\prod_{K\in\wtd\K}(\HH(K)\wr[\Tau_K]^{[\Tau_K]})$, which is isomorphic to $C_2\times (C_2\wr \mathfrak c^{\mathfrak c})^\w$.
\smallskip

6. By the preceding item, $\lambda(D_{2^\infty})$ contains a subsemigroup, isomorphic to the semigroup $\mathfrak c^{\mathfrak c}$ of all self-mapping of the continuum $\mathfrak c$. By Lemma~\ref{emb}, the latter semigroup contains an isomorphic copy of each semigroup of cardinality $\le\mathfrak c$.
\end{proof}

\subsection{Superextensions of finite groups of order $<16$}\label{s21.5}

Theorem~\ref{t19.3} and Proposition~\ref{p19.1} give us an algorithmic way of calculating the minimal left ideals of the superextensions of finitely-generated abelian groups. For non-abelian groups the situation is a bit more complicated. In this section we shall describe the minimal left ideals of finite groups $X$ of order $|X|<16$.

In fact, Theorem~\ref{t20.2} helps us to reduce the problem to studying superextensions of groups $X/\Odd$. The group $X/\Odd$ is trivial if the order of $X$ is odd. So, it suffices to check non-abelian groups of even order. If $X$ is a 2-group, then the subgroup $\Odd$ of $X$ is trivial and hence $X/\Odd=X$. Also the subgroup $\Odd$ is trivial for simple groups.

The next table describes the structure of minimal left ideals of the superextensions of groups $X=X/\Odd$ of order $|X|\le 15$. In this table $\mathcal E$ stands for a minimal idempotent of $\lambda(X)$, which generates the principal left ideal $\lambda(X)\circ\mathcal E$ and lies in the maximal subgroup $H(\mathcal E)=\mathcal E\circ\lambda(X)\circ\mathcal E$.
Below the cubes $2^n$ are considered as semigroups of left zeros.
\vskip20pt

\begin{center}
\begin{tabular}{|c|c|c|c|c|c|c|c|c|c|c|c|}
\hline
\phantom{$|^|_|$}$X$\phantom{$|^|_|$} & $|E(\lambda(X)\circ\mathcal E)|$ & $\mathcal E\circ\lambda(X)\circ\mathcal E$ & $\lambda(X)\circ\mathcal E$\cr
\hline
$C_2$&\phantom{$|^|_|$}$1$\phantom{$|^|_|$}& $C_2$          &$C_2$\cr
\hline
$C_4$&\phantom{\large$o^|$}$1$\phantom{\large $o^|$} & $C_2\times C_4$&$C_2\times C_4$\cr
$C_2^{\;2}$&\phantom{\Large$o_|$}$1$\phantom{\Large$o_|$}& $C_2^{\;3}$        &$C_2^{\;3}$\cr
\hline
$C_2^{\;3}$&\phantom{\Large$o^|$}$1$\phantom{\Large $o^|$}& $C_2^{\;7}$   & $C_2^{\;7}$ \cr
$C_2\oplus C_4$&$1$ & $C_2^{\;2}\times C_4^{\;2}$ & $C_2^{\;3}\times C_4^{\;2}$\cr
$C_8$&\phantom{\Large$o_|$}$2$\phantom{\Large$o_|$}& $C_2\times C_4\times C_8$ & $2\times C_2\times C_4\times C_8$\cr
\hline
$D_8$&\phantom{\Large$o^|$}$2$\phantom{\Large $o^|$}& $C_2^{\;5}$& $2^2\times C_2^{\;5}$\cr
$Q_8$&\phantom{\Large$o_|$}$2$\phantom{\Large$o_|$}& $C_2^{\;3}\times Q_8$& $2\times C_2^{\;3}\times Q_8$\cr
\hline
$A_4$&\phantom{\large$o^|_|$}$2^6$\phantom{\large $o^|_|$}& $C^{\;3}_2$ & $2^6\times C^{\;3}_2$\cr
\hline
\end{tabular}
\end{center}

\medskip

For abelian groups the entries of this table are calculated with help of Theorem~\ref{t19.3} and Proposition~\ref{p19.1}. Let us illustrate this in the example of the group $C_2\oplus C_4$.

 By Proposition~\ref{p19.1}, for the group $X=C_2\oplus C_4$ we get
\begin{itemize}
\item $q(X,C_2)=|\hom(X,C_2)|-|\hom(X,C_1)|=2\cdot 2-1=3$;
\item $q(X,C_4)=\frac12(|\hom(X,C_4)|-|\hom(X,C_2)|)=\frac12(2\cdot 4-2\cdot 2)=2$;
\item $q(X,C_{2^k})=0$ for $k>2$.
\end{itemize}
Then each minimal left ideal of $\lambda(C_2\oplus C_4)$ is isomorphic to $$(C_2\times 2^{2^{1-1}-1})^{q(X,C_2)}\times(C_4\times 2^{2^{2-1}-2})^{q(X,C_4)}=(C_2\times 2^0)^3\times (C_4\times 2^0)^2=C_2^{\;3}\times C_4^{\;2}.$$

Next, we consider the non-abelian groups. In fact, the groups $Q_8$ and $D_8$ have been treated in Theorems~\ref{t21.6} and \ref{t21.9}. So, it remains to consider the alternating group $A_4$.

This group has order 12, contains a normal subgroup isomorphic to $C_2\times C_2$ and contains no subgroup of order 6. This implies that all 2-cogroups of $A_4$ lie in $C_2\times C_2$ and consequently, $A_4$ contains 3 maximal 2-cogroups. Each maximal 2-cogroup $K\subset A_4$ contains two elements and has characteristic group $\HH(K)$ isomorphic to $C_2$. Since $|X/K^\pm|=3$,
Proposition~\ref{p15.2} guarantees that $|[\Tau_K]|=2^{|X/K^\pm|}/|\HH(K)|=2^{3-1}=2^2$.
Applying Theorem~\ref{t18.7}, we see that each minimal left ideal of the semigroup $\lambda(A_4)$ is isomorphic to $(C_2\times 2^2)^3=2^6\times C_2^{\;3}$.

\section{Some Open Problems}

\begin{problem} Describe the structure of (minimal left ideals) of superextensions of the simple groups $A_n$ for $n\ge 5$.
\end{problem}

\begin{problem} Describe the structure of (minimal left ideals) of superextensions of the finite groups of order 16.
\end{problem}

Since the free group $F_2$ with two generators is not twinic, the results obtained in this paper cannot be applied to this group.

\begin{problem} What can be said about the structure of the superextension $\lambda(F_2)$ of the free group $F_2$? 
\end{problem}

\begin{problem} Investigate the permanence properties of the class of twinic groups. Is this class closed under taking subgroups? products?
\end{problem}


\end{document}